\documentclass[11pt,reqno]{amsart}
\usepackage{graphicx}
\usepackage{amsfonts,amsmath,amssymb}
\newcommand{\cx}{{\mathbb{C}}}

\newcommand {\Q}{\mathcal Q}
\newtheorem{thm}{Theorem}[section]
\newtheorem*{thm*}{Theorem}
\newtheorem{propos}[thm]{Proposition}
\newtheorem{corol}[thm]{Corollary}

\theoremstyle{definition}
\newtheorem{dfn}[thm]{Definition}
\newtheorem{ex}[thm]{Example}
\newtheorem{rema}[thm]{Remark}
\newtheorem{lem}[thm]{Lemma}
\newcommand{\im}{\ensuremath{\mbox{\rm Im}\,}}
\newcommand{\re}{\ensuremath{\mbox{\rm Re}\,}}

\newcommand{\theor}[1]{\smallskip  \noindent \bf Theorem #1.\it\,\,}

\newcommand{\CC}[1]{\mathbb{C}^{#1}}
\newcommand{\CP}[1]{\mathbb{CP}^{#1}}
\newcommand{\RR}[1]{\mathbb{R}^{#1}}

\newcommand{\dz}{\frac{\partial}{\partial z}}

\newcommand{\lr}{\longrightarrow}

\numberwithin{equation}{section}

\title[Analytic Differential Equations]{Analytic
Differential Equations and  Spherical Real Hypersurfaces}
\author {I. Kossovskiy}
\address{Department of Mathematics, University of Vienna, Oskar-Morgenstern-Platz-1, Vienna, Austria}
\email{ilya.kossovskiy@univie.ac.at}
\author {R. Shafikov}
\email{}
\address{Department of Mathematics, The University of Western Ontario, London, Ontario N6A 5B7 Canada}
\email{shafikov@uwo.ca}

\begin{document}

\date{\today}

\begin{abstract}
We establish an injective correspondence $M\lr\mathcal E(M)$
between real-analytic nonminimal hypersurfaces $M\subset\CC{2}$,
spherical at a generic point, and a class of second order complex
ODEs with a meromorphic singularity. We apply this result to the
proof of the bound $\mbox{dim}\,\mathfrak{hol}(M,p)\leq 5$ for the
 infinitesimal automorphism algebra of an \it
arbitrary \rm germ $(M,p)\not\sim(S^3,p')$ of a real-analytic Levi
nonflat hypersurface $M\subset\CC{2}$ (the Dimension Conjecture).
This bound gives the proof of the dimension gap
$\mbox{dim}\,\mathfrak{hol}(M,p)=\{8,5,4,3,2,1,0\}$ for the
dimension of the automorphism algebra of a real-analytic Levi
nonflat hypersurface. As another application we obtain a new
regularity condition for CR-mappings of nonminimal hypersurfaces,
that we call \it Fuchsian type, \rm and prove its optimality for
extension of CR-mappings to nonminimal points. \\ We also obtain
an existence theorem for solutions of a class of singular complex
ODEs.
\end{abstract}

\maketitle

\tableofcontents

\section{Introduction}

The goal of this paper is to give solution to a number of
previously open problems in CR-geometry, including an old question
of H.~Poincar\'e, by introducing a new technique when a
CR-manifold under consideration is replaced by an appropriate
holomorphic dynamical system. By doing so we reduce the original
problem to a classical setting in local holomorphic dynamics.
Using this approach the authors \cite{divergence} proved recently
that for any positive CR-dimension and CR-codimension the
holomorphic moduli space in CR-geometry is bigger than the formal
one. We describe below the CR-geometry problems addressed in the
paper, and briefly explain our dynamical approach. To outline the
parallels between CR-geometry and complex dynamical systems we
summarize the connection between the geometric objects and the
corresponding dynamical analogues in a table at the end of this
introduction.

Let $M,M'\ni 0$ be two real-analytic hypersurfaces in the complex
space $\CC{2}$. A local biholomorphic mapping $\mathcal
F:\,(\CC{2},0)\lr (\CC{2},0)$ is called \it a holomorphic
equivalence between $(M,0)$ and $(M',0)$, \rm if $F(M)\subset M'$.
In 1907 H.~Poincar\'e formulated his {\it probl\`eme local}
\cite{poincare}: given two germs of real-analytic hypersurfaces
$M,M'\subset\CC{2}$, find all local holomorphic equivalences
between them. The discovery of Poincar\'e was that the problem is
highly nontrivial due to the fact that germs of Levi nondegenerate
hypersurfaces in $\CC{2}$ possess biholomorphic invariants. That
makes two germs in general position holomorphically inequivalent.
Another discovery of Poincar\'e was that the local automorphism
group $\mbox{Aut}\,(M,0)$ of a Levi nondegenerate hypersurface is
finite dimensional and is always a subgroup in the stability group
$\mbox{Aut}\,(S^3,o)$ of a point $o$ lying in the 3-dimensional
sphere $S^3\subset\CC{2}$. For the pseudogroup of local
self-mappings of a Levi nondegenerate hypersurface (or,
alternatively, for the well-defined associated infinitesimal
automorphism algebra $\mathfrak{hol}\,(M,0)$) Poincar\'e gives the
bound $\mbox{dim}\,\mathfrak{hol}\,(M,0)\leq
\mbox{dim}\,\mathfrak{hol}\,(S^3,o)=8.$ The considerations of
Poincar\'e were based on the existence of a "model" Levi
nondegenerate hypersurface, namely, the quadric $\mathcal Q=\{\im
w=|z|^2\}\cong S^3$.  Ideas of Poincar\'e were developed and
generalized in the work of E.~Cartan \cite{cartan}, N.~Tanaka
\cite{tanaka}, S.~Chern and J.~Moser \cite{chern}, who obtained a
complete solution for the local holomorphic equivalence problem
for real-analytic Levi nondegenerate hypersurfaces in
$\CC{n},\,n\geq 2$.

Today, after more than a century, \it probl\`eme local \rm is
still very far from being solved completely. We outline below some
recent results and explain the difficulties in completing the
problem.

For hypersurfaces in $\CC{2}$ with Levi degeneracies satisfying
the finite type condition (see, e.g., \cite{ber}), the equivalence
problem was studied by V.~Beloshapka, V.~Ezhov and M.~Kolar and
completed in the work  \cite{kolar} of Kolar. The problem in the
finite type case was treated in the spirit of Poincar\'e by using
{\it models}, i.e., hypersurfaces defined by $\im w = P_k(z,\bar
z)$, where $P_k(z,\bar z)$ is a nonzero homogeneous polynomial of
degree $k\geq 3$ without harmonic terms. These models allow one to
obtain a formal normal form for finite type real-analytic
hypersurfaces $M\subset\CC{2}$. Even though such a normal form can
be divergent (see~\cite{kolardiverg}), convergence results for
formal CR-equivalences (see, e.g., \cite{ber1}) show that such a
normal form is a biholomorphic invariant and thus a solution for
the holomorphic equivalence problem. By relaxing the finite type
condition one comes to the consideration of a significantly more
difficult to analyze class of the so-called {\it nonminimal}
hypersurfaces (the term coined in \cite{tumanov}), that is real
hypersurfaces $M$ containing a complex hypersurface $X$. The main
obstruction for solving {\it probl\`eme local} in the nonminimal
case is perhaps hidden in the fact that polynomial hypersurfaces
arising from the defining equation of a nonminimal hypersurface
can {\it no longer} be considered as models in the sense of
Poincar\'e-Chern-Moser. For example, in the class of nonminimal
hypersurfaces $\left\{\im w=(\re
w)\psi(|z|^2),\,\psi(0)=0,\,\psi'(0)\neq 0\right\}$, all of which
contain the complex hypersurface $X=\{w=0\}$, any polynomial model
has the isotropy group of dimension $2$, while the hypersurface
$\im w=(\re w)\tan\left(\frac{1}{2}\arcsin |z|^2\right)$ has the
isotropy group of dimension 5 (see \cite{belnew}, \cite{kl}). A
recent result of the authors~\cite{divergence} showing that formal
equivalences between nonminimal hypersurfaces can be actually
divergent, proves, in particular, that a formal normal form can no
longer be a solution for the equivalence problem for nonminimal
hypersurfaces, which further illustrates the difficulties for this
class of hypersurfaces. In fact, even the class of nonminimal
hypersurfaces \it spherical at a generic point \rm appears to be
highly nontrivial (we refer here to the work \cite{kowalski, elz,
belold, kl, nonminimal, divergence} of V.~Beloshapka, P.~Ebenfelt,
M.~Kolar, Kowalski, B.~Lamel,
 D.~Zaitsev and the authors), as it is not even known whether the
moduli space for this class of hypersurfaces is finite
dimensional.

One of the goals of the present paper is to give a complete
solution for the automorphism version of Poincar\'e's {\it
probl\`eme local}. We first give a solution in the nonminimal
case, more precisely, we prove the following
\smallskip

\theor{1} Let $M\subset\CC{2}$ be a real-analytic nonminimal at
the origin Levi nonflat hypersurface. Then the dimension of its
infinitesimal automorphism algebra satisfies the bound

\begin{equation} \label{bound5} \mbox{dim}\,\mathfrak{hol}\,(M,0)\leq
5.\end{equation}\rm

\smallskip

The previous example of the hypersurface $\im w=(\re
w)\tan\left(\frac{1}{2}\arcsin |z|^2\right)$ shows that the bound
in Theorem 1 is in fact sharp. As a corollary, we obtain the
following ``dimension gap" phenomenon, solving the \it probl\`eme
local \rm (in the automorphism interpretation) completely.

\medskip

\noindent\bf Corollary 1 \rm (see Theorem 3.11). \it Let
$M\subset\CC{2}$ be a real-analytic hypersurface, $0\in M$, and
let $M$ be Levi nonflat. Then $\mathfrak{hol}\,(M,0)$ is
isomorphic to a subalgebra in
$\mathfrak{hol}\,(S^3,o)\simeq\mathfrak{su}(2,1)$. Moreover, the
bound $\mbox{dim}\,\mathfrak{hol}\,(M,0)\leq 5$ holds unless
$(M,0)$ is biholomorphic to $(S^3,o)$ for $o\in S^3$. In
particular, the dimension gap $\mbox{dim}\,\mathfrak{hol}(M,0)\in
\{8,5,4,3,2,1,0\}$ holds for all possible dimensions  of the
infinitesimal automorphism algebra of  real-analytic Levi nonflat
hypersurfaces $M\subset\CC{2}$. \rm

\medskip

Corollary~1 should be compared with various dimension gap
phenomena in differential geometry, in particular, for isometries
of Riemannian manifolds (see, e.g., S.~Kobayashi
\cite{kobayashi}), or for automorphism groups of Kobayashi
hyperbolic manifolds (see, e.g., A.~Isaev \cite{isaevn2,isaevbook}
and references therein). An interesting parallel here is given by
the fact that the maximal dimension $8$ for  the automorphism
group of a two-dimensional hyperbolic manifold is realized only
for the special case of the $2$-ball $\mathbb{B}^2\subset\CC{2}$,
while for the automorphism algebra of a real-analytic Levi nonflat
hypersurface
 $M\subset\CC{2}$ the maximal dimension $8$ is realized only for
the $3$-sphere $S^3=\partial\mathbb{B}^2$.

\smallskip

We can further formulate

\medskip

\noindent\bf Corollary 2. \it Let $M\subset\CC{2}$ be a
real-analytic Levi nonflat hypersurface, $M\ni 0$. Suppose that
the stability group $\mbox{Aut}\,(M,0)$ is a Lie group in the
natural topology. Then $\mbox{dim\,Aut}\,(M,0)\leq 5$. \rm

\smallskip

The example of the $3$-sphere $S^3\subset\CC{2}$ (or the previous
example of the nonminimal hypersurface $\im w=(\re
w)\tan\left(\frac{1}{2}\arcsin |z|^2\right)$) show that the bound
in Corollary 2 is sharp. For the most recent results on Lie group
structures for automorphism groups of real-analytic CR-manifolds
we refer to the work \cite{jl1,jl2} of R.~Juhlin and B.~Lamel.

\medskip

The assertions of Theorem 1 and Corollaries 1 and 2 are known as
different versions of the Dimension Conjecture, see the survey
\cite{obzor} and also \cite{elz}, \cite{belnew}, and \cite{kl} for
partial results in this direction. For various corollaries of
Theorem~1 concerning infinitesimal automorphism algebras of
real-analytic germs, as well as intermediate results, we refer the
reader to Section~3. In particular, Theorem~3.7 gives a curious
description of the infinitesimal automorphism algebra of a
nonminimal spherical hypersurface as a subalgebra in the
centralizer of a special element $\sigma\in\mbox{Aut}(\CP{2})$.

Another question addressed in the paper is the analytic
continuation problem for a germ of a biholomorphism between
real-analytic hypersurfaces $M,M'\subset\CC{n}$. The question goes
back to another remarkable result of  Poincar\'e in
\cite{poincare}, which states that a local holomorphic equivalence
$\mathcal F:\,(S^3,o)\lr (S^3,o')$ extends to a global
linear-fractional automorphism of the  $2$-ball $\mathbb
B^2\subset\CC{2}$. The result of Poincar\'e was generalized by
S.\,Pinchuk \cite{pinchuk}, who proved that if a real-analytic
hypersurface $M\subset\CC{n}$ is strictly pseudoconvex, then a
local holomorphic equivalence $\mathcal F:\,(M,p)\lr (S^{2n-1},o)$
extends locally biholomorphically along any path $\gamma\subset
M,\,\gamma\ni p$ (for $M=S^{2n-1}\subset\CC{n}$ the result was
also obtained by H.\,Alexander \cite{alexander}). The importance
of the analytic continuation problem for the boundary regularity
of holomorphic mappings was demonstrated by the celebrated
Pinchuk's reflection principle  for strictly pseudoconvex domains with
real-analytic boundaries (see \cite{pinchuk2}). This result initiated further
generalizations of Poincar\'e's Theorem. Many significant results in this direction
were obtained by the school of A.\,Vitushkin, using the
convergence of the Chern-Moser normal form (see \cite{vitushkin}
and references therein) and also in \cite{Sh,HiSh,shaf} by using
extension along Segre varieties. Note that in all cited papers the
hypersurface $M$ in the preimage was assumed to be minimal.
However, as shown in the earlier paper \cite{nonminimal} of the
authors, when $M$ is nonminimal and $M'$ is the simplest possible
(namely, $M'$ is a hyperquadric in $\CP{n}$) the possibility to
extend the germ of a biholomorphic mapping $\mathcal F:\,(M,p)\lr
(M',p')$ analytically along a path $\gamma\subset M,\,\gamma\ni p$
\it fails \rm to hold in general, if the path $\gamma$ intersects
the complex hypersurface $X$, contained in $M$.  The difficulty
here is that neither the Chern-Moser-type technique (in view of
the absence of a convergent normal form), nor the technique of
extension along Segre varieties (in view of the fact that $Q_p\cap
X\neq\emptyset$ implies $p\in X$) can be used to extend a mapping
to nonminimal points in $M$. However, it was shown in
\cite{nonminimal} that \it if $M\setminus X$ is Levi nondegenerate
and $X\ni 0$, then one can choose an open set $U\subset\CC{n}$,
$U\ni 0$ in such a way that the desired analytic extension holds
(as a mapping into $\CP{n}$) for any choice of a point
$p\in(U\setminus X)\cap M$ and a path $\gamma\subset U\setminus
X,\,\gamma\ni p$ \rm (note that $\gamma$ here need not to lie in
$M$). Since $U\setminus X$ is not simply-connected, such an
extension can branch about the complex locus $X$, which forms the
first type of obstructions for extending a mapping into a quadric
to the complex locus $X$ (see various examples provided in
\cite{nonminimal}). We say that the resulting (multiple-valued)
analytic mapping $\mathcal F:\,U\setminus X\lr\CP{n}$ \it is
associated with $M$ \rm (this object is defined uniquely up to a
composition with an element $\sigma\in\mbox{Aut}(\CP{n})$). \rm
Surprisingly, the authors found an example (see Example 6.7 in
Section 6) where a local biholomorphic mapping $\mathcal
F_0:\,(M,p)\lr (S^3,o)$ of a nonminimal hypersurface
$M\subset\CC{2}$ at a Levi nondegenerate point $p$ does \it not
\rm extend holomorphically to the complex locus $X$, even though
the associated mapping $\mathcal F$ does not branch about $X$. The
latter example made the extension/no extension dichotomy
particularly intriguing, and also showed the existence of another
type of obstruction for analytic extension to nonminimal points.

Our second main result is the discovery of the \it non-Fuchsian
type \rm condition for a hypersurface $M\subset\CC{2}$ (see
Definition 1.1 below) as the second type of obstruction and the
proof of the fact that no further obstructions exist beside the
two mentioned previously.  We formulate  the results in detail
below.

Let $M\subset\CC{2}$ be a real-analytic nonminimal at the origin
Levi nonflat hypersurface, and $U\ni 0$ be a polydisc. We say that
\it $M$ is given in $U$ in prenormal coordinates \rm if the
defining equation of $M\cap U$ is of the form
\begin{equation}\label{prenormal}
v=u^m\left(\pm|z|^2+\sum\limits_{k,l\geq 2}\Phi_{kl} (u)z^k\bar
z^l\right),
\end{equation}
where $z=x+iy,\,w=u+iv$ denote the coordinates in $\CC{2}$ and
$\Phi_{kl}(u)$ are analytic near the origin functions. The complex
locus for $M$ in this case is given by $X=\{w=0\}$. Depending on
the sign in \eqref{prenormal} we call $M$ \it positive \rm or \it
negative \rm respectively. Examples in Section~2 below show that
prenormal coordinates for a nonminimal hypersurface fail to exist
in general. However, Theorem 3.1 (see Section 3) shows that \it
prenormal coordinates always exist for every real-analytic
nonminimal at the origin and spherical outside the complex locus
hypersurface. \rm

For a nonminimal hypersurface, given in prenormal coordinates, we
first prove the following geometric criterion for the analytic
continuation of a mapping into a sphere.

\medskip

\theor{2} Let $M\subset\CC{2}$ be a real-analytic hypersurface,
containing a complex hypersurface $X\ni 0$, which is Levi
nondegenerate and spherical in $M\setminus X$. Suppose that $M$ is
given in some polydisc $U=\{|z|<\delta\}\times\{|w|<\epsilon\}$ in
prenormal coordinates. Then a local biholomorphic mapping
$\mathcal F:\,(M,p)\lr(S^3,p')$, $p\in (M\setminus X)\cap U$,
$p'\in S^3$, extends to $X$ holomorphically if and only if for
each Segre variety  $Q_s$, $s\in U$, which is not a "horizontal"\,
line $\{w=const\}$, there exists a holomorphic graph
$$\tilde
Q_s=\left\{(z,w)\in \CP{1}\times\{|w|<\epsilon\}:\,
z=h_s(w)\right\},\,\,h_s\in\mathcal
O\left(\{|w|<\epsilon\}\right)$$ (called the extension of $Q_s$),
such that $Q_s= \tilde Q_s\cap U$. \rm

\medskip

We next formulate the crucial

\begin{dfn}
Suppose that $M$ satisfies the conditions of Theorem 2. We say
that \it $M$ is of Fuchsian type at the origin, \rm if its
defining function \eqref{prenormal} satisfies
\begin{equation} \label{fuchstype}
\mbox{ord}_0 \Phi_{22}\geq m-1,\,\mbox{ord}_0 \Phi_{33}\geq
2m-2,\,\mbox{ord}_0 \Phi_{23}\geq \frac{3}{2}(m-1),
\end{equation}
where $\mbox{ord}_0$ denotes the order of vanishing of a function
at the origin. If the conditions \eqref{fuchstype} fail to hold, we
say that \it $M$ is of non-Fuchsian type. \rm
\end{dfn}

We emphasize that the Fuchsian type condition holds automatically
if $m=1$, and fails to hold in general for $m>1$. It is shown in
Section~6 that the property of being Fuchsian is independent of
the choice of prenormal coordinate system.

\medskip

\theor{3} Let $M\subset\CC{2}$ be a real-analytic hypersurface,
containing a complex hypersurface $X\ni 0$, which is Levi
nondegenerate and spherical in $M\setminus X$, $U$ a sufficiently
small neighbourhood of the origin, $p\in (M\setminus X)\cap U$, and
let $\gamma$ be a generator of $\pi_1(U\setminus X)$, $p\in
\gamma$. Suppose that $M$ is of Fuchsian type. Then a local
biholomorphic mapping $\mathcal F_0:\,(M,p)\lr(S^3,p'),\,p'\in
S^3$, extends to $X$ holomorphically if and only if its analytic
extension $\mathcal F:\,U\setminus X\lr\CP{2}$ does not branch
along~$\gamma$.
\medskip
 \rm

 It is shown in Section 6 that the Fuchsian type condition in Theorem 3 is in a sense optimal.
 We also note that Theorem 3 demonstrates the difference between the
 geometry of $1$-nonminimal and $m$-nonminimal hypersurfaces with
 $m>1$ respectively. This difference became apparent already in the
 work of P.~Ebenfelt \cite{ebenfelt}, where the analyticity of
 CR-mappings from $1$-nonminimal hypersurfaces was proved. It also appeared in the paper \cite{divergence}
 of the authors, where it was
 shown that formal CR-mappings between $m$-nonminimal hypersurfaces
 with $m>1$ can be divergent (while for $m=1$ formal CR-mappings
 are always convergent, as shown by R.~Juhlin and B.~Lamel in
 \cite{jl2}). At the end of Section 3 we formulate a conjecture on universality
 of the Fuchsian type condition as a regularity condition for mappings from
 nonminimal hypersurfaces.

 As an intermediate step in the proof of Theorem 3, we prove the
 following existence theorem for singular ODEs: \it a singular
 holomorphic ODE
 \begin{equation}\label{briot}
 z''=\frac{1}{w}P(z,w)z'+\frac{1}{w^2}Q(z,w), \ \ P, Q  \in\mathcal O(\{|z|<\delta\}\times\{|w|<\epsilon\}),
 \end{equation} such that
$Q(z_0,0)=0$ for some $|z_0|<\delta$, has a holomorphic in a
neighbourhood of the origin solution $z=h(w)$ with $h(0)=z_0$,
provided that no local solution of it admits a multiple-valued
extension to an annulus $\{\epsilon'<|w|<\epsilon''\}$ with
$0<\epsilon'<\epsilon''<\epsilon$ \rm (see Theorem 3.5 below).

\smallskip

The \it nonlinear \rm complex ODE \eqref{briot} after the
substitution $u:=z'w$ can be rewritten as the first order system
\begin{equation}
\begin{cases}
wz'=u ,\\
wu'=(1+P(z,w))u+Q(z,w),
\end{cases}
\end{equation}
for which the right-hand side vanishes for $z=z_0,u=0,w=0$. This
is a particular case of the \it Briot-Bouquet type ODEs. \rm These
are first order singular holomorphic ODE systems of the form
$wz'=A(z,w)$, $z\in\CC{n},w \in\CC{},A(0)=0$ with
$A:\,\CC{n}\times\CC{}\lr\CC{n}$ holomorphic near the origin.
Briot-Bouquet type ODEs can be described as nonlinear
generalizations of Fuchsian ODEs. They are known to have a
holomorphic solution under the additional assumption that the
linearization matrix $\frac{\partial A}{\partial z}(0)$ has no
eigenvalues $k\in\mathbb{Z},\, k>0$ (nonresonant case, see
\cite{laine}). In the resonant case a holomorphic solution fails
to exist in general (a simple example is given by the scalar
equation $wz'=z+w$). It is easy to check that the "no-monodromy"
assumption in Theorem 3.5 does not imply the "no-resonance"
condition, and vice versa, so the assertion of Theorem 3.5 is
nontrivial. To the best of our knowledge, the result is new (see,
e.g., the recent surveys \cite{laine},\cite{gromak} and references
therein).

\smallskip

The main tool of the paper is a development, in the \it  Levi
degenerate \rm case, of the fundamental connection between
CR-geometry and the geometry of completely integrable systems of
complex PDEs, first observed by E.~Cartan and B.~Segre
\cite{cartan, segre}. In particular, the geometry of real-analytic
Levi nondegenerate hypersurfaces in $\CC{2}$ is closely related to
that of (nonsingular!) second order complex ODEs, as discussed in
Section~2. For modern treatment of the connection in the
nondegenerate case we refer to earlier work
\cite{sukhov1,sukhov2,gm,nurowski,merker} of H.~Gaussier,
J.~Merker, P.~Nurowski, G.~Sparling and A.~Sukhov. The mediator
between a real hypersurface $M$ and the associated ODE $\mathcal
E(M)$ is the Segre family of $M$, which in this case is (an open
subset of) the family of integral curves of $\mathcal E(M)$.  In
this paper we treat the significantly different case of a
nonminimal hypersurface $M$. By establishing an injective
correspondence $M\lr\mathcal E(M)$ between the class of all
real-analytic nonminimal hypersurfaces $M\subset\CC{2}$, spherical
at a generic point, and a class of second order complex ODEs with
an isolated meromorphic singularity at the origin, we were able to
reformulate the problems, addressed in the paper, in the language
of analytic theory of differential equations. This gives us a
powerful tool for the study of mappings and automorphisms of
nonminimal hypersurfaces. The central object of the paper appears
to be the nonlinear complex ODE
$$z''=\frac{1}{w^m}(Az+B)z'+\frac{1}{w^{2m}}(Cz^3+Dz^2+Ez+F),\eqno
(*)$$ where the holomorphic coefficients
$A(w),B(w),C(w),D(w),E(w),F(w)$ satisfy certain relations which
guarantee that $(*)$ can be locally mapped into the simplest ODE
$z''=0$ at its regular points. The latter property can be
interpreted as vanishing of the Tresse differential invariants of
$\mathcal E(M)$, or as vanishing of the Cartan curvature of
$\mathcal E(M)$ (see the work \cite{tresse}, \cite{cartan2} of
A.~Tresse and E.~Cartan respectively, and also V.~Arnold
\cite{arnoldgeom} for a modern treatment). With the additional
assumption that the hypersurface $M$ admits the rotational
infinitesimal symmetry $iz\frac{\partial}{\partial z}$, the
connection $M\longleftrightarrow\mathcal E(M)$ was studied in the
earlier paper \cite{divergence} of the authors. Remarkably, it
turns out that \it any such $M$ can be associated a linear ODE
$z''=\frac{B(w)}{w^m}z'+\frac{E(w)}{w^{2m}}z$, \rm and
furthermore, Fuchsian type hypersurfaces are associated with
Fuchsian ODEs. Note, however, that as examples in \cite{kl} show,
one cannot restrict considerations to hypersurfaces with the
rotational symmetry only.

\smallskip

The following table illustrates the relation between various
geometric and ODE properties arising from the correspondence
between $M$ and $\mathcal E(M)$.

\medskip\begin{center}
\begin{tabular}{|p{7cm}| p{7cm}|}
 \hline \vspace{0.01cm}
 Nonminimal hypersurface M, spherical in the complement of the
 complex locus $X$ &
Second order complex ODE with a meromorphic singularity and
vanishing
 Cartan-Tresse invariants at regular points\\
\hline
 Nonminimal locus $X=\{w=0\}$ & Singular point $w=0$\\
 \hline
 Segre varieties & Graphs of solutions\\
 \hline
 Monodromy of the associated mapping $\mathcal F$ &
Monodromy of solutions\\
\hline \vspace{0.01cm}
 Holomorphic extension of $\mathcal F$ to $X$ & Meromorphic
extension of solutions to $w=0$\\
\hline
Fuchsian type hypersurface & Fuchsian (Briot - Bouquet) type ODE\\
\hline Automorphisms of a nonminimal hypersurface & Point
symmetries of a singular ODE\\ \hline
\end{tabular}
\end{center}\medskip

The paper is organized as follows. In Section 2 we provide some
background material on CR-geometry and the analytic theory of
differential equation. In Section 3 we give detailed formulations
of the main results of the paper, and also formulate the necessary
intermediate results. Sections 4--9 contain proofs, their
organization is described at the end of Section 3.

\bigskip

\begin{center} \bf Acknowledgments \end{center}

\medskip

We would like to thank Victor Kleptsyn, Timur Sadykov, and Ilpo
Laine for useful discussions, and also Andrey Minchenko for
communicating to us the proof of Proposition 9.2.

\section{Preliminaries}\label{s:SV}

\subsection{Segre varieties.}
Let $M$ be a smooth connected real-analytic hypersurface in
$\cx^{n}$, $0\in M$, and $U$ a neighbourhood of the origin where
$M\cap U$ admits a real-analytic defining function
$\phi(Z,\overline Z)$. For every point $\zeta\in U$ we can
associate with $M$ its so-called Segre variety in $U$ defined as
$$
Q_\zeta= \{Z\in U : \phi(Z,\overline \zeta)=0\}.
$$
Segre varieties depend holomorphically on the variable $\overline
\zeta$. One can find  a suitable pair of neighbourhoods $U_2={\
U_2^z}\times U_2^w\subset \cx^{n-1}\times \CC{}$ and $U_1 \Subset
U_2$ such that
$$
  Q_\zeta=\left \{(z,w)\in U^z_2 \times U^w_2: w = h(z,\overline \zeta)\right\}, \ \ \zeta\in U_1,
$$
is a closed complex analytic graph. Here $h$ is a holomorphic
function. Following \cite{DiPi2} we call $U_1, U_2$ a {\it
standard pair of neighbourhoods} of the origin. The
antiholomorphic $n$-parameter family of complex hypersurfaces
$\{Q_\zeta\}_{\zeta\in U_1}$ is called \it the Segre family of $M$
at the origin. \rm From the definition and the reality condition
on the defining function the following basic properties of Segre
varieties follow:
\begin{equation}\label{e.svp}
  Z\in Q_\zeta \ \Leftrightarrow \ \zeta\in Q_Z,
\end{equation}
\begin{equation*}
  Z\in Q_Z \ \Leftrightarrow \ Z\in M,
\end{equation*}
\begin{equation*}
  \zeta\in M \Leftrightarrow \{Z\in U_1: Q_\zeta=Q_Z\}\subset M.
\end{equation*}
The fundamental role of Segre varieties for holomorphic mappings
is illuminated by their invariance property: if $f: U \to U'$ is a
holomorphic map sending a smooth real-analytic submanifold
$M\subset U$ into another such submanifold $M'\subset U'$, and $U$
is as above, then
$$
f(Z)=Z' \ \ \Longrightarrow \ \ f(Q_Z)\subset Q'_{Z'}.
$$ For
the proofs of these and other properties of Segre varieties see,
e.g., \cite{webster}, \cite{DiFo}, \cite{DiPi2}, \cite{shaf}, or
\cite{ber}.

In the particularly important case when $M$ is a \it real
hyperquadric, \rm i.e., when $$M=\left\{
[\zeta_0,\dots,\zeta_n]\in \cx\mathbb P^{n} : H(\zeta,\bar \zeta)
=0  \right\},$$ where $H(\zeta,\bar \zeta)$ is a nondegenerate
Hermitian form in $\CC{n+1}$ with $k+1$ positive and $l+1$
negative eigenvalues, $k+l=n-1,\,0\leq l\leq k\leq n-1$, the Segre
variety of a point $\zeta \in\CP{n}$ is the projective hyperplane
$ Q_\zeta = \{\xi\in \cx\mathbb P^n: H(\xi,\bar\zeta)=0\}$. The
Segre family $\{Q_\zeta,\,\zeta\in\CP{n}\}$  coincides in this
case with the space $(\CP{n})^*$ of all projective hyperplanes in
$\CP{n}$.

The space of Segre varieties $\{Q_Z : Z\in U_1\}$ can be
identified with a subset of $\cx^K$ for some $K>0$ in such a way
that the so-called \it Segre map \rm $\lambda : Z \to Q_Z$ is
holomorphic (see \cite{DiFo}). For a Levi nondegenerate at a point
$p$ hypersurface $M$ its Segre map is one-to-one in a
neighbourhood of $p$. When $M$ contains a complex hypersurface
$X$, for any point $p\in X$ we have $Q_p = X$ and $Q_p\cap
X\neq\emptyset\Leftrightarrow p\in X$, so that the Segre map
$\lambda$ sends the entire $X$ to a unique point in $\CC{K}$, and
$\lambda$ is not even finite-to-one near each $p\in X$ (i.e., $M$
is \it not essentially finite \rm at points $p\in X$). For a
hyperquadric $\Q\subset\CP{n}$ the Segre map $\lambda'$ is a
global natural one-to-one correspondence between $\CP{n}$ and the
space $(\CP{n})^*$.

\subsection{Defining equations for nonminimal hypersurfaces.}
Let $M\subset\CC{n}$ be again a smooth real-analytic nonminimal
hypersurface, containing a complex hypersurface $X\ni 0$ and Levi
nondegenerate in $M\setminus X$. We choose local coordinates
$(z,w) \in\CC{n-1}\times \CC{}$ near the origin in such a way that
the complex hypersurface, contained in $M$, is given by
$X=\{w=0\}$, and $M$ is given locally by the equation
$$\im w=(\re w)^m\Phi(z,\bar z,\re w),$$ where $\Phi(z,\bar z,\re w)$ is a real-analytic function
in a neighbourhood of the origin such that  $\Phi(z,\bar
z,0)\not\equiv 0$, $\Phi(z,0,\re w)=\Phi(0,\bar z,\re w)\equiv 0$,
and $m$ is a positive integer (see \cite{ber},\cite{ebenfelt} for
the existence of such coordinates). In this case $M$ is called \it
$m$-nonminimal, \rm and the integer $m$, known to be a
biholomorphic invariant of $M$, is called \it the nonminimality
order of $M$ at $0$. \rm We may further consider the so-called \it
complex defining
 equation \rm (see, e.g., \cite{ber})\,
$w=\Theta(z,\bar z,\bar w)$ \, of $M$ near the origin, which one
obtains by substituting $u=\frac{1}{2}(w+\bar
w),\,v=\frac{1}{2i}(w-\bar w)$ into the real defining equation and
applying the holomorphic implicit function theorem. Here
$\Theta=1+O(2)$ is a real-analytic function near the origin in
$\CC{2n-1}$ satisfying certain reality condition. For our purposes
it is convenient to use the so-called \it exponential \rm defining
equation for a nonminimal real hypersurface
\cite{nonminimal},\,\cite{kl}:
\begin{gather*}
w=\bar w\,e^{i\varphi(z,\bar z,\,\bar w)},
\end{gather*}
where the complex-valued real-analytic function $\varphi$ in a
polydisc $U\ni 0$ satisfies the conditions $\varphi(z,0,\re
w)=\varphi(0,\bar z,\re w)\equiv 0$ (here $m$ is the nonminimality
order of $M$ at $0$), $\varphi(z,\bar z,\,\bar w)=(\bar
w)^{m-1}\psi(z,\bar z,\,\bar w)$ for an appropriate real-analytic
function $\psi(z,\bar z,\,\bar w)\not\equiv 0$,  and also the
reality condition
\begin{equation}\label{reality}
\varphi(z,\bar z,w\,e^{-i\bar\varphi(\bar z,z,w)})\equiv
\bar\varphi(\bar z,z,w),
\end{equation}
reflecting the fact that $M$ is a real hypersurface.

\medskip

\noindent {\bf Convention.} In what follows in this paper, for a
series of the form
$$f(z_1,..,z_s)=\sum\nolimits_{k_j\in\mathbb{Z}}c_{k_1,...,k_s}z_1^{k_1}\cdot...\cdot
z_s^{k_s}$$  we denote by $\bar f(z_1,..,z_s)$ the series
$\sum_{k_j\in\mathbb{Z}}\overline c_{k_1,...,k_s}
\,z_1^{k_1}\cdot...\cdot z_s^{k_s}$.

\medskip

We then introduce the following property, strengthening the
$m$-nonminimality.

\begin{dfn}\label{d.2.1}
A real-analytic hypersurface $M\subset\CC{n}$, containing a
complex hypersurface $X=\{w=0\}$ and Levi nondegenerate in
$M\setminus X$, is called {\it Levi regular at the origin}, if in
appropriate local coordinates near the origin the function
$\varphi$ in the exponential defining equation of $M$ has the
form:
\begin{equation}\label{e.2.2}
\varphi(z,\bar z,\bar w)=(\bar w)^{m-1} (h(z,\bar z,\bar
w)+\tilde\varphi(z,\bar z,\bar w)),
\end{equation}
where $h(z,\bar z,\bar w)$ is a nondegenerate hermitian form in
$z,\bar z$ for each $\bar w$, $m$ is the nonminimality order of
$M$ at $0$, $\tilde\varphi(z,0,\bar w)\equiv\tilde\varphi(0,\bar
z,\bar w)\equiv 0$ and also $\tilde\varphi(z,\bar z,\bar
w)=O(||z||^3)$ (here $||z||$ is the standard Euclidian norm in
$\CC{n-1}$). Alternatively, the Levi regularity means that the
power series $\frac{1}{(\bar w)^{m-1}}\varphi(z,\bar z,\bar
w)|_{\bar w=0}$ has a nondegenerate hermitian part.
\end{dfn}

The following example shows that a generic nonminimal at the
origin and Levi nondegenerate outside the complex locus real
hypersurface does not have the Levi regularity property.

\begin{ex}\label{ex.2.2}
Let $M\subset\CC{2}$ be a 2-nonminimal at the origin hypersurface
of the form $\im w=(\re w)^4|z|^2+(\re w)^2|z|^4+O(|z|^4|w|^4)$.
Then it is not difficult to check, that $M$ is Levi nondegenerate
in $M\setminus X$, but is not Levi regular at the origin.
\end{ex}

However, it will be shown in the next section that for $n=2$ the
Levi regularity condition holds for \it spherical \rm nonminimal
hypersurfaces.

The Levi regularity condition can be naturally reformulated in
terms of the real defining function $(\re w)^m\Phi(z,\bar z,\re
w)$ above: one should require that the function $\Phi$ can be
expanded as
\begin{equation}\label{e.2.3}
\Phi(z,\bar z,\re w)=H(z,\bar z,\re w)+\tilde\Phi(z,\bar z,\re w)
\end{equation}
with $H(z,\bar z,\re w)$ being a nondegenerate hermitian form in
$z,\bar z$ for each $w$, $\tilde\Phi(z,0,\re
w)\equiv\tilde\Phi(0,\bar z,\re w)\equiv 0$ and also
$\tilde\Phi(z,\bar z,\re w)=O(||z||^3)$. The equivalence of the
definitions follows from the fact that the functions $\varphi$ and
$\Phi$ from the exponential and the real defining equations
respectively are related as
$$\left.\varphi\right|_{M\setminus
X}=\left.\frac{1}{i}\mbox{log}\frac{w}{\bar w}\right|_{M\setminus
X}=\frac{1}{i}\mbox{log}\frac{1+iu^{m-1}\Phi(z,\bar
z,u)}{1-iu^{m-1}\Phi(z,\bar z,u)}=2u^{m-1}\Phi(z,\bar
z,u)+O(u^{3m-3}\Phi^3(z,\bar z,u))).$$ Here $w=u+iv$.

\subsection{Real hypersurfaces and second order differential equations.}
Using the Segre family of a Levi nondegenerate real hypersurface
$M\subset\CC{n}$ , one can associate to it a system of second
order holomorphic PDEs with $1$ dependent and $n-1$ independent
variables. The corresponding remarkable construction goes back to
E.~Cartan \cite{cartan2},\cite{cartan} and Segre \cite{segre}, and
was recently revisited in \cite{sukhov1}, \cite{sukhov2},
\cite{nurowski}, \cite{gm}, \cite{merker} (see also references
therein). We describe here the procedure for the case $n=2$, which
will be relevant for our purposes. In what follows we denote the
coordinates in $\CC{2}$ by $(z,w)$, and put $z=x+iy,\,w=u+iv$. Let
$M\subset\CC{2}$ be a smooth real-analytic hypersurface, passing
through the origin, and let $(U_1,U_2)$ be its standard pair of
neighbourhoods. In this case
 one associates with $M$ a second order holomorphic ODE, uniquely determined
 by the condition that it is satisfied by the
Segre family $\{Q_\zeta\}_{\zeta\in U_1}$ of $M$ in a
neighbourhood of the origin where the Segre varieties are
considered as graphs $w=w(z)$. More precisely, it follows from the
Levi nondegeneracy of $M$ near the origin that the Segre map
$\zeta\lr Q_\zeta$ is injective and also that the Segre family has
the so-called transversality property: if two distinct Segre
varieties intersect at a point $q\in U_2$, then their intersection
at $q$ is transverse. Thus, $\{Q_\zeta\}_{\zeta\in U_1}$ is a
2-parameter holomorphic w.r.t. $\bar\zeta$ family of holomorphic
curves in $U_2$ with the transversality property. It follows from
the holomorphic version of the fundamental ODE theorem (see, e.g.,
\cite{ilyashenko}) that there exists a unique second order
holomorphic ODE $w''=\Phi(z,w,w')$, satisfied by the graphs
$\{Q_\zeta\}_{\zeta\in U_1}$.

This procedure can be made more explicit if one considers the
 complex defining
 equation
$w=\rho(z,\bar z,\bar w)$ of $M$ near the origin. The Segre
variety $Q_p$ of a point $p=(a,b)$ close to the origin is given by
\begin{equation} \label{segre}w=\rho(z,\bar a,\bar b). \end{equation}
Differentiating \eqref{segre} once, we obtain
\begin{equation}\label{segreder} w'=\rho_z(z,\bar a,\bar b). \end{equation}
Considering \eqref{segre} and \eqref{segreder}  as a holomorphic
system of equations with the unknowns $\bar a,\bar b$, and
applying the implicit function theorem near the origin, we get
$$
\bar a=A(z,w,w'),\,\bar b=B(z,w,w').
$$
The implicit function theorem here is applicable as the Jacobian
of the system coincides with the Levi determinant of $M$ for
$(z,w)\in M$. Differentiating \eqref{segre} twice and plugging
there the expressions for $\bar a,\bar b$ finally yields
\begin{equation}\label{segreder2}
w''=\rho_{zz}(z,A(z,w,w'),B(z,w,w'))=:H(z,w,w').
\end{equation}
Now \eqref{segreder2} is the desired holomorphic second order ODE
$\mathcal E$.

The concept of a PDE system associated with a CR-manifold can be
generalized for various classes of CR-manifolds.  The
correspondence $M\lr \mathcal E(M)$ has the following fundamental
properties:

\begin{enumerate}

\item[(1)] Every local holomorphic equivalence $F:\, (M,0)\lr (M',0)$
between two  CR-submanifolds is an equivalence between the
corresponding PDE systems $\mathcal E(M),\mathcal E(M')$;

\item[(2)] The complexification of the infinitesimal automorphism algebra
$\mathfrak{hol}(M,0)$ of $M$ at the origin coincides with the Lie
symmetry algebra  of the associated PDE system $\mathcal E(M)$
(see, e.g., \cite{olver} for the details of the concept).

\end{enumerate}

For the proof and applications of the properties (1) and (2) in
various settings we refer to
\cite{sukhov1},\cite{sukhov2},\cite{nurowski},\cite{gm},\cite{merker}.
We emphasize that for a nonminimal at the origin hypersurface
$M\subset\CC{2}$ there is \it no \rm a priori way to associate
with $M$ a second order ODE or even a more general PDE system near
the origin. However, in Section 5 we provide a way to connect
nonminimal spherical real hypersurfaces in $\CC{2}$ with a class
of complex differential equations with an isolated meromorphic
singularity.


\subsection{Complex linear differential equations with an isolated singularity}
Complex linear ODEs form one of the most important and geometric
class of complex ODEs. We refer to \cite{ilyashenko}, \cite{ai},
\cite{bolibruh}, \cite{vazow} and references therein for various
facts and problems, concerning complex linear differential
equations. A first order linear system of $n$ complex ODEs in a
domain $G\subset\CC{}$ (or simply a linear system in a domain $G$
in what follows) is a holomorphic ODE system $\mathcal L$ of the
form $y'(w)=A(w)y$, where $A(w)$ is an $n\times n$ matrix-valued
holomorphic in $G$
 function and $y(w)=(y_1(w),...,y_n(w))$
is an $n$-tuple of unknown functions. Solutions of $\mathcal L$
near a point $p\in G$ form a linear space of dimension~$n$.
Moreover, all the solution $y(w)$ of $\mathcal L$ are defined
globally in $G$ as (possibly multiple-valued) analytic functions,
i.e., any germ of a solution near a point $p\in G$ of $\mathcal L$
extends analytically along any path $\gamma\subset G$, starting
at~$p$. A \it fundamental system of solutions for $\mathcal L$ \rm
is a matrix whose columns form some collection of $n$ linearly
independent solutions of $\mathcal L$.

If $G$ is a punctured disc centred at $0$, we call $\mathcal L$
\it a system with an isolated singularity at $w=0$. \rm An
important (and sometimes even a complete) characterization of an
isolated singularity is its \it monodromy operator \rm defined as
follows. If $Y(w)$ is some fundamental system of solutions of
$\mathcal L$ in~$G$ and~$\gamma$ is a simple loop about the
origin, then  the monodromy of $Y(w)$ w.r.t. $\gamma$ is given by
the right multiplication by a constant nondegenerate matrix $M$,
called \it the monodromy matrix. \rm The matrix $M$ is defined up
to a similarity, so that it defines a linear operator
$\CC{n}\lr\CC{n}$, which is  called the monodromy operator of the
singularity.

If the matrix-valued function $A(w)$ is meromorphic at the
singularity $w=0$, we call it a {\it meromorphic singularity}. As
the solutions of $\mathcal L$ are holomorphic in any proper sector
$S\subset G$ of a sufficiently small radius with the vertex at
$w=0$, it is important to study the behaviour of the solutions as
$w\rightarrow 0$. If all solutions of $\mathcal L$ admit a bound
$||y(w)||\leq C|w|^b$  in any such sector (with some constants
$C>0,\ b\in \mathbb R$, depending possibly on the sector), then
$w=0$ is called \it a regular singularity, \rm otherwise it is
called \it an irregular singularity. \rm In particular, in the
case of the trivial monodromy the singularity is regular if and
only if all the solutions of $\mathcal L$ are meromorphic in $G$.
L.~Fuchs introduced the following condition: a singular point
$w=0$ is called \it Fuchsian, \rm if $A(w)$ is meromorphic  at
$w=0$ and has a pole of order $\leq 1$ there. The Fuchsian
condition turns out to be sufficient for the regularity of a
singular point. Another remarkable property of a Fuchsian system
is that every formal holomorphic (and even formal meromorphic)
solution of a Fuchsian system is in fact convergent.

A scalar linear complex ODE of order $n$ in a
 domain $G\subset\CC{}$ is an ODE $\mathcal E$ of the form
$$z^{(n)}=a_{n}(w)z+a_{n-1}(w)z'+...+a_1(w)z^{(n-1)},$$ where $\{a_j(w)\}_{j=1,...,n}$ is
a given collection of holomorphic functions in $G$ and $z(w)$ is
the unknown function. By a reduction of $\mathcal E$ to a first
order linear system (see the above references  for various
techniques of doing that) one can naturally transfer most of the
definitions and facts, relevant to linear systems, to scalar
equations of order $n$. The main difference here is contained in
the appropriate definition of Fuchsian: a singular point $w=0$ for
an ODE $\mathcal E$ is called \it Fuchsian, \rm if the orders of
poles $p_j$ of the functions $a_j(w)$ satisfy the inequalities
$p_j\leq j$, $j=1,2,\dots,n$. The Theorem of Fuchs for $n$-th
order scalar ODEs says that \it a singular point of a linear
$n$-th order ODE is regular if and only if it is Fuchsian. \rm In
particular, if the monodromy of the equation is trivial, then the
Fuchsian condition is equivalent to the fact that all solutions of
the equation are meromorphic at the singular point $w=0$.

Further information on the classification and behaviour of
solutions for singular linear ODEs  can be found in
\cite{ilyashenko} or \cite{vazow}.

\subsection{Holomorphic vector fields and automorphisms.}
We next give  some preliminaries related to local automorphisms of
real hypersurfaces. By a {\it holomorphic vector field} in a
neighbourhood of the origin in $\CC{n}$ we mean a complex vector
field
$$f_1(z)\frac{\partial}{\partial z_1}+...+f_n(z)\frac{\partial}{\partial
z_n},$$ where the functions $f_1(z),...,f_n(z)$ are holomorphic in
a neighbourhood of the origin. Real parts of holomorphic vector
fields are precisely the real vector fields in $\CC{n}$ generating
flows of local biholomorphic transformations. Let now
$M\subset\CC{n}$ be a smooth real-analytic hypersurface containing
the origin. The \it infinitesimal automorphism algebra \rm of $M$
at the origin (we denote it by $\mathfrak{hol}\,(M,0)$ in what
follows) is the Lie algebra of germs at the origin of holomorphic
vector fields $X$ such that $\re X$ is tangent to $M$ at any point
$p\in M$ where it is defined. If this algebra is
finite-dimensional, we may assume that all its elements are
defined in the same neighbourhood of the origin. The importance of
the infinitesimal automorphism algebra stems from by the fact that
real parts of elements of $\mathfrak{hol}\,(M,0)$ are precisely
the real vector fields in $\CC{n}$ near the origin that generate
real flows of local  biholomorphic automorphisms of $M$ at $0$.

One can also consider the \it stability algebra \rm
$\mathfrak{aut}\,(M,0)$ of $M$ a the origin. This Lie algebra
consists of vector fields $X\in\mathfrak{hol}\,(M,0)$, vanishing
at $0$. Real parts of vectors fields, lying in
$\mathfrak{aut}\,(M,0)$, are precisely the real vector fields in
$\CC{n}$ near the origin, generating flows of local  biholomorphic
automorphisms of $M$ near the origin, preserving the origin. In
many nondegeneracy settings \cite{ber} this algebra is the tangent
algebra to the stability group of the germ $(M,0)$.

For the compact complex manifold $\CP{n}$ its automorphism group
consists of projective transformations (that is given by elements
of $\mbox{GL}(n+1,\CC{})$, naturally acting on the homogeneous
coordinates and considered up to a scaling). This Lie group is
usually denoted by $\mbox{PGL}(n+1,\CC{})$. It is generated by the
Lie algebra $\mathfrak{hol}(\CP{n})$, which is a certain algebra
of quadratic vector fields in each fixed affine chart (see, for
example, \cite{chern}). The Lie algebra $\mathfrak{hol}(\CP{n})$
is isomorphic to $\mathfrak{sl}(n,\CC{})$ as a Lie algebra (see
\cite{vinberg} for more details). For any nondegenerate
hyperquadric $\mathcal Q\subset\CP{n}$ the algebra
$\mathfrak{hol}(\CP{n})$ is the complexification of the
infinitesimal automorphism algebra $\mathfrak{hol}(\mathcal Q)$.
It will be also important for us that the natural action of
$\mbox{PGL}(n+1,\CC{})$ on $\mathfrak{hol}(\CP{n})$ (i.e., the
natural "coordinate-change" action of biholomorphisms from
$\mbox{PGL}(n+1,\CC{})$ on vector fields from
$\mathfrak{hol}(\CP{n})$) corresponds to the adjoint action of the
Lie group $\mbox{PGL}(n+1,\CC{})$ on its tangent algebra
$\mathfrak{sl}(n,\CC{})$ (Lie algebra automorphisms, corresponding
to this action, are sometimes called \it conjugacies \rm or \it
inner automorphisms). \rm In the matrix realization of above Lie
groups and algebras, conjugacies are simply automorphisms, given
by a matrix conjugation.

\subsection{Nonminimal spherical hypersurfaces}

We give in this section more detailed formulations of the results
in \cite{nonminimal}, which will be used in various sections of
the present paper.

\begin{dfn} A real-analytic hypersurface $M\subset\CC{n},$ containing a complex hypersurface $X\ni 0$, is called \it Segre regular in
a neighbourhood $U$ of the origin, \rm if  the Segre map $\lambda$
is locally injective in $U\setminus X$. \end{dfn} It is shown in
\cite{nonminimal} that if $M$ is Levi nondegenerate in $M\setminus
X$, then one can choose a neighbourhood $U\ni 0$ in such a way
that $M$ is Segre regular in $U$.

Assume now that $M$ is Levi nondegenerate in $M\setminus X$ and is
Segre regular in a neighbourhood $U$. Denote by $M^+,M^-$ the two
connected components of $M\setminus X$ and assume, in addition,
that one of the components (say, $M^+$) is $(k,l)$-spherical
(i.e., it can be locally biholomorphically mapped into a
hyperquadric $\mathcal Q\subset\CP{n}$ with $k$ positive and $l$
negative eigenvalues, $k+l=n-1$). The hypersurface $M$ in this
case is called \it pseudospherical. \rm  Then it is proved in
\cite{nonminimal} that

\smallskip

 \it the second component $M^-$ is also
$(k',l')$-spherical (with, possibly, $(k',l')\neq (k,l)$) and
there exists an open neighbourhood $U$ of $X$ in $\mathbb C^n$
such that for  $p\in (M \setminus X)\cap U$  any biholomorphic map
$\mathcal F_p$ of $(M,p)$ into a $(k,l)$\,-\,hyperquadric $\Q$
extends analytically along any path in $U\setminus X$ as a locally
biholomorphic map into $\CP{n}$. In particular, $\mathcal F_p$
extends to a possibly multiple-valued locally biholomorphic
analytic mapping $\mathcal F:\,U\setminus X\longrightarrow\CP{n}$
in the sense of Weierstrass. \rm \smallskip

The above theorem implies the existence of a nontrivial
biholomorphic invariant of a nonminimal spherical real
hypersurface called \it the monodromy operator. \rm To define it
we consider a generator $\gamma$ of $\pi_1(U\setminus X)$ with
$\gamma\ni p$ and consider the analytic continuation $\mathcal
F_{\gamma,p}$ of $\mathcal F_p$ along $\gamma$. There exists an
element $\sigma\in\mbox{Aut}(\CP{n})$ such that $\mathcal
F_{\gamma,p}=\sigma\circ \mathcal F_{p}$. It is convenient to
interpret $\sigma$ as an $(n+1)\times(n+1)$-matrix, defined up to
scaling, that we call \it the monodromy matrix. \rm The monodromy
matrix is defined up to similarity: namely, a replacement of the
mapping $\mathcal F_{p}:\,(M,p)\lr\CP{n}$ by any other mapping
$\tau\circ \mathcal F_{p}:\,(M,p)\lr\CP{n}$
 leads to a similar monodromy
matrix
\begin{equation}\label{changesigma}\tilde\sigma=\tau\circ\sigma\circ\tau^{-1}.\end{equation} Thus
we get a well-defined linear operator $\CC{n+1}\lr\CC{n+1}$,
defined up to scaling and independent of the choice of the initial
mapping $\mathcal F_p$, the target quadric $\mathcal Q$ and the
path $\gamma$, which is called the monodromy operator. If the
analytic continuation $\mathcal F_{\gamma,p}$ of the initial
mapping $\mathcal F_p$ leads to the same element $\mathcal F_p$,
then the monodromy operator is the identity. The analytic mapping
$\mathcal F$ in this case is a well-defined single-valued locally
biholomorphic mapping $U\setminus X\lr\CP{n}$.

\section{Formulations of the principal results}

We give in this section more detailed formulation of Theorems
1,~2, and 3, and also state some intermediate results that are of
independent interest.

The first result provides the existence of prenormal coordinates
for a nonminimal spherical hypersurface in $\CC{2}$. As was
explained earlier, prenormal coordinates do not exist for a
nonminimal Levi nonflat hypersurface in general.

\begin{thm}
Let $M\subset\CC{2}$ be a real-analytic nonminimal at the origin
hypersurface, and let $X$ be its complex locus. Suppose that
$M\setminus X$ is Levi nondegenerate and spherical. Then in
suitable local holomorphic coordinates near the origin, called
prenormal coordinates,  $M$ can be represented by an exponential
defining equation $w=\bar w e^{i\varphi(z,\bar z,\bar w)}$ with
\begin{gather}\label{complexdef}
\varphi(z,\bar z,\bar w)=(\bar w)^{m-1}\left(\pm
|z|^2+\sum\limits_{k,l\geq 2}\varphi_{kl}(\bar w)z^k\bar
z^l\right),
\end{gather}
or, equivalently, by a real defining equation $\im w=(\re w)^{m}
\Phi(z,\bar z,\re w)$ with
\begin{gather*}
\Phi(z,\bar z,\re w)= \pm|z|^2+\sum\limits_{k,l\geq 2}\Phi_{kl}
(\re w)z^k\bar z^l,
\end{gather*}
where  $\varphi_{kl}$ and $\Phi_{kl}$ are analytic functions near
the origin, and  $m\geq 1$ is the nonminimality order of $M$ at
the origin.
\end{thm}

To formulate the next result we will need the following
definition.

\begin{dfn}
We denote by $\mathcal P_0$ the class of nonminimal smooth
real-analytic hypersurfaces $M\subset\CC{2}$, containing the
complex hypersurface $X=\{w=0\}$, Levi-nondegenerate and spherical
in $M\setminus X$ and given in a neighbourhood $U$ of $0$ in
prenormal coordinates. If in addition $U$ is a polydisc chosen in
such a way that $M$ is Segre regular in $U$, we call $U$ a
neighbourhood \it associated with $M$. \rm We also call the
multiple-valued mapping $\mathcal F:\,U\setminus X\lr\CP{2}$,
extending a germ $\mathcal F_p:\,(M,p)\lr(S^3,p')$ (see Section
2.6), \it the mapping associated with $M$. \rm We call the
hypersurface $M\in\mathcal P_0$ \it positive \rm or {\it negative}
depending on the sign in~\eqref{complexdef}.
\end{dfn}

The $\mathcal P_0$-notation used in this paper is inherited from
the analytic theory of differential equations (see Section 6 for
details). Our next result establishes a fundamental connection
between
 hypersurfaces of class $\mathcal P_0$ and a special class of singular complex ODEs.
 In what follows in the paper we denote by $\Delta_\epsilon$ a disc, centred
 at $0$ of radius $\epsilon$, and by $\Delta^*_\epsilon$ the
 corresponding punctured disc.

\begin{thm}\label{t.3.3}
Suppose that $M\in\mathcal P_0$ and
$U=\Delta_\delta\times\Delta_\epsilon$ is the associated
neighbourhood. Then

\medskip

\noindent(i) There exists a second order ODE
\begin{equation}\label{w''(z)}
w''=-\frac{1}{w^m}(Az+B)(w')^2-\frac{1}{w^{2m}}(Cz^3+Dz^2+Ez+F)(w')^3,
\end{equation} where $A(w),B(w),C(w),D(w),E(w),F(w)$ are holomorphic functions in
the disc $\Delta_\epsilon$ such that~\eqref{w''(z)} is satisfied
by all Segre varieties $Q_p=\{w=w_p(z)\},\,p\in U\setminus X$,
considered as graphs $w=w(z)$.

\medskip

\noindent(ii) The ODE \eqref{w''(z)} and the complex defining
function of $M$, as in \eqref{complexdef}, are related as
\begin{gather}
\notag F(w)=2\varphi_{23}(w),\,A(w)=\pm 6i\varphi_{32}(w),\,B(w)=\pm 2i\varphi_{22}(w)-w^{m-1},\\
\label{ODEviadef}E(w)=6\varphi_{33} \pm 2i(m-1)\varphi_{22}w^{m-1}-8(\varphi_{22})^2 \mp 2i\varphi'_{22}w^m.\\
 \label{ODErelations}A(w)=\pm 3i\bar
F(w),\,C(w)=-\frac{1}{9}A^2(w),\,D(w)=\frac{1}{3}w^{2m}\left(\frac{A(w)}{w^m}\right)'-\frac{1}{3}A(w)B(w),\end{gather}
where the signs are determined by the sign of $M$. \smallskip

\noindent(iii) For a possibly smaller polydisc $U$ the Segre
varieties $Q_p$ of $M$ with $p\in
\Delta^*_\delta\times\Delta^*_\epsilon$, considered as graphs
$z=z(w)$, satisfy the second order meromorphic ODE $\mathcal
E(M)$, given by
\begin{equation}\label{d2z}
z''=\frac{1}{w^m}(Az+B)z'+\frac{1}{w^{2m}}(Cz^3+Dz^2+Ez+F),
\end{equation} where $A(w),B(w),C(w),D(w),E(w),F(w)$ are the same
as in \eqref{w''(z)}. The correspondence $M\lr \mathcal E(M)$
between hypersurfaces of class $\mathcal P_0$ and ODEs of the form
\eqref{d2z}, satisfying \eqref{ODErelations}, is injective.

\smallskip

We say that the ODE $\mathcal E(M)$ is \it associated with $M$.
\end{thm}

The main application of Theorem 3.3 is the possibility to
reformulate questions, concerning the initial hypersurface $M$, in
terms of the associated ODE $\mathcal E(M)$. This turns out to be
a powerful tool for the proofs of delicate facts concerning the
geometry of nonminimal hypersurfaces.

We start with the applications to the problem of analytic
continuation. Even though the defining equation of $M$ suggests
that one should consider Segre varieties of $M$ as graphs
$w=w(z)$, it appears more natural to consider them as graphs
$z=z(w)$ in appropriate local coordinates. This gives
characterization of nonminimal spherical hypersurfaces for which
the associated mapping $\mathcal F$ extend holomorphically to the
complex locus.

\begin{thm} Let $M\in \mathcal P_0$, $U$ be the associated neighbourhood, and
$\mathcal F$ the associated mapping. Then:

\smallskip

\noindent (i) There exist six (multiple-valued) analytic functions
$\alpha_j(w)$ and $\beta_j(w)$, $j=0,1,2$, in a punctured disc
$\Delta^*_\epsilon=\{0<|w|<\epsilon\}$ such that the mapping
$\mathcal F:\,U\setminus X\longrightarrow\CP{2}$ has the following
linear w.r.t. the variable $z$ representation in homogeneous
coordinates:
\begin{equation}\label{fralin}
\mathcal
F(z,w)=(\alpha_0(w)z+\beta_0(w),\alpha_1(w)z+\beta_1(w),\alpha_2(w)z+\beta_2(w))\end{equation}
In particular, $\mathcal F$ restricted to $U^0=\mathcal
F^{-1}(\CC{2})$, $U_0\subset U\setminus X$, is a linear-fractional
 w.r.t. $z$ mapping $U^0\longrightarrow\CC{2}$.
Moreover, $\mathcal F$ extends as a (multiple-valued) holomorphic
mapping $\mathbb{CP}^1\times \Delta^*_\epsilon \longrightarrow
\mathbb{CP}^2$ that is locally biholomorphic in $\CC{1}\times
\Delta^*_\epsilon$.

\smallskip

\noindent (ii) Each Segre variety $Q_p,\,p=(a,b)$, of $M$ with
$a,b\neq 0$, considered as a subset of the strip
$\mathbb{CP}^1\times \Delta^*_\epsilon$, extends to a graph
$\tilde Q_p=\{z=h_p(w)\}\subset \mathbb{CP}^1\times
\Delta^*_\epsilon$ of an appropriate (multiple-valued) analytic
mapping $h_p:\,\Delta^*_\epsilon\longrightarrow\CP{1}$. All
functions $h_p(w)$ satisfy the ODE $\mathcal E(M)$.\smallskip

\smallskip

\noindent (iii) The mapping $\mathcal F$ is single-valued if and
only if for each Segre variety $Q_p,\,p=(a,b),\,a,b\neq 0$, with
the extension $\tilde Q_p=\{z=h_p(w)\}$, the mapping $h_p(w)$ is
single-valued;

\smallskip

\noindent (iv) The mapping $\mathcal F$ extends to $X$
holomorphically if and only if for each Segre variety
$Q_p,\,p=(a,b),\,a,b \neq 0$, with the extension $\tilde
Q_p=\{z=h_p(w)\}$, the mapping $h_p(w)$ is single-valued and
extends to the origin holomorphically.\rm\end{thm}

\smallskip

Theorem 3.4 implies Theorem 2 of Introduction.

\smallskip

We will need the following existence theorem for singular complex
ODEs, which is applicable, in particular, to the ODE $\mathcal
E(M)$ of Theorem 3.3, provided the associated mapping is
single-valued.

\begin{thm} Consider a second order singular at the origin complex ODE\,
$\mathcal E$, given by
\begin{equation}\label{bb}
z''=\frac{1}{w}P(z,w)z'+\frac{1}{w^2}Q(z,w),
\end{equation}
with holomorphic in some polydisc
$\Delta_\delta\times\Delta_\epsilon$ functions $P(z,w)$ and
$Q(z,w)$.
 Suppose the ODE $\mathcal E$ satisfies the following condition: if a
local solution $z=\psi(w)$ of $\mathcal E$ near some point
$w_0\in\Delta^*_\epsilon$ admits an analytic continuation to an
annulus $\epsilon'<|w|<\epsilon''$,
$0<\epsilon'<\epsilon''<\epsilon$, then the analytic continuation
is single-valued. Suppose also that there exists $z_0\in
\Delta_\delta$ such that $Q(z_0,0)=0$. Then the ODE $\mathcal E$
has a holomorphic at the origin solution $z=h(w)$ with $h(0)=z_0$.
\rm \end{thm}

Combination of Theorem 3.4 with Theorem 3.5 yields the following
result.

\begin{thm}
Let $M\in\mathcal P_0$, $U$ be the associated neighbourhood,
$\mathcal F$ be the associated mapping, and $m\geq 1$ be the
nonminimality order of $M$ at $0$. If $M$ is of Fuchsian type,
then $\mathcal F$ extends to $X$ holomorphically if and only if it
is single-valued. In particular, if $m=1$, then $\mathcal F$
extends holomorphically to $X$ if and only if it is single-valued.
 \end{thm}

Theorem 3.6 implies Theorem 3 stated in Introduction. Next we use
the above results to study the behaviour of local automorphisms
for real hypersurfaces at nonminimal points. We formulate the
Dimension Conjecture, mentioned in Introduction, in two different
versions.

\medskip

\noindent{\bf Dimension Conjecture (weak version).} Let
$(M,0)\subset\CC{2}$ be a smooth real-analytic Levi nonflat germ.
Then the following bound for the dimension of the stability
algebra of $M$ at $0$ holds:
$$\mbox{dim}\,\mathfrak{aut}\,(M,0)\leq \mbox{dim}\,\mathfrak{aut}(S^3,o)=5,\,o\in S^3.$$

\medskip

\noindent{\bf Dimension Conjecture (strong version).} Let
$(M,0)\subset\CC{2}$ be a smooth real-analytic Levi nonflat germ,
and suppose that $M$ is not spherical at $0$. Then the following
bound for the dimension of the infinitesimal automorphism algebra
of $M$ at $0$ holds:
$$\mbox{dim}\,\mathfrak{hol}\,(M,0)\leq 5.$$

As explained in Section 8 only the nonminimal case remained open
for the complete proof of the strong version of the Dimension
Conjecture. To treat this case, we first prove the following
embedding theorem for the infinitesimal automorphism algebra of a
nonminimal pseudospherical hypersurface in $\CC{n}$.

\begin{thm} Let $M\subset\CC{n},\,n\geq 2$, be a real-analytic nonminimal at the origin pseudospherical hypersurface.
Let $\sigma$ be the monodromy operator of $M$. Then the
infinitesimal automorphism algebra $\mathfrak{hol}\,(M,0)$ can be
injectively embedded into the subalgebra
$c=z(\sigma)\cap\mathfrak{hol}\,(\mathcal Q)$, where
$z(\sigma)\subset\mathfrak{hol}\,(\CP{n})$ is the centralizer of
the element $\sigma\subset\mbox{Aut}\,(\CP{n})$.\end{thm}
\smallskip

Theorem 3.7, while being effective for hypersurfaces with
nontrivial monodromy, does not give new information in the case of
trivial monodromy. To treat the latter case, we use the
linear-fractional representation of $\mathcal F$ asserted in
Theorem 3.4, which gives the following.

\begin{thm} For any hypersurface $M\in\mathcal P_0$ the bound $\mbox{dim}\,\mathfrak{hol}\,(M,0)\leq 5$
 holds. \end{thm}

 Theorem 3.8 implies Theorem 1 in the introduction. Examples obtained in
\cite{belnew} and~\cite{kl} show that the bound in  this theorem
is indeed sharp. Combined with other known results on
automorphisms of real-analytic hypersurfaces in $\CC{2}$, Theorem
3.7 yields the strong version of the Dimension Conjecture.

\begin{thm} The Strong Dimension Conjecture holds true for any smooth
real-analytic hypersurface $M\ni 0$. \end{thm} \smallskip

In fact, we can formulate even a stronger statement. A nonminimal
at the origin smooth real-analytic hypersurface $M\subset\CC{2}$
is called \it a sphere blow-up, \rm if for some open neighbourhood
$U$ of the origin there exists a holomorphic mapping $\mathcal
F:\,U\longrightarrow\CC{2}$ such that $\mathcal F(M)\subset S^3$,
$\mathcal F$ is locally biholomorphic in the complement
$U\setminus X$ of the complex locus $X\subset M$ and $\mathcal
F(X)=\{p'\}$ for some point $p'\in S^3$. Observe that \it not \rm
every nonminimal and spherical in $M\setminus X$ hypersurface is a
sphere blow-up, as the associated mapping $\mathcal F$ in this
case might not extend holomorphically to the complex locus $X$. We
then obtain the following characterization of all real-analytic
hypersurfaces with high-dimensional automorphism algebra.

\begin{thm}
Let $M\subset\CC{2}$ be a smooth real-analytic hypersurface,
passing through the origin. Then one of the following mutually
exclusive conditions hold.

\medskip

\begin{enumerate}

\item[(1)] $\mbox{dim}\,\mathfrak{hol}\,(M,0)=\infty$ and
$(M,0)$ is equivalent to the germ of the real hyperplane $\{\im
w=0\}\subset\CC{2}$.\smallskip

\item[(2)] $\mbox{dim}\,\mathfrak{hol}\,(M,0)=8$, and
$(M,0)$ is equivalent to the germ of the 3-sphere
$S^3\subset\CC{2}$.\smallskip

\item[(3)] $\mbox{dim}\,\mathfrak{hol}\,(M,0)=5$, and
$(M,0)$ is a nonminimal at the origin sphere blow-up. Moreover,
the Lie algebra $\mathfrak{hol}\,(M,0)$ is isomorphic to the
stability algebra $\mathfrak{aut}(S^3)$ of the 3-sphere
$S^3\subset\CC{2}$.
\smallskip

\item[(4)] $\mbox{dim}\,\mathfrak{hol}\,(M,0)\leq 4$.

\end{enumerate}
\end{thm}

Finally, we deduce the following description of the infinitesimal
automorphism algebras of real-analytic hypersurfaces
$M\subset\CC{2}$.

\begin{thm} Let $M\subset\CC{2}$ be a real-analytic
hypersurface, $0\in M$, and $M$ be Levi nonflat. Then
$\mathfrak{hol}\,(M,0)$ is isomorphic to a subalgebra in
$\mathfrak{hol}\,(S^3)\simeq\mathfrak{su}(2,1)$, and
$\mbox{dim}\,\mathfrak{hol}\,(M,0)\leq 5$ unless $(M,0)$ is
biholomorphic to $(S^3,o)$ for $o\in S^3$.\end{thm}

In the end  we would like to formulate the following conjecture.
It is possible to show that the Levi regularity condition,
guaranteeing the existing of prenormal coordinates
\eqref{prenormal}, holds on an open dense subset of the complex
locus $X$ of a nonminimal Levi nonflat hypersurface. Thus one can
use \eqref{fuchstype} to introduce the notion of Fuchsian type at
a generic point $p\in X$. Following carefully the arguments in
\cite{divergence} and in the present paper, one can see that the
sphericity of $M$ at a generic point does not seem to be necessary
for the effect of splitting nonminimal hypersurfaces into the
Fuchsian and non-Fuchsian classes (we refer again to the
regularity results \cite{ebenfelt},~\cite{jl2} in the
$1$-nonminimal case). Thus we conjecture the following.\smallskip

\noindent \bf Conjecture 3.12. \rm (i) The Fuchsian type is a
sufficient condition for convergence of formal equivalences
between nonminimal hypersurfaces. (ii) The Fuchsian type condition
is sufficient for  analyticity of CR-mappings between nonminimal
hypersurfaces. (iii) The Fuchsian type condition is sufficient for
the moderate growth, as $p\longrightarrow X$, of a mapping
$\mathcal F:\,(M,p)\lr (K,p')$ from $M$ into a compact algebraic
strictly pseudoconvex hypersurface $K$.\smallskip

The remaining of the paper is organized as follows. In Section 4
we prove the prenormalization Theorem 3.1. Its proof is based on
the globalization result \cite{nonminimal} and the properties of
the so-called complex Levi determinant. In Section 5 we use the
associated mapping $\mathcal F$ to obtain a holomorphic ODE  with
an isolated singularity at $w=0$, associated with $M\in\mathcal
P_0$. We then use the existence of prenormal coordinates to obtain
an associated ODE, arising from the defining function of the
hypersurface. Comparing the two ODEs, we prove the meromorphic
character of the associated ODE and obtain estimates for the
orders of poles. We then prove in the same section Theorems 3.3
and 3.4. The crucial step is to show that the associated mapping
$\mathcal F$ is linear-fractional in prenormal coordinates. The
latter fact is proved by means of solving explicitly certain
"Monge-Amp\`ere-like" equations $I_0(z,w)=I_1(z,w)=0$ (see
Section~5 for the notations).  The linear-fractional form of
$\mathcal F$ first allows us to specify the form of the associated
ODE (Theorem 3.3) and second obtain the globalization of Segre
varieties and characterize the analytic continuation in terms of
the behaviour of the extended Segre varieties (Theorem 3.4). As
the (globalized) Segre varieties are solutions of the associated
ODE $\mathcal E(M)$, we reformulate in Section~6 the analytic
continuation problem in terms of the growth of solutions for
$\mathcal E(M)$ as $w\longrightarrow 0$. We then reformulate the
Fuchsian type condition, described in the introduction, in terms
of $\mathcal E(M)$ and show that, under the Fuchsian type
assumption, the ODE $\mathcal E(M)$ can be reduced by a polynomial
substitution to a "Fuchs-like" ODE $\mathcal E^r(M)$. The latter
ODE is a particular case of a Briot-Bouquet type ODE. Section~6.2
is dedicated to various examples of hypersurfaces of class
$\mathcal P_0$ and the connections between the associated mapping,
the associated ODE and the analytic continuation problem. At the
end of the section we perform a crucial step in the proof of
Theorem 3.6, namely, we prove that solutions of the ODE $\mathcal
E(M)$ have a moderate growth, provided the reduced ODE $\mathcal
E^r(M)$ has at least one holomorphic solution, thus reducing
Theorem 3.6 (and Theorem 3 from Introduction) to Theorem~3.5. In
Section 7 we prove Theorem~3.5. For that one needs to prove the
existence of a formal solution, which involves a simple
nonresonant case and a more complicated resonant case. There we
significantly use the single-valuedness of the solutions and apply
the Poincar\'e perturbation method to show that the nonexistence
of a formal solution in the resonant case leads to
multiple-valuedness of some other (existing) solution, which gives
a contradiction. In Section~8 we discuss the connection between
the monodromy of the associated mapping $\mathcal F$ and the
infinitesimal automorphism algebra $\mathfrak{hol}(M,0)$. This
gives the  proof of Theorem~3.7 and the bound
$\mbox{dim}\,\mathfrak{hol}(M,0)\leq 4$ in the case of nontrivial
monodromy. We also prove the bound
$\mbox{dim}\,\mathfrak{hol}(M,0)\leq 5$ in the case when the
associated mapping $\mathcal F$ extends to the complex locus, and
thus reduce the proof of Theorem~3.8 to the case when $\mathcal F$
has a trivial monodromy, but does not extend to $X$. The remaining
case is treated in Section~9, essentially, by proving the fact
that the symmetry algebra of the associated ODE $\mathcal E(M)$
(at a \it singular \rm point) has dimension at most $4$. This
proves Theorem 3.8 and implies Theorems 3.9, 3.10 and 3.11.

\section{A prenormal form for a pseudospherical nonminimal hypersurface}\label{s.normal}

In this section we prove the prenormalization result stated in
Theorem 3.1. It is analogous to the preliminary normalization of
Chern-Moser in \cite{chern}. Throughout this section we denote the
coordinates in $\CC{n}$ by $(z,w) \in \cx^{n-1} \times \CC{}$,
$w=u+iv$, and for a polydisc $U$ centred at the origin we denote
by $U^z$ and $U^w$ its projections onto the $z$- and the
$w$-coordinate spaces respectively.  Further, we assume that
$M\subset\CC{n}$ is a nonminimal real-analytic hypersurface at
$0\in M$, $X\subset M$ is the complex hypersurface through $0$,
$M\setminus X$ is Levi-nondegenerate and the coordinates are
chosen as in Section 2.2.

As Example~\ref{ex.2.2} shows, Theorem 3.1 fails to hold in
general for nonminimal hypersurfaces, even if $M\setminus X$ is
Levi nondegenerate. The proof in the spherical case is based on
the study of the geometry of Segre varieties of $M$ near the
complex locus $X$ and uses in essential way the result
of~\cite{nonminimal}. The proof of the theorem is divided into
several propositions.

\begin{dfn}\label{d.3.2}
Let $M\subset\CC{n}$ be a real-analytic hypersurface as above,
 given in a neighbourhood $U\ni 0$ by a
complex defining equation $w=\rho(z,\bar z,\bar w)$. Consider
$\rho(z,\bar a,\bar b)$ as a function defined in a neighbourhood
of the origin in $\CC{2n-1}$. Then the \it Levi determinant \rm of
$M$ in $U$ is the real-analytic functional $n\times n$ determinant
$$\Delta(z,\bar a,\bar b)=\begin{vmatrix} \rho_{\bar a_1} &...
& \rho_{\bar a_{n-1}}& \rho_{\bar b}\\
\rho_{z_1\bar a_1} &...
& \rho_{z_1\bar a_{n-1}}& \rho_{z_1\bar b}\\
\hdotsfor{4}\\
\rho_{z_{n-1}\bar a_1} &... & \rho_{z_{n-1}\bar a_{n-1}}&
\rho_{z_{n-1}\bar b}\end{vmatrix}.$$  For points $(z,w) \in M\cap
U$, the number $\Delta(z,\bar z,\bar w)$ coincides with the
determinant of the Levi form of $M$, so $\Delta(z,\bar z,\bar
w)=0$ for $(z,w)\in X$, and $\Delta(z,\bar z,\bar w)\neq 0$ for
$(z,w)\in M\setminus X$, if $M\setminus X$ is Levi nondegenerate.
\end{dfn}

The Levi determinant becomes useful in determination of Levi
regularity. We first prove

\begin{propos}[Transversality Lemma]\label{p.3.3}
Suppose that $M\subset \cx^n$ is Segre-regular in a neighbourhood
$U\ni 0$ and $M$ is pseudospherical. Then if two distinct Segre
varieties $Q_p,Q_q,\, p,q\in U\setminus X$ intersect at a point
$s\in U\setminus X$, then their intersection is transversal.
\end{propos}

\begin{proof}
Applying \cite{nonminimal} we conclude that there is a germ of a
biholomorphic mapping $\mathcal F_s:\,(\cx^n,s)\longrightarrow
(\CP{n}, \mathcal F_s(s))$ such that $\mathcal F_s$ sends germs of
$Q_p$ and $Q_q$ at $s$ to two germs of projective hyperplanes
$L_1=\mathcal F_s(Q_p)$, $L_2=\mathcal F_s(Q_q)$. Since $\mathcal
F_s$ is biholomorphic, $L_1$ and $L_2$ are distinct, and so their
intersection at $\mathcal F_s(s)$ is transversal. The same holds
for the intersection $Q_p\cap Q_q$ at $s$.
\end{proof}

\begin{propos}\label{p.3.4}
Suppose that $M\subset \cx^n$ is Segre-regular in a polydisc $U\ni
0$ and $M$ is pseudospherical. Then the Levi determinant
$\Delta(z,\bar a,\bar b)$ of $M$ is nonzero in $U^z\times
U^z\times (U^w\setminus\{0\})$.
\end{propos}

\begin{proof}
Let $M$ be given by a complex defining equation $w=\rho(z,\bar
z,\bar w)$ and suppose that on the contrary, for some
$(z^*,a^*,b^*)\in U^z\times U^z\times U^w$ with $b^*\neq 0$, the
Levi determinant $\Delta(z^*,\bar a^*,\bar b^*)$ vanishes.
Consider an anti-holomorphic map $\mathfrak L_{z^*} : U^z\times
U^w\longrightarrow \CC{n}$ given by
$$
\mathfrak L_{z^*}(a,b)=\left(\rho(z^*,\bar a,\bar b),
\rho_{z_1}(z^*,\bar a,\bar b),...,\rho_{z_{n-1}}(z^*,\bar a,\bar
b) \right).
$$
The map $\mathfrak L_{z^*}$ assigns to $(a,b)$  the 1-jet of the
Segre variety $Q_{(a,b)}=\{w=\rho(z,\bar a,\bar b)\}$ at the point
$(z^*,\rho(z^*,\bar a,\bar b))$. Also note that $\Delta(z^*,\bar
a^*,\bar b^*)$ is the Jacobian of $\mathfrak L_{z^*}$ at the point
$(a^*,b^*)$. This implies that the map $\mathfrak L_{z^*}$ is
degenerate at $(a^*,b^*)$, and therefore, in any small
neighbourhood of $(a^*,b^*)$ there exist points $p=(a',b')$,
$q=(a'',b'')$ in $U\setminus X$, $p\neq q$, such that $\mathfrak
L_{z^*}(p) = \mathfrak L_{z^*}(q) $, in particular, the 1-jets of
the Segre varieties $Q_p$ and $Q_q$ coincide. On the other hand,
the Segre map of $M$ is locally injective, so for a sufficiently
small neighbourhood of $(a^*,b^*)$ we have $Q_p\neq Q_q$. This
contradicts Proposition~\ref{p.3.3}, which proves the result.
\end{proof}

\begin{propos}\label{p.3.5}
Suppose that for an $m$-nonminimal hypersurface $M\subset\CC{2}$
in a sufficiently small neighbourhood of the origin its Levi
determinant $\Delta(z,\bar a,\bar b)\neq 0$ for $b\neq 0$. Then
$M$ is Levi regular at $0$.
\end{propos}

\begin{proof}
Choose a neighbourhood $U\ni 0$ such that $M$ is given in $U$ by
an exponential defining equation $w=\bar w e^{i\varphi(z,\bar
z,\bar w)}$ and denote $\rho(z,\bar a,\bar b):=\bar b
e^{i\varphi(z,\bar a,\bar b)}$. Then $\rho_{\bar a}=i\bar
b\varphi_{\bar a}e^{i\varphi}$, $\rho_{\bar b}=e^{i\varphi}+i\bar
b \varphi_{\bar b}e^{i\varphi}$, $\rho_{z\bar a}=\bar b
e^{i\varphi}(i\varphi_{z\bar a}-\varphi_{\bar a}\varphi_{z})$,
$\rho_{z\bar b}=(i\varphi_{z}+i\bar b\varphi_{z\bar b}-\bar
b\varphi_{z}\varphi_{\bar b})e^{i\varphi}$, and so
$$
\Delta(z,\bar a,\bar b)=\bar b e^{2i\varphi}\left(-i\varphi_{z\bar
a}+ \bar b\varphi_{z\bar a}\varphi_{\bar b}- \bar b\varphi_{\bar
a}\varphi_{z\bar b}\right ).
$$
Applying the Weierstrass Preparation Theorem and taking possibly a
smaller polydisc $U$, we conclude that there exists an integer
$d\geq 0$ such that
$$
-i\varphi_{z\bar a}+\bar b\varphi_{z\bar a}\varphi_{\bar b}-\bar
b\varphi_{\bar a}\varphi_{z\bar b}= (\bar b)^d\delta(z,\bar a,\bar
b),
$$
where $\delta(z,\bar a,\bar b)$ is holomorphic  in $U^z\times
U^z\times U^w$ and does not vanish there. Since $\varphi (z, \bar
a, \bar b) = \bar b^{m-1}\psi$, $\psi = \psi_0(z, \bar a) + O(\bar
b)$, and $\psi_0\not\equiv 0$ does not contain harmonic terms, we
conclude that the expression $\frac{1}{\bar
b^{\,m-1}}\left(-i\varphi_{z\bar a}+\bar b\,\varphi_{z\bar
a}\,\varphi_{\bar b}-\bar b\,\varphi_{\bar a}\,\varphi_{z\bar
b}\right)|_{\bar b=0}$ is holomorphic in $z,\bar a$ and does not
vanish identically. Hence $d=m-1$, and
$$
\frac{1}{\bar b^{\,m-1}}\left(-i\varphi_{z\bar a}+\bar
b\,\varphi_{z\bar a}\,\varphi_{\bar b}-\bar b\,\varphi_{\bar
a}\,\varphi_{z\bar b}\right)(0,0,0)=\delta(0,0,0)\neq 0.
$$
Now since $\varphi_{\bar a}(z,\bar a,\bar b)=O(|z|), \varphi_{\bar
b}(z,\bar a,\bar b)=O(|z||a|)$, we conclude that $\frac{1}{(\bar
b)^{m-1}}\varphi_{z\bar a}(0,0,0)\neq 0$, which is equivalent to
Levi regularity.
\end{proof}

\noindent Propositions~\ref{p.3.4} and~\ref{p.3.5} imply the
following key

\begin{corol}\label{c.3.6}
Suppose that $M\subset \cx^2$ is an $m$-nonminimal at the origin
real-analytic hypersurface, and $M\setminus X$ is Levi
nondegenerate and spherical. Then $M$ is Levi regular at the
origin, i.e., it can be represented in each of the
forms~\eqref{e.2.2},~\eqref{e.2.3}.
\end{corol}

Now, in the presence of a \it leading Hermitian term \rm in the
defining equation of a nonminimal hypersurface, we can prove
Theorem 3.1 using Chern-Moser-type transformations.

\begin{proof}[Proof of Theorem 3.1 ]
First, using Corollary~\ref{c.3.6}, we may represent $M$ in some
polydisc $U$ by a defining equation $v = u^{m}\,\Phi(z, \bar z,
u)$, where $\Phi(z, \bar z, u)$ is given as in~\eqref{e.2.3}. In
the proof we denote by $O_{22}$ a power series in $z$, $\bar z$,
and $u$ containing only monomials $z^k\bar z^l u^j,\,k,l\geq 2$,
$j\geq 0$. We consider the expansion
$$
\tilde \Phi(z,\bar z,u)=z\lambda(\bar z,u)+\bar
z\bar\lambda(z,u)+O_{22} ,
$$
and $H(z,\bar z,u)=\alpha(u)|z|^2$, where $\alpha(u)\neq 0$ in
$U^w$. Define the function $f(z,w)
=\frac{\bar\lambda(z,w)}{\alpha(w)}$. Note that for $(z,w)\in M$,
we have $\bar w = u -i u^m\,O(|z|^2)$, so
$\left.\frac{\alpha(u)}{\alpha(\bar w)}\right|_M = 1 + O(|z|^2)$.
Therefore,
\begin{eqnarray*}
H(z+f(z,w),\bar z+\bar f(\bar z,\bar w), u)|_M= \\
(H(z,\bar z, u)+z\lambda(\bar z,\bar w)+\bar z\bar\lambda(z,w))|_M+O_{22}=\\
H(z,\bar z,u)+z\lambda(\bar z,u)+\bar z\bar\lambda(z,u)+O_{22}.
\end{eqnarray*}
From this it follows that the transformation
$$
z^*=z+f(z,w),\ w^*=w
$$
maps $M$ onto a hypersurface $M^*$ given by
\begin{equation}\label{e.3.2}
v^*=(u^*)^m\left(H(z^*,\bar z^*,u^*)+\sum\limits_{k,l\geq
2}\varphi^*_{kl}(u^*)(z^*)^k(\bar z^*)^l\right).
\end{equation}
Finally, to make $H$ independent of $u$ for $M$ given
by~\eqref{e.3.2}, we drop the asterisks for the sake of simplicity
and set $H(z,\bar z,u) = \alpha(u)|z|^2$ with $\alpha(u)\neq 0$.
Since $\alpha(u)$ is real-valued, we may assume first that
$\alpha(u)>0$. The transformation
$$
z^*=z\sqrt{\alpha(w)},\ w^*=w,
$$
where the root is chosen to be positive for the positive argument,
maps $M$ onto the hypersurface of the form~\eqref{prenormal}. This
follows from $\left|z\sqrt{\alpha(w)}\right|^2=H(z,\bar
z,u)+O_{22}$ whenever $(z,w) \in M$. The proof for $\alpha(u)<0$
is analogous.
\end{proof}

\section{Ordinary differential equation associated with a nonminimal spherical hypersurface}

In this section we prove Theorem~3.3, which describes a (singular)
second order ODE associated with a real hypersurface $M$. We also
prove Theorem 3.4, which allows us to reduce the study of the
associated mapping $\mathcal F$ to the study of solutions for the
associated ODE. As explained in Section~2, in general, a
nonminimal real hypersurface does not admit a second order ODE
associated with it. However, such ODE always exists in the
spherical case. The proof of this crucially depends on the global
properties of the mapping $\mathcal F$ associated with $M$.

\subsection{Existence of an associated singular ODE}

In what follows we assume that $M$ is a hypersurface of class
$\mathcal P_0$, $U$ is the associated neighbourhood and $\mathcal
F$ is the associated mapping. We start with introducing the {\it
regular set} $U^0= \mathcal F^{-1}(\CC{2})\subset U\setminus X$
and the {\it exceptional set} $E=(U\setminus X)\setminus
U^0=\mathcal F^{-1}(\mathbb{CP}^2\setminus\CC{2})$. The set $E$ is
the pre-image of the projective line
$\mathbb{CP}^2\setminus\CC{2}$, and since each element of
$\mathcal F$ at a point $p\in U\setminus X$ is biholomorphic in a
sufficiently small polydisc, $E$ is a locally countable union of
one-dimensional locally complex-analytic sets in $U\setminus X$.
This implies that $E$ has Hausdorff dimension 2, so that $U^0$ is
an open, connected
 and dense subset in $U\setminus X$, see, e.g, \cite{chirka}. We first study the
 behaviour of $\mathcal F$ on the regular set.

Fix a point $p\in U^0$ and a biholomorphic element $\mathcal F_p$
of $\mathcal F$ at $p$, defined in a sufficiently small polydisc
$U_p\subset U^0$. We claim that in $U_p$ there exists a second
order ODE that is satisfied by all Segre varieties of $M$ that
have nonempty intersection with $U_p$.

To prove the claim we write $\mathcal F_p =(f,g)$, as the
components of $\mathcal F$ are well-defined in $U^0$. For some
point $s\in Q_p$ there exists a polydisc $U_s \subset U \setminus
X$ such that $\cup_{q\in U_s} Q_q$ contains a neighbourhood of
$p$. By shrinking $U_p$, we may assume that this neighbourhood is
$U_p$. The $\mathfrak{Q}$-Segre property of $\mathcal F$ (see
\cite[Prop. 4.1]{nonminimal}) implies  that $\mathcal F_p$ sends
open pieces $Q_q\cap U_p$ of Segre varieties to affine complex
lines $\Pi_q\subset \cx^2$. For a fixed $q\in U_s$, assume that
$\Pi_q$ is given by
\begin{equation}\label{e.Pi}
z^*\lambda+w^*\mu=1
\end{equation}
for some $\lambda, \mu \in \mathbb C$, with $\mu\neq 0$. Setting
$(z^*,w^*)=(f,g)$ we see that the set $Q_q\cap U_p$, considered as
a graph $w=w_q(z),z\in U_p^z$, satisfies the equation:
\begin{equation}\label{e.4.3}
f(z,w_q(z))\lambda+g(z,w_q(z))\mu=1 .
\end{equation}
Differentiation of~\eqref{e.4.3} once w.r.t. $z$ yields
\begin{equation}\label{e.4.4}
f_z(z,w_q(z))\lambda+g_z(z,w_q(z))\mu+f_w(z,w_q(z))w'_q(z)\lambda+g_w(z,w_q(z))w'_q(z)\mu=0
.
\end{equation}
Consider \eqref{e.4.3} and~\eqref{e.4.4} as a system of linear
equations w.r.t. $\lambda$ and $\mu$. This system correctly
defines a map $(z,q) \to (\lambda, \mu)$. Indeed, suppose that for
some $(z^0,q^0)$ there exist more than one solution $(\lambda,
\mu)$ of the system \eqref{e.4.3},~\eqref{e.4.4}. Then
\eqref{e.4.3} implies that for all solutions $(\lambda, \mu)$ the
corresponding complex lines~\eqref{e.Pi} pass through the point
$\mathcal F(z^0, w_{q^0}(z^0))$, while~\eqref{e.4.4} implies that
the line $D\mathcal F (T_{(z^0, w_{q^0}(z^0))}Q_{q^0})$ is tangent
to~\eqref{e.Pi}. But since $D\mathcal F \ne 0$, it follows that
there exists only one such pair $(\lambda, \mu)$.  By solving the
system \eqref{e.4.3},~\eqref{e.4.4} we may express $\lambda$ and
$\mu$ as functions of $(z,q)$. By the invariance of Segre
varieties, these are, in fact, functions of $q$ only.

Differentiating~\eqref{e.4.3} twice yields (we omit the arguments
for simplicity of the formula)
\begin{equation}
w''(\lambda f_w+\mu g_w)+(w')^2(\lambda f_{ww}+\mu
g_{ww})+w'(2\lambda f_{zw}+2\mu g_{zw})+(\lambda f_{zz}+\mu
g_{zz}) =0 .
\end{equation}
Now, substitution of $\lambda$ and $\mu$ with solutions of the
system, gives
$$w''(f_wg_z-f_zg_w)=(f_z+f_ww')(g_{zz}+2g_{zw}w'+g_{ww}(w')^2)-(g_z+g_ww')(f_{zz}+2f_{zw}w'+f_{ww}(w')^2).$$
Since $\mathcal F_p$ is biholomorphic in $U_p$, the Jacobian
$J=f_wg_z-f_zg_w$ is nonzero in $U_p$, and we obtain
\begin{equation}\label{e.4.6}
w''=I_0+I_1w'+I_2(w')^2+I_3(w')^3,
\end{equation}
where
\begin{eqnarray}\label{e.4.7}
& \notag  I_0 =\frac{1}{f_wg_z-g_wf_z}\left(f_zg_{zz}-g_zf_{zz}\right) ,\\
& I_1 =\frac{1}{f_wg_z-g_wf_z}\left(f_wg_{zz}-g_wf_{zz}+2f_zg_{zw}-2g_zf_{zw}\right) ,\\
& \notag  I_2 =\frac{1}{f_wg_z-g_wf_z}\left(f_zg_{ww}-g_zf_{ww}+2f_wg_{zw}-2g_wf_{zw}\right) ,\\
& \notag I_3
=\frac{1}{f_wg_z-g_wf_z}\left(f_wg_{ww}-g_wf_{ww}\right) .
\end{eqnarray}
Furthermore, \eqref{e.4.6} is satisfied by $Q_q \cap U_p$ for all
$q\in U\setminus X$ with $Q_q\cap U_p\neq\emptyset$, not just for
$q\in U_s$. To see this, observe that there exists a pair of
polydiscs $U_1\Subset U_2 \Subset U_p$ with the property that if
$Q_q \cap U_1 \ne \varnothing$, then $Q_q \cap U_2$ is a graph
$w=w_q (z)$ over $U_2^z$. We shrink $U_p$ to $U_1$ and consider
$Q_q$ as graphs in $U_2$. Then the assertion follows from the
analytic dependence of $Q_q$ on $q$, and the fact that the set
$\{q:\,Q_q\cap U_p\neq\emptyset\}$ coincides with $\cup Q_r,r\in
U_p$, and hence is open and connected. This proves our claim.

Since $\mathcal F_p$ extends analytically along any path in $U^0$,
so do the four analytic elements
$I_0(z,w),I_1(z,w),I_2(z,w),I_3(z,w)$. On the other hand,
equation~\eqref{e.4.6} is independent of the choice of the germ of
$\mathcal F$ at~$p$. This can be argued as follows: from the
previous discussion we may conclude that $\{Q_q \cap U_p,\ q\in
U_s\}$ is an anti-holomorphic 2-parameter family of holomorphic
curves in $U_p$. Then this family has the transversality property,
i.e., the map $(z, \alpha, \beta) \to (z, w_{(\alpha, \beta)} (z),
w'_{(\alpha,\beta)}(z))$ is injective, and thus there exists a
unique  second order ODE $w'' = \theta (z, w, w')$ satisfied by
the family $\{Q_q \cap U_p,\ q\in U_s\}$ (see Section 2.3). From
this we conclude that the ODE~\eqref{e.4.6} is unique, i.e., is
independent of the choice of $\mathcal F_p$.

From the uniqueness of~\eqref{e.4.6} we conclude that the four
functions $I_0$, $I_1$, $I_2$, $I_3$ are holomorphic in all of
$U^0$, in particular, single-valued. For the same reason the
replacement of the mapping $\mathcal F$ by a mapping $\sigma\circ
\mathcal F,\,\sigma\in\mathfrak{hol}(\mathbb{CP}^2)$, does not
change the expressions $I_0$, $I_1$, $I_2$, $I_3$ in a
neighbourhood of $p$, provided $\sigma\circ \mathcal F_p$ is still
a mapping to the affine chart $\CC{2}\subset\mathbb{CP}^2$.

Take now a point $p\in E$ and replace $\mathcal F$ by the mapping
$\tilde{\mathcal F}=\sigma\circ \mathcal F$ such that
$\sigma\in\mathfrak{hol}(\mathbb{CP}^2)$ and $\tilde{\mathcal
F_p}=\sigma\circ \mathcal F_p\subset\CC{2}$ maps $U_p$ into
$\CC{2}$. Then the regular set $U^0$ is replaced by an open dense
set $\tilde U^0$ and using the map $\tilde{\mathcal F}$ we obtain
a second order ODE in a neighbourhood of $p$ with the properties
analogous to those of~\eqref{e.4.6}. This shows that
$I_0,I_1,I_2,I_3$ extend holomorphically to $E$.

Finally, since $I_0$, $I_1$, $I_2$, $I_3$ are holomorphic in
$U\setminus X$, we conclude that~\eqref{e.4.6} is satisfied by all
\it entire \rm (i.e., in all of $U\setminus X$) Segre varieties
$Q_q$ for $q\in U\setminus X$. We summarize our arguments in the
following key

\begin{propos}\label{p.4.3}
In the assumptions of Theorem~3.3, there exist four holomorphic in
$U\setminus X$ functions $I_0(z,w)$, $I_1(z,w)$, $I_2(z,w)$,
$I_3(z,w)$ such that the differential equation~\eqref{e.4.6} is
satisfied by the defining function $w_q(z)$ of each of the Segre
varieties $Q_q$, $q\in U\setminus X$, considered as graphs
$w=w_q(z)$. In each neighbourhood $U_p$, $p\in U^0$, and for any
element $\mathcal F_p$ of $\mathcal F$ with $\mathcal
F_p(U_p)\subset\CC{2}$ that has components $(f,g)$ as a map
$U_p\longrightarrow\CC{2}$, the four functions $I_0,I_1,I_2,I_3$
are given by~\eqref{e.4.7}. The expressions in~\eqref{e.4.7} are
invariant under the exterior action of elements $\sigma\in
Aut(\mathbb{CP}^2)$ with $\sigma(\mathcal F_p(U_p))\subset\CC{2}$.
At points $p\in E$ the four expressions $I_0,I_1,I_2,I_3$ can be
computed by formulas \eqref{e.4.7} applied to $\sigma\circ
\mathcal F_p$ with $\sigma(\mathcal F_p(U_p))\subset\CC{2}$.
\end{propos}

We now determine the behaviour of $I_0,I_1,I_2,I_3$ near the
complex locus $X$, using smoothness of $M$ given in prenormalized
form~\eqref{complexdef}.

\bigskip

\begin{propos}\label{p.4.4}
In the assumptions of Theorem~\ref{t.3.3} one has $I_0=I_1\equiv
0$. Furthermore, the functions $w^mI_2(z,2)$,  $w^{2m}I_3(z,w)$
extend to $X$ holomorphically, i.e., $I_2$ has the pole of order
$\leq m$ w.r.t $w$ at $w=0$ and $I_3$ has the pole of order $\leq
2m$ w.r.t $w$ at $w=0$.
\end{propos}

\begin{proof} We find a relationship between the defining function $\varphi$ as in~\eqref{complexdef} and \eqref{e.4.6}.
Assume first that $M$ is positive. Let $q=(a,b)\in U\setminus X$
so that $b\neq 0$. Then $Q_q$ is given by
\begin{equation}\label{e.4.8}
w=w(z)=\bar b e^{i\varphi(z,\bar a,\bar b)}=\bar b+i\bar
b^{m}z\bar a+O(z^2\bar a^2\bar b^m).
\end{equation}
Differentiation of \eqref{e.4.8} w.r.t. $z$ yields
$$
w'=i\bar b e^{i\varphi(z,\bar a,\bar b)}\varphi_z(z,\bar a,\bar
b)=(i\bar b-z\bar a\bar b^m+O(z^2\bar a^2\bar b^m))(\bar a\bar
b^{m-1}+O(z\bar a^2\bar b^{m-1}))=i\bar a\bar b^m+O(z\bar a^2\bar
b^{m}).
$$
Now set $\zeta:=\frac{w'}{w^m}$; this is well-defined because
$w(z)\neq 0$ for a Segre variety $Q_q$, $q\in U\setminus X$. Then,
combining the last equalities, we obtain
\begin{equation}\label{e.4.9}
\zeta=i\bar a+O(z\bar a^2\bar b^{m}).
\end{equation}
Equation~\eqref{e.4.9} shows that choosing possibly a smaller
initial neighbourhood $U$ of the origin we can apply the implicit
function theorem for the system~\eqref{e.4.8},~\eqref{e.4.9} near
the origin to get
\begin{equation}\label{e.4.10}
\bar a=P(z,w,\zeta)=-i\zeta+O(z\zeta^2w^{m}),\ \bar
b=Q(z,w,\zeta)=w+O(z\zeta w^m).
\end{equation}
Differentiating~\eqref{e.4.8} twice w.r.t. $z$ and plugging
~\eqref{e.4.10} into the result  we conclude that for each point
$(z,w)\in Q_q$ the values $z,w,w',w''$ are related by
\begin{equation}\label{OPQ}
w''=O(P(z,w,\zeta)^2 \cdot Q(z,w,\zeta)^m)=\sum\limits_{j\geq
0,k\geq 2,l\geq m}h_{jkl}z^j \left(\frac{w'}{w^m}\right)^k
w^l:=\Phi(z,w,w') .
\end{equation}
We note that the values $(z,w,w')$ in $\eqref{OPQ}$ belong to some
open domain $\Omega\subset\CC{3}$ (to see the openness we argue as
in the proof of Proposition~\ref{p.4.3} and consider the locally
biholomorphic mapping
$\chi:\,(z,q)\longrightarrow(z,w_q(z),w'_q(z)),\,q\in U\setminus
X,\,z\in U^z$; then simply $\Omega=\chi(U\setminus X)$). The
series $\eqref{OPQ}$ converges in $\Omega$ uniformly on compact
subsets. From $\eqref{complexdef}$ we get $\Phi(z,w,w')\equiv
I_0(z,w)+I_1(z,w)w'+I_2(z,w)(w')^2+I_3(z,w)(w')^3$ (as the
uniqueness implies). On the other hand, considering the
biholomorphic in $\Omega$ mapping $\psi:\,(z,w,w')\longrightarrow
(z,w,\frac{w'}{w^m})=(z,w,\zeta)$ we obtain a domain $\tilde
\Omega =\psi(\Omega)\subset\CC{3}$ and may consider the
holomorphic in $\tilde\Omega$ function
$H(z,w,\zeta):=\Phi(z,w,\zeta w^m)$. Then $H(z,w,\zeta)$ is given
in $\tilde\Omega$ by the power series
\begin{equation}\label{hjkl}
\sum\limits_{j\geq 0,k\geq 2,l\geq m}h_{jkl}z^j \zeta^k w^l.
\end{equation}
This implies that there exists a polydisc $V\subset\CC{3}$,
centred at $0$, such that $V\cap\tilde\Omega\neq\emptyset$ and the
power series~\eqref{hjkl} converges in $V$. Then we get in
$V\cap\tilde\Omega$:
 $$H(z,w,\zeta)=\Phi(z,w,\zeta
w^m)=I_0(z,w)+I_1(z,w)\zeta w^m+I_2(z,w)\zeta^2
w^{2m}+I_3(z,w)\zeta^3 w^{3m}.$$ Comparing with $\eqref{hjkl}$ we
finally obtain that
\begin{equation*}I_0(z,w)\equiv 0,\,I_1(z,w)\equiv
0,\,I_2(z,w)w^{2m}=\sum\limits_{j\geq 0,l\geq m}h_{j2l}z^j
w^l,\,I_3(z,w)w^{3m}=\sum\limits_{j\geq 0,l\geq m}h_{j3l}z^j w^l .
\end{equation*}
This proves the proposition in the positive case. The negative
case is analogous.
\end{proof} \rm

\subsection{Proof of statement (i) of Theorem~3.3.}

\begin{proof}
We start with the proof of the representation \eqref{fralin}.
Choose a point $p\in U^0$ and an element $\mathcal
F_p=(f,g):\,U_p\longrightarrow\CC{2}$ of $F$ in a polydisc
$U_p=U^z\times U^w_p \subset U^0_p$ centred at $p$. We
consider~\eqref{e.4.7} and use the two identities $I_0(z,w)\equiv
0$ and $I_1(z,w)\equiv 0$ proved in Proposition~\ref{p.4.4}. The
first one gives $f_zg_{zz}-g_zf_{zz}=0$, so that
$\left(\frac{g_z}{f_z}\right)_z=0$ assuming $f_z\neq 0$, while the
second implies $g_z=\lambda(w)f_z$, so that
\begin{equation}\label{e.4.12}
g(z,w)=\lambda(w)f(z,w)+\mu(w)
\end{equation}
for some $\lambda(w),\mu(w)$ holomorphic in $U^w_p$.
Plugging~\eqref{e.4.12} into $I_1(z,w)\equiv 0$ yields
\begin{equation}\label{e.4.13}
-\lambda'ff_{zz}-\mu'f_{zz}+2\lambda'(f_z)^2=0.
\end{equation}
By the implicit function theorem, using the condition $f_z\neq 0$,
there exists a function $P(\zeta,w)$, holomorphic in
$\{f(U_p)\}\times U^w_p$, such that $f_z=P(f,w)$. Then
$f_{zz}(z,w)=P(f(z,w),w)P_{\zeta}(f(z,w),w)$, which can be
rewritten in a simple form $f_{zz}=PP_f$. Substituting this
into~\eqref{e.4.13} gives
$$-\lambda'fPP_f-\mu'PP_f+2\lambda'P^2=0.$$
This can be considered, for each fixed $w$, as a first-order
elementary differential equation with the independent variable $f$
and the dependent variable $P$. Separation of variables gives
$\frac{P_f}{P}=\frac{2\lambda'}{\mu'+\lambda'f}.$ After
integration we conclude that $P=\rho(w)(\mu'(w)+\lambda'(w)f)^2$
for some function $\rho(w)$ holomorphic in $U^w$. So finally we
obtain another first-order elementary ODE
$$f_z=\rho(\mu'+\lambda'f)^2$$
with the independent variable $z$ and the dependent variable $f$.
Separating variables and integrating, we get
$-\frac{1/\lambda'}{\mu'+\lambda'f}=\rho z+\nu$ for an appropriate
holomorphic function $\nu(w)$. The latter equality implies that
$f$ is linear-fractional w.r.t. $z$ in $U_p$. Changing the
notation and using~\eqref{e.4.12}, we conclude that
\begin{equation}\label{e.4.14}
f(z,w)=\frac{\alpha_1(w)z+\beta_1(w)}{\alpha_0(w)z+\beta_0(w)},\
g(z,w)=\frac{\alpha_2(w)z+\beta_2(w)}{\alpha_0(w)z+\beta_0(w)}
\end{equation}
for appropriate holomorphic functions $\alpha_0(w),...,\beta_2(w)$
in $U^w$, which is equivalent to~\eqref{fralin} restricted to the
polydisc $U_p$. The collection $\alpha_0(w),...,\beta_2(w)$ is
defined uniquely up to scaling by a function $h(w)$, holomorphic
and nonzero in $U^w$. Returning to the assumption $f_z\neq 0$, we
note that the Jacobian $f_wg_z-g_wf_z$ is nonzero in $U_p$, so
that interchanging, if necessary, $f$ and $g$, we may still assume
that the condition $f_z\neq 0$ holds true in a sufficiently small
polydisc, centred at $p$, and conclude~\eqref{e.4.14} in the
general case as well. Note also that the form~\eqref{e.4.14} is
invariant under projective transformations in the image-space
$\CP{2}$. This means that after replacing $F$ by an appropriate
composition, equation~\eqref{e.4.14} holds for a small
neighbourhood of an arbitrary point $p\in U\setminus X$.

Consider now two polydiscs $U_p$ and $U_q$, $p,q\in U\setminus X$,
with $U_p\cap U_q\neq\emptyset$, and two elements $F_p,F_q$ there
such that $F_p=F_q$ in $U_p\cap U_q$. Given the
representations~\eqref{e.4.14} in both polydiscs, we may solve a
simple multiplicative Cousin problem to show that the collections
$\alpha_0(w),...,\beta_2(w)$ can be scaled in such a way that they
coincide in $U_p\cap U_q$. This means that each fixed collection
$\alpha_0(w),...,\beta_2(w)$ in a polydisc $U_p$ can be extended
analytically along an arbitrary path in $U\setminus X$, starting
at $p$, because the mapping $F$ does, and this proves the
representation~\eqref{fralin}.

To prove~\eqref{w''(z)} we find the functions $I_2(z,w)$,
$I_3(z,w)$, using the linear-fractional
representation~\eqref{fralin}. As the functions $w^mI_2(z,w)$,
$w^{2m}I_3(z,w)$ are holomorphic in the entire neighbourhood $U$
(from Proposition~\ref{p.4.4}), for the proof of~(4.1) it suffices
to show that $I_2$ is linear and $I_3$ is cubic w.r.t. the
variable $z$. In fact, we can do that in a neighbourhood
$U_p\subset U^0$ of a point $p\in U^0$. We first suppose
$\alpha_0(w)\not\equiv 0$ and change the form of the
representation~\eqref{fralin}, rewriting it in $U_p$ for a fixed
element $F_p$ of $F$ as
\begin{equation}\label{fgaalpha}
f(z,w)=\frac{\alpha}{z+\delta}+\beta,\,g(z,w)=\frac{a}{z+\delta}+b,
\end{equation}
where $\alpha(w),\beta(w),\delta(w),a(w),b(w)$ are meromorphic in
$U^w_p$. Then $I_1\equiv 0$ and~\eqref{e.4.7} imply
\begin{equation}\label{specialrelation}
\beta'a-b'\alpha\equiv 0,
\end{equation}
which after differentiation gives
$\beta''a-b''\alpha+\beta'a'-b'\alpha'=0$. Straightforward
computations using~\eqref{specialrelation} give
\begin{equation}\label{jacobian}
J=f_wg_z-g_wf_z=\frac{a'\alpha-\alpha'a}{(z+\delta)^3},
\end{equation}
\begin{gather}\label{i2}
I_2(z,w)=\left[\frac{a\alpha''-\alpha
a''}{a'\alpha-\alpha'a}\right]+3\left[\frac{b'\alpha'-\beta'a'}{a'\alpha-\alpha'a}\right](z+\delta),\\
\label{i3} I_3(z,w)=\left[\delta''+\delta'\frac{a\alpha''-\alpha
a''}{a'\alpha-\alpha'a}\right]+\left[\frac{a''\alpha'-\alpha''
a'}{a'\alpha-\alpha'a}+3\delta'\frac{b'\alpha'-\beta'a'}{a'\alpha-\alpha'a}\right](z+\delta)+\\
\notag
+\left[\frac{\beta'a''-b'\alpha''+\alpha'b''-a'\beta''}{a'\alpha-\alpha'a}\right](z+\delta)^2+
\left[\frac{\beta'b''-b'\beta''}{a'\alpha-\alpha'a}\right](z+\delta)^3.
\end{gather}
The identities~\eqref{i2} and~\eqref{i3} demonstrate now the
desired polynomial dependence. Suppose that $\alpha_0 \equiv 0$.
Since $f$ is a local biholomorphism, $\alpha_1$ and $\alpha_2$
cannot be both  identically zero. Thus, after relabelling the
functions, we return to the previous case. This completely proves
statement (i) of Theorem 3.3.
\end{proof}

\subsection{Proof of Theorem 3.4.}
In what follows, by $\mathcal M(0)$ (resp. $\mathcal O(0)$) we
denote the space of germs at the origin of meromorphic (resp.
holomorphic) functions in $w\in\CC{}$.

\begin{proof}[Proof of Theorem 3.4]
 In order to prove
statement (i) we need, in view of Section 5.2, to show that
$\mathcal F$ extends from $U\setminus X=\{|z|<\delta\}\times
\Delta^*_\epsilon$ to $\CP{1}\times \Delta^*_\epsilon$
analytically and the restriction of $\mathcal F$ to $\CC{}\times
\Delta^*_\epsilon$ is locally biholomorphic in
$\Delta^*_\epsilon$. Using representation~\eqref{fralin}, we
extend $\mathcal F$ as
\begin{equation}\label{fralin0}
\tilde{\mathcal F}(z,w):=(\alpha_0(w)\bold z+\beta_0(w)\bold
t,\alpha_1(w)\bold z+ \beta_1(w)\bold t,\alpha_2(w)\bold
z+\beta_2(w)\bold t)\in\CP{2}
\end{equation}
(we fixed here a germ of each of the functions
$\alpha_0,...,\beta_2$ and denote by $(\bold z,\bold t)$ the
homogeneous coordinates in $\CP{1}$). To prove that
$\tilde{\mathcal F}$ is, in fact, analytic, we need to show that
the 3 expressions $\alpha_0(w)\bold z+\beta_0(w)\bold
t,\alpha_1(w)\bold z+\beta_1(w)\bold t,\alpha_2(w)\bold
z+\beta_2(w)\bold t$ cannot vanish for $(\bold z,\bold t)\neq
(0,0)$, $0<|w|<\epsilon$.

We first observe that
$\alpha_0(w^*)=\alpha_1(w^*)=\alpha_2(w^*)=0$ is not possible for
any fixed $w^*,\,0<|w^*|<\epsilon$, since otherwise,
by~\eqref{fralin}, $\mathcal F(z,w^*)$ is independent of $z$, but
$\mathcal F$ is biholomorphic in $U\setminus X$. Assume now that
for some $\bold z^*, \bold t^*, w^*$ one has
$$
\alpha_0(w^*)\bold z^*+\beta_0(w^*)\bold t^*=\alpha_1(w^*)\bold
z^*+\beta_1(w^*)\bold t^*= \alpha_2(w^*)\bold
z^*+\beta_2(w^*)\bold t^*=0 .
$$
Suppose first $\bold t^*\neq 0$. Then $\beta_j (w^*)=-\frac{\bold
z^*}{\bold t^*}\alpha_j,\,j=0,1,2$. Then for $z\ne z^*/t^*$,
$$
\mathcal \mathcal F (z, w^*) = \left[ \alpha_0(w^*) (z -
\frac{z^*}{t^*}),\ \alpha_1(w^*)(z - \frac{z^*}{t^*},\
\alpha_2(w^*)(z - \frac{z^*}{t^*}) \right]  = [\alpha_0(w^*),\
\alpha_1(w^*),\ \alpha_2(w^*)] .
$$
This means that the line $\{(z,w^*)\}$ is mapped into a point,
which is a contradiction. Similarly, if $t^*=0$ then, in view of
$z^*\neq 0$, we conclude that
$\alpha_0(w^*)=\alpha_1(w^*)=\alpha_2(w^*)=0$ which is not
possible by the above argument.

To show that all elements of $\tilde{\mathcal F}$, i.e., local
maps obtained by analytic continuation, at points lying in
$\CC{}\times \Delta^*_\epsilon$ are locally biholomorphic, we fix
$p=(z^*,w^*)\in \CC{}\times \Delta^*_\epsilon$, choose a polydisc
$U_p\subset \CC{}\times \Delta^*_\epsilon$ and replace
$\tilde{\mathcal F}$, if necessary, with $\sigma\circ
\tilde{\mathcal F}$ for an appropriate
$\sigma\in\mbox{Aut}(\CP{2})$ in order to have $\tilde{\mathcal
F}(U_p)\subset\CC{2}$. Note that $\tilde{\mathcal F}_p$ admits a
single-valued extension to $U(w^*)\times\CC{}$ for some disc
$U(w^*)$, centred at $w^*$, using~\eqref{fralin0}.
Then~\eqref{fgaalpha} and~\eqref{jacobian} show that $\mathcal
F_p$ is biholomorphic near $p$, unless
$(a'\alpha-\alpha'a)(w^*)=0$. Choosing now $z$ such that
$|z|<\delta$ and $\mathcal F(z,w^*)\in\CC{2}$, which is possible
since $\mathcal F_p$ maps an open piece of the line
$\CC{}\times\{w^*\}$  to $\CC{2}$, we conclude that $\mathcal F_p$
is not biholomorphic at $(a,w^*)\in U\setminus X$. This is a
contradiction, and statement (i) is proved.

In order to prove (ii), we first fix $p=(p_1,p_2)\in U$, $p_1,
p_2\neq 0$, and consider $Q_p$ as the graph $w=\theta_p(z)$.
Expanding as in~\eqref{e.4.8}, we get $\theta_p(z)=\bar p_2+i\bar
p_1\bar p_2^m z+O(z^2\bar p_1^2\bar p_2^{m})$. Choosing now a
possibly smaller polydisc $U$, we may assume $\theta_p(z)$ is
injective in $\{|z|<\delta\}$, \it once for all $p$ as above. \rm
Indeed, $\theta_p(z)=\theta_p(z^*)$ implies from~\eqref{e.4.8}
that $(z-z^*)[1+O(|p_1|)]=0$, and that implies injectivity of
$\theta_p(z)$ for all $p$. We may then consider the inverse
holomorphic function $z=\psi(w)$ in some domain $\Delta_p\subset
\Delta^*_\epsilon$. The graphs $w=\theta(z)$ and $z=\psi(w)$ both
coincide with $Q_p$. As $Q_p$ is simply-connected, we may consider
a single-valued restriction $\mathcal F_p$ of $\mathcal F$ to a
simply-connected neighbourhood $V$ of $Q_p$. Then $\mathcal
F_p(Q_p)$ is contained in a projective line
$\lambda_0\xi_0+\lambda_1\xi_1+\lambda_2\xi_2=0$.
From~\eqref{fralin0}, the substitution
$$
(\xi_0,\xi_1,\xi_2)=(\alpha_0(w)\bold z+\beta_0(w)\bold
t,\alpha_1(w)\bold z+ \beta_1(w)\bold t,\alpha_2(w)\bold
z+\beta_2(w)\bold t)
$$
into the equation of the projective line shows that $Q_p$ is a
subset of a bigger set
\begin{equation}\label{e.lines}
\{P(w)\bold z+Q(w)\bold t=0\}\subset \CP{1}\times
\Delta^*_\epsilon
\end{equation}
for some (multiple-valued) analytic functions $P(w),Q(w)$ in
$\Delta^*_\epsilon$, where $P(w)$ is not a zero function,
as~\eqref{e.lines} contains the graph $\{z=\psi(w)\}=Q_p$. Hence,
there exists a (multiple-valued) analytic mapping
$h_p(w):\,\Delta^*_\epsilon\longrightarrow \CP{1}$ such that the
graph $\{z=h_p(w)\}$ is contained in~\eqref{e.lines} (in fact, the
union of the  graph with a countable collection of horizontal
projective lines $\{w=const\}$ is given by~\eqref{e.lines}. The
latter means that $Q_p$ is contained in the graph $z=h_p(w)$, as
required for statement (ii).

To prove (iii) we note that the mapping $\mathcal F$ is
single-valued if and only if the functions $\alpha_j(w)$,
$\beta_j(w)$, $j=0,1,2,$ can be scaled to be single-valued. Now
for each Segre variety $Q_p$ of a point in $p=(p_1,p_2)$, $p_1,
p_2 \neq 0$, we may represent $h_p(w)$ explicitly,
using~\eqref{e.lines}, as
\begin{equation}\label{hp}
h_p(w)=-\frac{\lambda_0\beta_0(w)+\lambda_1\beta_1(w)+\lambda_2\beta_2(w)}
{\lambda_0\alpha_0(w)+\lambda_1\alpha_1(w)+\lambda_2\alpha_2(w)}.
\end{equation}
As $\mathcal F$ is locally biholomorphic, the parameter
$\lambda\in\CP{2}$ in \eqref{hp} runs over some open subset of
$\CP{2}$. This implies that $h_p(w)$ as in \eqref{hp} is
single-valued for all $p\in U\setminus X$ if and only if the
functions $\alpha_j(w)$ and $\beta_j(w)$ can be scaled to be
single-valued. This completes the proof of (iii).

The proof of (iv) also uses representation \eqref{hp} and is
analogous to (iii). However, one needs to take care of certain
details. Suppose first that $\mathcal F$ extends holomorphically
to $X$. Replacing $\mathcal F$ by $\sigma\circ \mathcal F$ for
some $\sigma\in\mbox{Aut}(\CP{2})$ if necessary, we may use the
representation \eqref{fgaalpha}. For each $z\in\Delta_r$ consider
the discrete set $E_z=\{w\in\Delta^*_\epsilon:\,z=-\delta(w)\}$.
Then, considering the two expressions $f|_{z=z_0},g|_{z=z_0}$ as
in \eqref{fgaalpha} for a fixed $z_0\in\Delta_r$, we conclude that
these expressions, defined on the set $E_{z_0}$, extend to $w=0$
meromorphically. Hence, they extend meromorphically to the disc
$\Delta_\epsilon$. The latter fact, applied to an arbitrary
$z_0\in\Delta_r$, implies that
$\alpha(w),a(w),\beta(w),b(w),\delta(w)\in \mathcal M(0)$. We
conclude that the functions $\alpha_j(w),\beta_j(w)\in\mathcal
M(0)$ in \eqref{hp}, so that $h_p(w)\in\mathcal M(0)$, as
required.

Suppose now that each of the functions $h_p(w)\in\mathcal M(0)$.
After taking a composition of $\mathcal F$ with an element of
$\mbox{Aut}(\CP{2})$, the representation \eqref{fgaalpha} can be
applied (note that, from statement (iii) of the theorem, all
functions in \eqref{fgaalpha} are single-valued). Then \eqref{hp}
takes the form
\begin{equation}\label{hpaalpha}
h_p(w)=-\delta(w)-\frac{\lambda_1\alpha(w)+\lambda_2
a(w)}{\lambda_0+\lambda_1\beta(w)+\lambda_2 \beta(w)} .
\end{equation}
 Using the fact that the right-hand side in
\eqref{hpaalpha} belongs to the class $\mathcal M(0)$ for
arbitrary $(\lambda_0,\lambda_1,\lambda_2)\in\CP{2}$, we conclude
that $\alpha(w),a(w),\beta(w),b(w),\delta(w)\in \mathcal M(0)$. We
then can assume, performing in \eqref{fralin} scaling by an
appropriate $w^l,l\in\mathbb{Z}$, that the functions
$\alpha_j(w),\beta_j(w)\in\mathcal O(0)$ in \eqref{fralin} and,
moreover, that at least one of the six functions is nonzero at
$w=0$. It is then clear that \eqref{fralin} with
$\alpha_j(w),\beta_j(w)\in\mathcal O(0)$ allows us to extend the
mapping $\mathcal F$ to any point $(z_0,0)\in X,\,z_0\in\Delta_r$,
unless there exists $z_0\in\Delta_r$ such that
$$
\alpha_0(0)z_0+\beta_0(0)=\alpha_1(0)z_0+\beta_1(0)=\alpha_2(0)z_0+\beta_2(0)=0
.
$$
We claim that this is not possible. Indeed, assume, without loss
of generality, that $z_0=0$. Then
$\beta_0(0)=\beta_1(0)=\beta_2(0)=0$, and for some $j\in\{0,1,2\}$
we have $\alpha_j(0)\neq 0$. Applying now \eqref{hp}, we conclude
that for an appropriate open dense set of the projective line,
determined by an element
$(\lambda_0,\lambda_1,\lambda_2)\in\CP{2}$, the corresponding
function $z=h(w)$, as in \eqref{hp}, satisfies $h(0)=0$. Denote by
$\mathcal Q\subset\CP{2}$ the quadric, containing $\mathcal
F(M\setminus X)$. Since the set of projective lines $L$ in
$\CP{2}$ with $L\cap\mathcal Q=\emptyset$ is open, we choose a
graph $z=h(w)$, as in \eqref{hp}, such that $h(0)=0$ and $\mathcal
F(\{z=h(w),w\neq 0\})\cap\mathcal Q=\emptyset$. However, $\mathcal
F(\{z=h(w),w\neq 0\})\cap\mathcal Q$ contains the set
$$
\mathcal F(\{z=h(w),\,\im w=\rho(h(w),\overline{h(w)},\re
w),\,0<|w|<\epsilon\}),
$$
where $\im w=\rho(z,\bar z,\re w)$ is the defining function of the
hypersurface $M$ with $d\rho(0)=0$. Since $\{z=h(w),\,\im
w=\rho(h(w),\overline{h(w)},\re w),\,|w|<\epsilon\}\subset M$ is a
nonconstant real curve passing through the origin and $\mathcal F$
is locally biholomorphic for $w\neq 0$, we obtain a contradiction.
The proof for $0<|z_0|<r$ is analogous.
\end{proof}

\subsection{Proof of statements (ii) and (iii) of Theorem~3.3.}
The following two computations furnish the proof of part~(ii).

\begin{propos}
The following relations hold for the equation \eqref{w''(z)}:
\begin{equation}\label{*relations}
C(w)=-\frac{1}{9}A^2(w),\ \
D(w)=\frac{1}{3}w^{2m}\left(\frac{A(w)}{w^m}\right)'-\frac{1}{3}A(w)B(w).
\end{equation}
\end{propos}

\begin{proof} By taking the composition with an appropriate element $\sigma \in \mbox{Aut}(\CP{2})$,
we choose the associated mapping $\mathcal F$ to be given as in
\eqref{fgaalpha} with $\delta \ne 0$. Using the representations
$I_2 =-\frac{Az +B}{w^m}$ and $I_3 =
-\frac{Cz^3+Dz^2+Ez+F}{w^{2m}}$ from \eqref{w''(z)}
and~\eqref{complexdef}, and applying~\eqref{i2} and~\eqref{i3}, we
obtain
$$
-\frac{A(w)}{3w^m}=\frac{b'\alpha'-\beta'a'}{a'\alpha-\alpha'a},\
\
-\frac{C(w)}{w^{2m}}=\frac{\beta'b''-b'\beta''}{a'\alpha-\alpha'a}.
$$
We let $k(w)=\frac{a(w)}{\alpha(w)}$. Using~\eqref{jacobian} we
conclude that $k(w)$ is not a constant. Then, using
\eqref{specialrelation} and expressing everything in term of $k$,
$\alpha$, and $\beta$, we calculate that
\begin{equation}\label{b'kbeta}
b'=k\beta',\ \ \frac{A(w)}{3w^m}=\frac{\beta'}{\alpha},\ \
\frac{C(w)}{w^{2m}}=-\frac{\beta'^2}{\alpha^2} ,
\end{equation}
and so $C(w)=-\frac{1}{9}A^2(w)$. Further,~\eqref{i2}
and~\eqref{i3} show that
$$
-\frac{B(w)}{w^m}=\frac{a\alpha''-\alpha
a''}{a'\alpha-\alpha'a}-\frac{A(w)}{w^m}\delta,\ \
-\frac{D(w)}{w^{2m}}=\frac{\beta'a''-b'\alpha''+\alpha'b''-a'\beta''}{a'\alpha-\alpha'a}-\frac{3C(w)}{w^{2m}}\delta.
$$
Expressing everything in terms of $k$, $\alpha$, $\beta$, and
$\delta$ again gives~\eqref{*relations}.
\end{proof}

\begin{propos}
The following relations hold between the ODE \eqref{w''(z)} and
the exponential defining equation~\eqref{complexdef} of an
$m$-nonminimal hypersurface $M\in\mathcal P_0$:
\begin{gather}\notag
F(w)=2\varphi_{23}(w),\,A(w)=\pm 6i\varphi_{32}(w),\,B(w)=\pm 2i\varphi_{22}(w)-w^{m-1},\\
\label{*defequat}
E(w)=6\varphi_{33} \pm 2i(m-1)\varphi_{22}w^{m-1}-8(\varphi_{22})^2 \mp 2i\varphi'_{22}w^m.\\
\notag A(w)=\pm 3i\bar F(w).\end{gather}
\end{propos}

\begin{proof}
Consider the case when $M$ is positive. We use the form
$Q_{(a,b)}=\{w=\bar be^{i\varphi(z,\bar a,\bar b)}\}$ for Segre
varieties of $M$ and substitute this representation into
\eqref{w''(z)}. As a result we obtain an identity for two power
series in $z,\bar a,\bar b$.  We rewrite both sides of this
identity as power series in $z$ and $\bar a$ with coefficients
depending on $\bar b$. If we equate the coefficients of $\bar
a^3$, we obtain $2\phi_{23}(\bar b) = F(\bar b)$. If we equate the
terms $z \bar a^2$ we obtain $6i \phi_{32} (\bar b) = A(\bar b)$.
Similar computations for $\bar a^2$ and $z^3 \bar a$ give the
formulas for $B$ and $E$.

In order to prove the relation $A(w)=3i\bar F(w)$ we consider the
reality condition~\eqref{reality} as equality of power series in
$z$, $\bar z$, and $\bar w$, and compare the terms with $z^3\bar
z^2$. Taking into account that $\varphi$ does not contain $z^2\bar
z$-degree terms (as $M\in\mathcal P_0$),  we get
$\varphi_{32}(w)=\bar\varphi_{23}(w)$, which gives, using
\eqref{*defequat}, $A(w)=3i\bar F(w)$, as required. The proof in
the negative case is analogous.
\end{proof}

Propositions 5.3 and 5.4 prove statement (ii) of Theorem 3.3.

\smallskip

To prove statement (iii) we argue as in the proof of statement
(ii) of Theorem 3.4 and conclude that there exists a possibly
smaller associated neighbourhood $U$ such that each Segre variety
$Q_p,\,p\in\Delta^*_\delta\times\Delta^*_\epsilon$, is the graph
of an injective function $w_p(z)$, so that it can be also
represented as a graph $z=z_p(w)$. It is straightforward then to
recalculate the derivatives:
$$
w_z=\frac{1}{z_w},\ w_{zz}=\left(\frac{1}{z_w}\right)_w\cdot
w_z=-\frac{z_{ww}}{(z_w)^3}.
$$
Substituting these into \eqref{w''(z)} we obtain~\eqref{d2z}, so
that all the functions $z_p(w)$ satisfy~\eqref{d2z}.

The injectivity of the correspondence $M\lr \mathcal E(M)$ follows
from statement (ii). This completely proves the theorem.

\section{Equation $\mathcal E(M)$ and the analytic continuation}

The main conclusion that can be drawn from the results of the
previous section is that we can associate with a hypersurface
$M\subset\CC{2}$ of class $\mathcal P_0$ the complex differential
equation $\mathcal E(M)$, given by \eqref{d2z} and satisfying the
relations \eqref{ODErelations}, in such a way that the Segre
varieties of $M$ are open domains on the graphs of solutions of
the equation $\mathcal E(M)$. In particular, statements (iii) and
(iv) of Theorem 3.4 admit the following ODE-interpretation:\it

\smallskip

All solutions in the annulus $\Delta^*_\epsilon$ of the equation
$\mathcal E(M)$ exist as globally defined, possibly
multiple-valued, analytic mappings
$h:\,\Delta^*_\epsilon\longrightarrow\CP{1}$. Furthermore:

\smallskip

\noindent$\mbox{(iii)}'$ The analytic mapping $\mathcal
F:\,U\setminus\{w=0\}\longrightarrow\CP{2}$ associated with $M$ is
single-valued if and only if all the solutions of the equation
$\mathcal E(M)$ are single-valued mappings
$\Delta^*_\epsilon\lr\CP{1}$.

\smallskip

\noindent$\mbox{(iv)}'$ The analytic mapping $\mathcal
F:\,U\setminus\{w=0\}\longrightarrow\CP{2}$ associated with $M$
extends to the complex locus $\{w=0\}$ holomorphically if and only
if all local solutions of the equation $\mathcal E(M)$ extend
meromorphically to $\Delta_\epsilon$. \rm

\smallskip

Statements (iii)' and (iv)' now give a hint on how to prove
Theorem 3: we need to show the moderate growth of solutions of the
ODE $\mathcal E(M)$ as $w\longrightarrow 0$. This allows us to
reduce Theorem 3 to a question that can be formulated purely in
terms of analytic theory of differential equations. Realization of
this strategy is the content of Sections~6 and~7.

\subsection{Fuchsian and non-Fuchsian hypersurfaces.}
Equation $\mathcal E(M)$ obtained in Section~5 is an ordinary
second order meromorphic differential equation defined in the
domain $\Delta_\epsilon\subset\CC{}$. $\mathcal E(M)$ is
polynomial w.r.t. the unknown function $z$ and its derivative
$z'$, and has in $\Delta_\epsilon$ a unique (and hence isolated)
meromorphic singularity at the point $w=0$. The study of this type
of equations was initiated by Poincar\'e and Painlev\'e
(see~\cite{poincare}, \cite{golubev}, \cite{ai}, \cite{vazow}),
and it continues to be an active area of research (see, for
example,~\cite{ilyashenko},\cite{bolibruh},
\cite{laine},\cite{gromak},\cite{iwasaki} and references therein).
In his celebrated work \cite{painleve} Painlev\'e classified
second order complex ODEs, rational in the dependent variable $z$
and its derivative, meromorphic in some domain $\Omega$ in the
independent variable $w$, and having no movable critical points
(ODEs of this type are called \it ODEs of class $\mathcal P$). \rm
The mapping, bringing an ODE of class $\mathcal P$ to its standard
form in this classification, is locally biholomorphic in
$\CP{1}\times\Omega$ and is linear-fractional in the dependent
variable (see, e.g., \cite{ai}). Note that the associated mapping
$\mathcal F$, considered in the present paper, has the above
described form and brings the associated ODE $\mathcal E(M)$ to
its standard form $z''=0$. Thus real hypersurfaces, considered in
the paper, are associated with ODEs of class $\mathcal P$ with the
simplest standard form $z''=0$. This explains the $\mathcal
P_0$-notation for them.

As explained in Section 2, in the particularly important \it
linear \rm case the behaviour of solutions for the ODE $\mathcal
E(M)$ is characterized by the Fuchsian condition. The Fuchsian
type condition for a hypersurface $M\in\mathcal P_0$, described in
Introduction, can be stated in terms of the associated equation
$\mathcal E(M)$ and is imposed by a similarity with the linear
case. To show that, we first observe that \it a hypersurface
$M\in\mathcal P_0$ satisfies the Fuchsian type condition if and
only if the associated equation $\mathcal E(M)$ satisfies
\begin{equation}\label{fuchsianODE}
\mbox{ord}_0 B(w)\geq m-1,\,\mbox{ord}_0 E(w)\geq
2m-2,\,\mbox{ord}_0 A(w) =\mbox{ord}_0 F(w)\geq \frac{3}{2}(m-1).
\end{equation} \rm
The formulated statement follows directly from formulas
\eqref{ODEviadef}. Here for a nonzero function $h(w) \in \mathcal
M(0)$ we denote by $\mbox{ord}_0 h$ the order of vanishing of $h$
if it is holomorphic at $0$, and the negative order of pole for
$h$ otherwise.  \rm

Next we investigate the Fuchsian type condition. For that we
introduce an alternative to~\eqref{fuchsianODE} description.

\begin{dfn} A hypersurface $M\in\mathcal P_0$ is called \it $l$-reducible,\,$l\in\mathbb Z$, if
the change of variables $Z=zw^l,\,W=w$ brings the associated ODE
$\mathcal E(M)$ to an ODE of the form
\begin{equation}\label{lreducedODE}
Z''=\frac{1}{W}(\hat A Z+\hat B)Z'+\frac{1}{W^2}(\hat CZ^3+\hat
DZ^2+\hat EZ+\hat F)
\end{equation}
for some holomorphic near the origin functions $\hat A(W),\hat
B(W),\hat C(W),\hat D(W),\hat E(W),\hat F(W)$.
\end{dfn}

The $l$-reducibility condition turns out to be equivalent to the
Fuchsian type. In particular, it is a biholomorphic invariant of
$M$.

\begin{propos} $\mbox{}$

\smallskip

(1) A hypersurface $M\in\mathcal P_0$ is of Fuchsian type if and
only if the associated ODE $\mathcal E(M)$ is $l$-reducible for
some $l\geq 0$. Moreover, $l$ can be chosen in such a way that the
polynomial $\hat C(0)t^3+\hat D(0)t^2+\hat E(0)t+\hat F(0)$ is not
a nonzero constant.

\smallskip

(2) The Fuchsian type condition for a nonminimal hypersurface
$M\subset\CC{2}$, spherical in the complement to the complex
locus, is biholomorphically invariant. In particular, this
condition is independent of the choice of prenormal coordinates.
\end{propos}

\begin{proof}
(1) Suppose first that $F(w)\equiv 0$ in $\mathcal E(M)$. It
follows from \eqref{ODErelations} that $A=C=D=F\equiv 0$, and the
equation $\mathcal E(M)$ is linear. In this case it can be seen
immediately that the Fuchsian type condition is equivalent to
$\mathcal E(M)$ being Fuchsian in the sense of theory of linear
ODEs, which means $0$-reducibility. Moreover, the polynomial $\hat
C(0)t^3+\hat D(0)t^2+\hat E(0)t+\hat F(0)$ has a root $t_0=0$,
which proves the proposition under the assumption $F(w)\equiv 0$.

Consider now the case when $F\not\equiv 0$. Suppose first that $M$
is $l$-reducible for some $l\in\mathbb Z$. Perform in the equation
$\mathcal E(M)$ associated with the hypersurface $M\in\mathcal
P_0$ the change of variables $Z=zw^l,\,W=w$, and rewrite the new
equation in the form $Z''=(p_1Z+p_0)Z'+(q_3Z^3+q_2Z^2+q_1Z+q_0)$
for certain $p_i,q_j\in\mathcal M(0)$. Then, by recalculating the
derivatives and substituting them into $\mathcal E(M)$, it is not
difficult to check that the properties $\mbox{ord}_0p_0\geq -1$
and $\mbox{ord}_0q_1\geq -2$ hold simultaneously if and only if
the terms $\frac{B(w)}{w^m}$ and $\frac{E(w)}{w^{2m}}$ have the
same properties simultaneously, so that from $l$-reducibility we
have $\mbox{ord}_0 B\geq m-1,\,\mbox{ord}_0 E\geq 2m-2$. Also we
compute that $\mbox{ord}_0 q_0=\mbox{ord}_0 F+l-2m$. From the
$l$-reducibility, $\mbox{ord}_0 q_0=-2+s$ for some integer $s\geq
0$, and thus $l=2m-2+s-\mbox{ord}_0 F$. From \eqref{*defequat} we
have $\mbox{ord}_0 A=\mbox{ord}_0 F,\,\mbox{ord}_0 C=2\mbox{ord}_0
F$, so that, after a computation, $\mbox{ord}_0
p_1=2\mbox{ord}_0F-3m-s+2$. From the $l$-reducibility now
$2\mbox{ord}_0F-3m-s+2\geq -1$, and we obtain $2\mbox{ord}_0 F\geq
s+3(m-1)\geq 3(m-1)$, as required for the Fuchsian type.

Suppose now that $M$ is of Fuchsian type. Put $l:=\mbox{ord}_0
F-m+1$. Now arguing as above and using $\mbox{ord}_0
\frac{B(w)}{w^m}\geq-1,\,\mbox{ord}_0
\frac{E(w)}{w^{2m}}\geq-2,\,\mbox{ord}_0 A=\mbox{ord}_0
F,\,\mbox{ord}_0 C=2\mbox{ord}_0 F,$ we get $\mbox{ord}_0 p_0\geq
-1,\,\mbox{ord}_0 q_1\geq -2,\,\mbox{ord}_0 q_0=\mbox{ord}_0
F+l-2m\geq -2,\,\mbox{ord}_0 p_1= \mbox{ord}_0 A - l - m =
-1,\,\mbox{ord}_0 q_3= \mbox{ord}_0 C-2l-2m= -2,\,\mbox{ord}_0
q_2\geq -2$, so that we obtain an equation of the form, required
for $l$-reducibility. The integer $l$ here is equal to
$\mbox{ord}_0 F-m+1\geq\frac{m-1}{2}$ and thus is nonnegative. To
check that the polynomial $\hat C(0)t^3+\hat D(0)t^2+\hat
E(0)t+\hat F(0)$ is not a constant, we note that for the latter
choice of $l$ we have $\mbox{ord}_0 q_3=-2$, so that $\hat
C(0)\neq 0$. This finally proves (1).

In order to prove (2) we consider two hypersurfaces $M,\tilde
M\in\mathcal P_0$ and a local biholomorphism $G:\,(M,0)\lr (\tilde
M,0)$ between them. Suppose that $M$ is of Fuchsian type. Then,
according to (1), the transformation $H:\,(z,w)\lr (zw^l,w)$ for
an appropriate integer $l\geq 0$ brings $\mathcal E(M)$ into an
ODE of the form \eqref{lreducedODE}. Hence the transformation
$H\circ G^{-1}$, which has the form
$(f(z,w)w^l+O(|z|^2|w|^l)+O(|w|^{l+1}),g(z,w))$ for an appropriate
local biholomorphism $(f,g):\,(\CC{2},0)\lr (\CC{2},0)$, brings
$\mathcal E(\tilde M)$ into an ODE of the form
\eqref{lreducedODE}. Arguing now similarly to the proof of (1) we
deduce from here that $\tilde M$ is of Fuchsian type, which proves
statement (2) and the proposition.
\end{proof}

\begin{dfn}
Let $M\in \mathcal P_0$ be an $m$-nonminimal hypersurface of
Fuchsian type. The ODE
\begin{equation*}
Z''=\frac{1}{W}(\hat A Z+\hat B)Z'+\frac{1}{W^2}(\hat CZ^3+\hat
DZ^2+\hat EZ+\hat F),
\end{equation*}
obtained from $\mathcal E(M)$ by the change of variables
$Z=zw^l,\,W=w$ with $l:=\mbox{ord}_0 F-m+1\geq 0$ (as in the proof
of Proposition 6.2), is called \it the associated ODE $\mathcal
E^r(M)$.
\end{dfn}

According to Proposition 6.2, the associated ODE $\mathcal E^r(M)$
always exists in the Fuchsian type case, and the polynomial $\hat
C(0)t^3+\hat D(0)t^2+\hat E(0)t+\hat F(0)$ is not a nonzero
constant.\rm

\subsection{Hypersurfaces with rotational symmetries. Examples.}

The associated ODE $\mathcal E(M)$ is particularly simple in the
special case when a hypersurface $M\in\mathcal P_0$ is invariant
under the group $(z,w) \lr (e^{it}z,\,w)$, $t\in\RR{}$, of
rotational symmetries. As each above rotational symmetry sends a
Segre variety of $M$ into another Segre variety, it must be a
symmetry of the ODE $\mathcal E(M)$, and it is not difficult to
see that the associated ODE $\mathcal E(M)$ is linear in the
rotational case. \rm Thus we conclude that \it Theorems 2 and 3.5
follow from the Fuchs theorem in the rotational case. \rm This
also shows that the regularity condition in Theorem~3 (namely, the
Fuchsian type condition) is optimal in the rotational case.

\begin{rema} As follows from the described connection between rotational
hypersurfaces of class~$\mathcal P_0$, Theorem~3.15 in
\cite{divergence} and Theorem 3.3 of the present paper, \it the
algorithm for obtaining nonminimal spherical hypersurfaces with
rotational symmetries, described in Remark 3.18 in
\cite{divergence}, gives a complete description of hypersurfaces
of class $\mathcal P_0$ with rotational symmetries. \rm
\end{rema}

However, as the example of hypersurfaces $M_{R,0}$ in \cite{kl}
shows, the investigation of nonminimal spherical hypersurfaces in
$\CC{2}$ {\it cannot} be reduced to the rotational case. Below we
demonstrate applications of Theorems 1 and 2 (or, alternatively,
Theorems 3.4 and 3.6) and give explicit examples of the associated
ODE construction in the rotational case.

\begin{ex}
The 1-nonminimal hypersurfaces $L_s,\,s\in\RR{},s\neq 0$, with the
complex locus $\{w=0\}$, given by
$$v=u\tan\left(\frac{1}{s}\ln(1+s|z|^2)\right),$$ were obtained in
\cite{belnew} as examples of nonminimal hypersurfaces with
$4$-dimensional infinitesimal automorphism algebras (see also
\cite{kl}). It is not difficult to check that each $L_s$ is of
class $\mathcal P_0$. Indeed, one has to check only the sphericity
of $L_s$ at Levi nondegenerate points, and this follows from the
fact that only spherical hypersurfaces admit $\geq 4$ dimensional
infinitesimal automorphism algebras at Levi nondegenerate
points~\cite{belold}. The complex defining equation of $L_s$ has
the form $w=\bar w \exp\left(\frac{2i}{s}\ln(1+sz\bar z)\right)$.
For a point $(a,b)\in\CC{2}$ with $a,b\neq 0$ its Segre variety
$Q_{(a,b)}$ equals (locally)
$$
z(w)=h_{(a,b)}(w)=\frac{1}{s\bar a}\left(\frac{w}{\bar
b}\right)^{\frac{s}{2i}}-\frac{1}{s\bar a}.
$$
Clearly, for any $s\in\RR{},a,b\in\CC{},s,a,b\neq 0$, the germ
$h_{(a,b)}(w)$ does not extend to the origin meromorphically, so
by Theorem 2 the associated mapping $\mathcal F$ \it does not \rm
extend to the complex locus holomorphically.
\end{ex}

The next example illustrates in detail the connection between a
family of hypersurfaces $M_\gamma\in \mathcal P_0$, the associated
ODEs $\mathcal E(M_\gamma)$, and the associated mappings $\mathcal
F_\gamma$.

\begin{ex} For the 1-nonminimal hypersurfaces $M_{\gamma}\subset\CC{2},\,\gamma\in\mathbb R\setminus\{0\}$,
containing the complex hypersurface $X=\{w=0\}$ and given in a
neighbourhood of the origin by
$$
w=\bar w \left(i|z|^2+\sqrt{1-|z|^4}\right)^{\frac{1}{\gamma}}
$$
(see \cite{kl}), the family of Segre varieties near the origin has
the form $Q_{(a,b)}=\{w=\bar b(iz\bar a+\sqrt{1-z^2\bar
a^2})^{\frac{1}{\gamma}}\}$. Elementary computations show that
$Q_{(a,b)}$ with $a,b\neq 0$ are open domains on the graphs
\begin{equation} \label{mgamma}
\tilde Q_{(a,b)}=\left\{z=\frac{1}{2i\bar
a}\left(\frac{w^\gamma}{\bar b^\gamma}- \frac{\bar
b^\gamma}{w^\gamma}\right) \right\}.
\end{equation}
By Theorem 2, the associated mapping $\mathcal F_\gamma$ extends
to the complex locus holomorphically if and only if
$\gamma\in\mathbb{Z}$. In fact one can see that $F_\gamma$ is
given by $z\lr zw^{\gamma},\,w\lr w^{2\gamma}.$

Following the elimination process described in Section~2 it is not
difficult to conclude that all the graphs $\tilde
Q_{(a,b)},\,a,b\neq 0$ satisfy the linear ODE
$$z''=-\frac{1}{w}z'+\frac{\gamma^2}{w^2}z,$$ which coincides, by uniqueness, with $\mathcal
E(M_\gamma)$. This ODE is Fuchsian for any $\gamma\in\RR{}$.
\end{ex}

The next two examples show that for $m>1$ the ODE $\mathcal E(M)$
associated with a hypersurface of class $\mathcal P_0$ may be both
of Fuchsian and non-Fuchsian type.

\begin{ex} Consider the $m$-nonminimal with $m\geq 2$ hypersurfaces $M^m_{0}\in\mathcal P_0$ (see \cite{divergence}),
given near the origin by the complex defining equations
\begin{equation}\label{Mm0} w=\bar w\left(1+\frac{i}{2}(1-m)\bar
w^{m-1}\ln\frac{1}{1-2|z|^2}\right)^{\frac{1}{1-m}}.
\end{equation} The Levi nondegenerate part of $M^m_0$ is the preimage of a domain in the quadric $\mathcal Q=\left\{2|Z|^2+|W|^2=1
\right\}\subset\CC{2}$ under the single-valued mapping
$$\Lambda_m:\, (Z,W) = \left(z,\, e^{\frac{2i}{1-m}w^{1-m}}
\right).$$ It follows that the mapping $\Lambda_m$ is associated
with $M^m_0$. Remarkably, \it each mapping $\Lambda_m$ does not
extend to the complex locus $\{w=0\}$, even though it is
single-valued. \rm From the elimination procedure from Section 2
(or the arguments from \cite{divergence}), we conclude that the
associated ODE $\mathcal E(M^m_0)$ is of non-Fuchsian form
$z''=\left(\frac{2i}{w^m}-\frac{m}{w}\right)z'$. This agrees with
Theorem~3.
\end{ex}

\begin{ex} For the 2-nonminimal hypersurface $M\in\mathcal P_0$, given by $v=(u^2+v^2)|z|^2$, it is not difficult to see
that the polynomial mapping $\mathcal F(z,w)= (zw,w)$ maps $M$
into the hyperquadric $\{\im w=|z|^2\}\subset\CC{2}$. The
associated ODE $z''=-\frac{2}{w}z'$ is Fuchsian. \end{ex}

\begin{rema} As the family of hypersurfaces $M^m_\beta\in\mathcal P_0$ in
\cite{divergence} shows, the associated mapping $\mathcal F$
cannot be in general expressed in terms of elementary functions
when $m>1$, even though the associated ODE is given by elementary
functions. In this case the extension/no extension dichotomy can
be resolved \it only \rm using the associated equation $\mathcal
E(M)$ and Theorem 3.
\end{rema}

\subsection{Reduction of Theorem 3 to the existence of a holomorphic solution}

In this subsection we perform an important step toward the proof
of sufficiency in Theorem 3, reducing it to Theorem~3.5, i.e., the
question that can be formulated purely in terms of analytic theory
of differential equations.

\begin{propos} Suppose that an $m$-nonminimal hypersurface $M\in\mathcal P_0$
is of Fuchsian type and the associated mapping $\mathcal F$ is
single-valued. Suppose, in addition, that the associated equation
$\mathcal E^r(M)$ admits a holomorphic at the origin solution
$z=h(w)$. Then $\mathcal F$ extends to the complex locus
$X=\{w=0\}$ holomorphically.\end{propos}

\begin{proof} We choose $l:=\mbox{ord}_0 F-m+1\geq 0$ as in the definition
of the ODE $\mathcal E^r(M)$, and reduce the ODE $\mathcal E(M)$
to the ODE $\mathcal E^r(M)$ by the change of variables
$Z=zw^l,\,W=w$. Using Theorem~3.4, we represent all solutions of
the equation $\mathcal E(M)$ in the form \eqref{hp} with
single-valued $\alpha_0(w),...,\beta_2(w)$. We introduce a locally
biholomorphic mapping $\widehat{\mathcal
F}:\,\CC{1}\times\Delta_\epsilon^*\longrightarrow\CP{2}$ given by
$$
(Z,W)\longrightarrow(\hat\alpha_0(W)Z+\hat\beta_0(W),\hat\alpha_1(W)Z+
\hat\beta_1(W),\hat\alpha_2(W)Z+\hat\beta_2(W)),
$$
where the single-valued functions $\hat \alpha_j,\hat\beta_j$ are
defined as $\hat
\alpha_j:=\frac{1}{w^l}\alpha_j,\,\hat\beta_j:=\frac{1}{w^l}\beta_j$.
According to Theorem 3.4, it is sufficient to prove that the
collection of functions $\alpha_j,\beta_j$ can be scaled to belong
to the class $M(0)$. Obviously, it is sufficient to prove the same
fact for the collection $\hat \alpha_j,\hat\beta_j$.

Since $Z=h(W)$ is a solution of the ODE $\mathcal E^r(M)$, the
mapping $\widehat{\mathcal F}$ sends $\{Z=h(W)\}$ into some
projective hyperplane in $\CP{2}$. We then compose
$\widehat{\mathcal F}$ with an element of $\mbox{Aut}(\CP{2})$ in
such a way that $\{Z=h(W)\}$ is mapped into
$\CP{2}\setminus\CC{2}$. Using the representation of type
\eqref{fgaalpha} for the mapping $\widehat{\mathcal F}$ with
appropriate functions $\hat\alpha(W),\hat a(W),\hat \beta(W),\hat
b(W),\hat\delta(W)$, we conclude that $\hat\delta(W)=-h(W)\in
\mathcal O(0)$. Arguments similar to those in the proof of Theorem
3.3 show that the fact that $\widehat{\mathcal F}$ transforms the
ODE $\mathcal E$ into $(Z^*)''=0$ yields formulas identical to
\eqref{i2},\eqref{i3} in terms of $\hat\alpha(W),\hat a(W),\hat
\beta(W),\hat b(W),\hat\delta(W)$. Set $\hat k(W):=\frac{\hat
a(W)}{\hat\alpha(W)}$. Then
\begin{equation}\label{formulas}
\hat b'=\hat k\hat
\beta',\,\frac{\hat\beta'}{\hat\alpha}=\frac{\hat A(W)}{3w},\,\hat
a'\hat\alpha-\hat\alpha'\hat a=\hat k'\hat\alpha^2,\,\frac{(\hat
k'\hat\alpha^2)'}{\hat k'\hat\alpha^2}=\frac{\hat
A(W)}{W}\hat\delta-\frac{\hat B(W)}{W}.
\end{equation}
Formulas~\eqref{formulas} show that if $\hat\alpha\in M(0)$, then
$\hat\beta,\hat k,\hat a,\hat b\in M(0)$. The reason is that if a
meromorphic in a punctured disc $\Delta^*_\epsilon (0)$ function
$u(W)$ satisfies $\frac{Wu'}{u}\in \mathcal O(0)$, then $u\in
M(0)$.

To verify the fact $\hat\alpha\in M(0)$, we continue a detailed
expansion of \eqref{i3}, using \eqref{formulas}, in terms of
$\hat\alpha,\hat k$. Then a computation shows that $$-\frac{\hat
E(W)}{W^2}=\left(\frac{\hat B(W)}{W}-\frac{\hat
A(W)}{W}\hat\delta\right)\frac{\hat\alpha'}{\hat\alpha}+\frac{\hat\alpha''}{\hat\alpha}-\hat\delta'\frac{\hat
A(W)}{W}-2\delta\frac{\hat D(W)}{W^2}-3\delta\frac{\hat
C(W)}{W^2}.$$ The obtained equality can be considered as a second
order Fuchsian ODE with the unknown function $\hat\alpha(W)$. By
the Fuchs theorem  we conclude that $\hat\alpha(W)\in M(0)$, which
proves $\hat\alpha,\hat\beta,\hat a,\hat b,\hat\delta\in M(0)$.
Hence, the collection $\hat\alpha_j,\hat\beta_j$ can be scaled to
become holomorphic at $W=0$, as required.
\end{proof}

\section{Existence of a holomorphic solution}

By the results of the previous section, in order to prove
Theorem~2 we need to show that equation $\mathcal E^r(M)$
associated with an $m$-nonminimal Fuchsian type hypersurface
$M\in\mathcal P$ admits a holomorphic at the origin solution
$z=h(w)$, provided its solutions are single-valued. In this
section we prove a more general fact (Theorem 3.5), stating that
\it any \rm ODE similar to $\mathcal E^r(M)$ must have at least
one holomorphic at the origin solution, provided that no solution
can branch about the origin.

Section 6.1 shows that if the ODE $\mathcal E^r(M)$ associated
with a Fuchsian type hypersurface $M\in\mathcal P$ is such  that
the associated mapping $\mathcal F$ is single-valued, then it
satisfies the conditions of Theorem~3.5. To see this it is enough
to choose $z_0$ as a root of the polynomial $\hat C(0)t^3+\hat
D(0)t^2+\hat E(0)t+\hat F(0)$). Hence, Theorem~3.5  implies
Theorems~3.6 and~3.

The idea of the proof of Theorem~3.5 is as follows: The result is
trivial if the function $Q(z,w)$ is independent of $w$ because we
may simply take  $z(w):=z_0$ as a holomorphic solution. For the
general case we apply the Poincar\'e Small Parameter Method.
Further, thanks to the convergence result in \cite{gromak} (see
Theorem~A.12 there), in order to prove Theorem~3.5 it is
sufficient to prove the existence of a {\it formal} holomorphic
solution for the equation $\mathcal E$, as any such solution is
automatically convergent, without any assumption on the
eigenvalues of the linearization matrix. We note that the
convergence result can be also proved  using the standard
technique of majorizing functions, but we do not provide the proof
here. By a {\it formal holomorphic solution} for the equation
$\mathcal E$ we mean a formal power series
$z(w)=\sum\limits_{r=0}^\infty a_rw^r$, that makes $\mathcal E$ an
identity of two Laurent series in $w$ (with finite principal
parts).

After a simple substitution $z\lr z-z_0$ we may  assume $z_0=0$.
Thus, for the proof of Theorem~3.5 it remains to prove the
following

\begin{thm}\label{p.6.3}
In the assumptions of Theorem~3.5 with $z_0=0$, the equation
$\mathcal E$ admits a formal solution
$z(w)=\sum\limits_{r=1}^\infty a_rw^r$.
\end{thm}

\begin{proof}

We represent equation $\mathcal E$ as a system by introducing a
new unknown function
$$u(w):=wz'(w).$$ Then we have $z'=\frac{u}{w}$ and
$z''=\frac{u'}{w}-\frac{u}{w^2}$, so that $\mathcal E$ becomes the
system
\begin{equation}\label{mainsystem}
\begin{cases}
z'=\frac{u}{w} ,\\
u'=\frac{1}{w}\left[(1+P(z,w))u+Q(z,w)\right].
\end{cases}
\end{equation}

Recall that we assume $z_0=0$, so that $Q(0,0)=0$. Clearly, the
existence of the desired solution is equivalent to the existence
of a formal holomorphic solution $z(w)=\sum\limits_{r=1}^\infty
a_rw^r,\,u(w)=\sum\limits_{r=1}^\infty b_rw^r$ for the
system~\eqref{mainsystem}. We expand the functions $1+P(z,w)$ and
$Q(z,w)$ as $1+P(z,w)= \sum\limits_{k,j\geq
0}p_{kj}z^kw^j,\,Q(z,w)=q_{10}z+q_{01}w+
\sum\limits_{k,j>0}q_{kj}z^kw^j$. Plugging all the power series
representations into~\eqref{mainsystem} and gathering terms with
$w^{r-1},\,r\geq 1$, we obtain

\begin{eqnarray}\label{gather1}
&  a_1-b_1=0 , \\
& \notag b_1-p_{00}b_1-q_{10}a_1=q_{01} ,
\end{eqnarray}
for $r=1$, and
\begin{eqnarray}\label{gatherr}
& \notag ra_r-b_r=0 \\
& rb_r-p_{00}b_r-q_{10}a_r=\sum\limits_{2\leq k+j\leq
r}q_{kj}\sum\limits_{i_1+...+i_k=r-j}a_{i_1}\cdot...\cdot
a_{i_k}+\\
& \notag +\sum\limits_{l=1}^{r-1}b_l\sum\limits_{1\leq k+j\leq
r-l}p_{kj}\sum\limits_{i_1+...+i_k=r-j-l}a_{i_1}\cdot...\cdot
a_{i_k} ,
\end{eqnarray}
for $r>1$. It is presumed in  \eqref{gatherr} that a sum of the
form $\sum a_{i_1}\cdot...\cdot a_{i_k}$ equals $1$ for $k=0$.  It
is also \it important \rm that for a fixed $r$ on the left-hand
side, the right-hand side in both \eqref{gather1} and
\eqref{gatherr} contains only $a_i,b_l$ with $i,l<r$.

Now let us introduce some vector and matrix notation. We denote by
$h_r\in\CC{2}$ the vector with components $a_r,b_r$, and by $L$
the $2\times 2$ matrix $\begin{pmatrix} 0 & 1 \\ q_{10} & p_{00}
\end{pmatrix}$.
 Then,
if $I$ denotes the identity matrix, the equations
\eqref{gather1},\eqref{gatherr} can be rewritten for all $r\geq 1$
as:

\begin{equation}\label{matrixequation}
(rI-L)h_r=\begin{pmatrix} 0  \\
K_r
\end{pmatrix},
\end{equation}
where  $K_1=q_{01}$,  and for $r\geq 2$,
$$
K_r(a_1,...,a_{r-1},b_1,...,b_{r-1},\{p_{kj}\}_{1\leq k+j\leq
r-1},\{q_{kj}\}_{2\leq k+j\leq r})
$$
is a polynomial scalar expression from the right-hand side of
\eqref{gatherr}. It is crucial that all polynomials $K_r$ have
nonnegative coefficients. We now consider two cases.

\medskip

{\bf Nonresonant case.} We assume that $L$ does not have any
eigenvalues $r\in\mathbb{Z}^+$. In this case each of the
equations~\eqref{matrixequation} has a unique solution $h_r$, if
$h_1,...,h_{r-1}$ are already found, and this determines the
collection $\{h_r\}_{r\geq 1}$ uniquely. We then put
\begin{equation}\label{e.formals}
\begin{pmatrix} z^*  \\ u^* \end{pmatrix}(w):=
\sum\limits_{r=1}^\infty h_r w^r,
\end{equation}
and $(z^*(w),u^*(w))$ becomes a formal holomorphic solution of the
equation the system~\eqref{mainsystem} by construction. This
proves the theorem in the nonresonant case.

\medskip

{\bf Resonant case.} This case turns out to be much more delicate
and requires additional considerations. We will prove the
existence of a collection $\{h_r\}_{r\geq 1}$,
satisfying~\eqref{matrixequation}, which will imply the existence
of a formal holomorphic solution~\eqref{e.formals}.  Our main
strategy is to show that the absence of a solution for the system
of equations~\eqref{matrixequation} leads to multiple-valuedness
of certain solutions of $\mathcal E$, which contradicts the
assumption of Theorem 3.5. In order to do that, we consider the
case of a general equation $\mathcal E$ as a perturbation of the
above "constant coefficient" case $Q=Q(z)$, by introducing a small
parameter $\varepsilon$. Perform in the system~\eqref{mainsystem}
the change of variables $w=\varepsilon
w^*,\,z=z^*,\,0<|\varepsilon|<1,\,\varepsilon\in\CC{}$. In the new
coordinates the system becomes
\begin{equation}\label{epsilonsystem}
\mathcal S_\varepsilon = \begin{cases}
z'=\frac{u}{w},\\
u'=\frac{1}{w}\left[(1+P(z,\varepsilon w))u+Q(z,\varepsilon
w)\right].
\end{cases}\end{equation}
Although for the change of variables we have $\varepsilon\neq 0$,
we may extend~\eqref{epsilonsystem} holomorphically to
$\{|\varepsilon|<1\}$. Thus we get a holomorphic in the unit disc
family $\mathcal S_\varepsilon$ of first-order systems. Each
$\mathcal S_\varepsilon$ is a holomorphic perturbation of the
system $\mathcal S_0$, that has the holomorphic solution
$z=0,u=0$. So the strategy now is to find {\it analytic} solutions
of $\mathcal S_\varepsilon$ in annuli $\{r_1<|w|<r_2\},
\,0<r_1<r_2<\epsilon$ for sufficiently small $\varepsilon$ as
perturbations of the constant solution for $\mathcal E_0$. This
general approach is known as the {\it Small Parameter Method}. It
was invented by H.~Poincar\'e to investigate solutions of
nonlinear systems considering them as perturbations of already
known solutions of initial "simple" systems. In the modern
language, the method simply uses the analytic dependence of
solutions of a system of first-order holomorphic ODEs on the
initial conditions and holomorphic parameters, see
\cite{ilyashenko}. We give below a convenient formulation of this

\begin{thm}[Poincar\'e, 1892, see, e.g., \cite{golubev}.]
Let $F(x,y,\varepsilon)$, $x\in\CC{}$, $y\in\CC{2}$,
$\varepsilon\in\CC{}$, be a holomorphic function in the domain
$D\times G\times E$, $x_0\in D$ is a fixed point and
$\gamma(t),0\leq t\leq 1$, is a smooth real-analytic path with
$\gamma(t)\subset D$ and $\gamma(0)=x_0$. Suppose that $0\in E$
and the ODE system $y'=F(x,y,0)$ has a holomorphic solution
$y_0(x)$ in a neighborhood $U$ of $[\gamma(t)]$ with
$y_0(x_0)=p_0$. Then for any sequence $p_r\in \CC{2}$, $r\geq 1$,
such that the power series $\sum p_r\varepsilon^r$ is convergent
in some disc, and any sufficiently small $\varepsilon$, the ODE
system $y'=F(x,y,\varepsilon)$ has a holomorphic w.r.t. the time
$t$ on $\gamma$ solution of the form
\begin{equation}\label{e.6.5}
y^\varepsilon(\gamma(t))=\sum\limits_{r=0}^\infty
y_r(t)\varepsilon^r,
\end{equation}
where $y_r(t)$, $r\geq 1$, are analytic on $[0,1]$, with
$y_r(0)=p_r,\,r\geq 0$, and the series~\eqref{e.6.5} is uniformly
convergent w.r.t. $t$ and $\varepsilon$. Each of the $y_r(t)$
extends to an open neighbourhood $\tilde U$ of $[\gamma]$ as a
(possibly multiple-valued) analytic function $y_r(x)$ such that
$y^\varepsilon(x)=\sum\limits_{r=0}^\infty y_r(x)\varepsilon^r$ is
a (possibly multiple-valued) solution of $y'=F(x,y,\varepsilon)$.
Moreover, each $y_r(x)$, $r\geq 1$, is a solution of some
first-order inhomogeneous linear system of ODEs with homogeneous
part independent of $r$.
\end{thm}

We proceed now with Poincar\'e's Small Parameter Method. We
suppose, without loss of generality, $\epsilon>1$ (where
$\{0<|w|<\epsilon\}$ is the punctured disc where $\mathcal E$ is
defined) and let $\gamma$ be the unit circle and $w_0=1\in\gamma$
be the starting point in Poincar\'e's theorem. We expand
$$
z^\varepsilon (w)=\sum\limits_{r=1}^\infty z_r(w)\varepsilon^r,\ \
u^\varepsilon (w)=\sum\limits_{r=1}^\infty u_r(w)\varepsilon^r.
$$
We now substitute the expansions for $z^\varepsilon (w)$,
$1+P(z,w)$, and $Q(z,w)$ into~\eqref{epsilonsystem} and collect
terms with $\varepsilon^r$, $\,r\geq 1$. For $r=1$ we obtain the
following inhomogeneous first-order linear ODE system in
$z_1,u_1$:
\begin{equation*}\begin{cases}
z_1'=\frac{u_1}{w} ,\\
u'_1=\frac{1}{w}(p_{00}u_1+q_{10}z_1)+q_{01},
\end{cases}\end{equation*}
which can be rewritten as
\begin{equation}\label{e.6.6}
\begin{pmatrix} z_1'  \\ u_1'\end{pmatrix}=\frac{1}{w}L \begin{pmatrix} z_1  \\ u_1\end{pmatrix}+
\begin{pmatrix} 0  \\ K_1
\end{pmatrix},
\end{equation} where $L,K_1$ are as in~\eqref{matrixequation}.

\begin{dfn}
By a \it logarithmic quasipolynomial \rm we mean a (possibly
multiple-valued) analytic in $\CC{}\setminus\{0\}$ function
$P(w^{\lambda_1},...,w^{\lambda_s},\ln w)$, where
$s\in\mathbb{Z}_{\geq 0}$, $P$ is a complex polynomial in $s+1$
variables, and $\lambda_j\in\CC{}$.
\end{dfn}

We need now the following

\begin{lem}\label{l.6.4}
The eigenvalues of $L$ are two distinct integers.
\end{lem}

\begin{proof} Consider~\eqref{e.6.6} as a inhomogeneous Euler system (see \cite{ilyashenko}). The characteristic
roots of this system are the eigenvalues of $L$. Let
$\varphi(w),\psi(w)$ be two vector-functions, forming a basis of
the space of solutions for the homogeneous part of~\eqref{e.6.6}.
Suppose that the eigenvalues of $L$ coincide, or at least one of
them is not an integer. Then at least one of the two non-zero
vector-functions $\varphi(w),\psi(w)$ (say, $\varphi(w)$) contains
either a factor $w^\lambda$, $\lambda\notin\mathbb{Z}$, or a
factor $w^\lambda\ln w$, $\lambda\in\CC{}$, and hence is not
single-valued along $\gamma$. The general solution
of~\eqref{e.6.6} has the form:
\begin{equation}\label{e.6.7}
\begin{pmatrix} z_1\\ u_1\end{pmatrix}=c_1\varphi+\tilde c_1\psi+\theta_1,
\end{equation}
where $c_1,\tilde c_1$ are constants and $\theta_1$ is a
vector-function with components being logarithmic quasipolynomials
(the latter fact follows from the variation of constants
algorithm, applied to the Euler system, see \cite{ilyashenko}). We
may assume, without loss of generality, $\psi(1)\neq 0$ (otherwise
we replace $\gamma$ with a circle $\{|w|=R\}$ with $0<R<1$ and
$\psi(R)\neq 0$, and take $w_0=R$ as
a starting point). Choose in~\eqref{e.6.7} any $c_1,\tilde c_1$ with $c_1\neq 0$ and $\begin{pmatrix} z_1\\
u_1\end{pmatrix}(1)=0$. This fixes the term $\begin{pmatrix} z_1\\
u_1\end{pmatrix}\varepsilon$ in the expansion~\eqref{e.6.5} of the
solution.

We continue with the iteration process and collect terms with
$\varepsilon^r,\,r\geq 2$. We obtain the following series of
inhomogeneous Euler systems (with the homogeneous part identical
to that in~\eqref{e.6.6} for arbitrary $r\geq 2$):
\begin{equation}
\begin{pmatrix} z_r'  \\
u_r'\end{pmatrix}=\frac{1}{w}L\begin{pmatrix} z_r  \\
u_r\end{pmatrix}+M_r,
\end{equation}
where the components of the vector-function $M_r$ are logarithmic
quasipolynomials, depending on $M_j$ with $j<r$ (this again
follows by induction from the variation of constants algorithm).
The general solution has the form
\begin{equation}\label{e.6.9}
\begin{pmatrix} z_r\\ u_r\end{pmatrix}=c_r\varphi+\tilde c_r\psi+\theta_r,
\end{equation}
where the components of the vector-function $\theta_r$ are again
logarithmic quasipolynomials. We choose in~\eqref{e.6.9} any
$c_r,\tilde c_r$ with $\begin{pmatrix} z_r  \\
u_r\end{pmatrix}(1)=0$. Then, applying Poincar\'e's theorem, for
sufficiently small $\varepsilon$ we obtain a (possibly
multiple-valued) analytic in an open neighborhood of $\gamma$
solution of the system $\mathcal S_\varepsilon$, given by
$z(w:)=\sum\limits_{r=1}^\infty
z_r(w)\varepsilon^r,\,u(w):=\sum\limits_{r=1}^\infty
u_r(w)\varepsilon^r$. The uniform convergence in Poincar\'e's
theorem implies that this solution is not single-valued along
$\gamma$, because the first term in its expansion $\begin{pmatrix}
z_1\\ u_1\end{pmatrix}\varepsilon$ is not single-valued along
$\gamma$. As the system~\eqref{epsilonsystem} is obtained
from~\eqref{mainsystem} by scaling of the independent variable
$w$, we conclude that there exists a nonsingle-valued solution
for~\eqref{mainsystem} in some annulus. We get a contradiction
with the assumptions of Theorem~3.5, which proves the lemma.
\end{proof}

\noindent{\bf End of the proof of Theorem 7.1.} Let $k_1\geq 1$ be
the smallest positive eigenvalue of the matrix $L$ (which exists
by the assumption), and $k_2\neq k_1$ be the second eigenvalue
(not necessarily positive).

Suppose first $k_1=1$. Then we claim that $K_1=0$ in
\eqref{e.6.6}, and one can put $\begin{pmatrix} z_1\\
u_1\end{pmatrix}=0$. Indeed, the system \eqref{e.6.6} implies the
scalar inhomogeneous Euler equation
\begin{equation}\label{scalarEuler}z_1''=\frac{p_{00}+1}{w}z'_1+\frac{q_{10}}{w^2}z_1+\frac{q_{01}}{w},\end{equation}
for which the basic solutions of the homogeneous equation are some
single-valued rational functions of the form $\mbox{const}\cdot w$
and $\mbox{const}\cdot w^{k_2}$. Then it is straightforward to
check that the variation of constants gives a partial solution
containing two terms of the form $\mbox{const}\cdot w$ and
$\mbox{const}\cdot w \ln w$, and that the second term is non-zero
(and hence not single-valued) iff $K_1\neq 0$. Proceeding now as
in the proof of Lemma~\ref{l.6.4}, we see that the possibility
$K_1\neq 0$ contradicts the assumptions of Theorem~3.5, and so
$K_1$ must vanish. Hence, for $r=1$ in~\eqref{matrixequation} one
can simply put  $h_1:=0$.

If $k_2$ is not positive, we may repeat the proof of the
proposition in the nonresonant case, as there are no more
obstructions to solve equations~\eqref{matrixequation}. If $k_2$
is a positive integer, we return to Poincar\'e's Small Parameter
method and analyze it simultaneously with
system~\eqref{matrixequation}. As $K_1=0$, we put $\begin{pmatrix}
z_1
\\ u_1\end{pmatrix}=0$  and $h_1=0$
in~\eqref{matrixequation}. Then, using the expansions for
$z^\varepsilon (w)$, $u^\varepsilon (w)$, $P(z,\varepsilon
w),Q(z,\varepsilon w)$ and collecting terms with $\varepsilon^{r}$
for $r=2$ in~\eqref{matrixequation}, we have
\begin{equation}\label{e.6.10}
\begin{pmatrix} z_2'  \\ u_2'\end{pmatrix}=\frac{1}{w}L\begin{pmatrix} z_2  \\
u_2\end{pmatrix}+\begin{pmatrix} 0  \\ K_2 \end{pmatrix}\cdot w,
\end{equation}
where $L,K_2$ are as in~\eqref{matrixequation} (more precisely we
substitute the values $a_1=0$ and $b_1=0$, found in the previous
step, into $K_2$). We consider~\eqref{e.6.10}, again, as an
inhomogeneous Euler equation. The basic solutions are
$\mbox{const}\cdot w$ and $\mbox{const}\cdot w^{k_2}$. If
$r=k_2=2$ is the resonant integer, we apply the variation of
constants and conclude, in the same way as for the resonant value
$r=k_1=1$, that $K_2\neq 0$ contradicts the assumptions of
Theorem~3. We may then put $h_2:=0$ in~\eqref{matrixequation} and
the rest of the proof repeats that of the proposition in the
nonresonant case, as no more resonant integers can exist. If,
otherwise, $k_2>2$ and hence $r=2$ is not a resonant integer, one
can check that the variation of constants
gives a partial solution of the form $\begin{pmatrix} z_2  \\
u_2\end{pmatrix}=h_2w^2$, where $h_2$ is a constant vector. It is
easy to see that the fact that $h_2w^2$ is a solution
of~\eqref{e.6.10} implies that $h_2$ is a (unique!) solution
of~\eqref{matrixequation}. It is then straightforward to check
that, proceeding further with the small parameter method and
gathering terms with $\varepsilon^3$, one has, in the same spirit
as before,
\begin{equation}
\begin{pmatrix} z_3'  \\ u_3'\end{pmatrix}=\frac{1}{w}L\begin{pmatrix} z_3  \\
u_3\end{pmatrix}+\begin{pmatrix} 0  \\ K_3 \end{pmatrix}\cdot w^2,
\end{equation}
where $L,K_3$ are as in~\eqref{matrixequation} (more precisely,
one has to substitute the values $a_1,b_1,a_2,b_2$, found on the
previous steps, into $K_3$). The latter follows from the fact that
the second term $h_2w^2\varepsilon^2$ in the small parameter
expansion agrees with the solution $h_2$ of \eqref{matrixequation}
for $r=2$. In the same way as before, we conclude now that if
$k_2=3$, then $K_3=0$ and we set $h_3=0$ in
\eqref{matrixequation}, in order to avoid a contradiction with the
assumptions of Theorem~3.5. We then repeat the proof as in the
nonresonant case. Otherwise, we again obtain a partial solution
$\begin{pmatrix} z_3  \\ u_3\end{pmatrix}=h_3w^3$, where $h_3$ is
a constant vector, satisfying~\eqref{matrixequation} for $r=3$.

We continue with the similar arguments until we reach the step
$r=k_2$, to get $K_{k_2}=0$, $h_{k_2}=0$ in~\eqref{matrixequation}
and then repeat the proof as in the nonresonant case. This
completes the case $k_1=1$. The proof in the case $k_1>1$ uses the
same arguments as above and is completely analogous.
\end{proof}

Thus Theorem 3.5 is finally proved. Theorem 3.5 and Proposition
6.10 now imply Theorem~3.6.

\section{Analytic continuation and infinitesimal automorphisms}

It was explained in Section 2 that the monodromy of a mapping,
associated with a nonminimal pseudospherical hypersurface, is
given by some $\sigma \in {\rm Aut}(\CP{n})$. This allows us to
obtain in this section a useful representation of the
infinitesimal automorphism algebra $\mathfrak{hol}(M,p)$ for $p\in
X$ of a nonminimal pseudospherical hypersurface. Combining this
representation with Theorem~3.4, we will prove in the next section
the  Dimension Conjecture.

\begin{proof}[Proof of Theorem~3.7]
Fix a collection $\{p,U,\mathcal F_0,\mathcal F,\mathcal{Q}\}$,
where $p\in M$ is a Levi-nondegenerate point, $\mathcal
Q\subset\CP{n}$ a nondegenerate hyperquadric, $\mathcal
F_0:\,(\CC{n},p)\longrightarrow (\CP{n},p')$ a biholomorphic
mapping with $\mathcal F_0(M)\subset\mathcal Q$, and $U$ is an
open neighbourhood of the origin such that $\mathcal F_0$ extends
in $U\setminus X$ to a (multiple-valued) locally biholomorphic
mapping $\mathcal F$ into $\CP{n}$ in the sense of Weierstrass. We
denote by $M^+,M^-$ the two sides of $M\setminus X$ and assume,
without loss of generality, that $p\in M^+$. Fix an element $L\in
\mathfrak{hol}\,(M,0)$ and consider the (connected) flow
$\psi_t:\,(\CC{n},0)\longrightarrow
(\CC{n},0),\,t\in\RR{},\,\psi_0=\mbox{Id}$, generated by $\re L$.
Note that any local automorphism $\psi_t$ must preserve the
complex hypersurface $X$, and so we may assume that
$\psi_t(M^+)\subset M^+$.
For $p$ sufficiently close to $0$, we may suppose that $\psi_t$
with sufficiently small $t$ are defined in a neighbourhood of $p$
and consider their push-forwards
$$
\tau_t:=\mathcal F_0\circ\psi\circ \mathcal F_0^{-1}.
$$
Then $\tau_t$ is a flow of local CR-automorphisms of $\mathcal Q$
at $p'=\mathcal F_0(p)$ and, according to~\cite{chern},
$\tau_t\in\mbox{Aut}\,(\mathcal Q)$. It is also shown in
\cite{chern} that $\mbox{Aut}\,(\mathcal Q)$ is a maximally
totally real subgroup of $\mbox{Aut}(\CP{n})$. Note that the
correspondence $\psi_t \to \tau_t$ is injective w.r.t. the flows.
Now let us consider the analytic mappings $\mathcal F^t:=\mathcal
F\circ\psi_t$ in $U_t\setminus X$ for a sufficiently small
polydisc $U_t\subset U$, centred at $0$. It is easy to see from
the definition of $\mathcal F^t$ that its germ at $p$ also maps
$(M,p)$ into $\mathcal Q$, and if $\sigma$ is the monodromy matrix
associated with $\mathcal F$ then $\mathcal F^t$ has the same
monodromy matrix $\sigma$. On the other hand, \eqref{changesigma}
shows that the monodromy of $\mathcal F^t$ is given by the matrix
$\tau_t\circ\sigma\circ\tau_t^{-1}$ with $\tau_t$ being exactly
the push-forward of $\psi_t$. Hence,
$$
\sigma=\tau_t\circ\sigma\circ\tau_t^{-1} .
$$
Therefore, the push-forward of the automorphisms $\tau_t$ belong
to the subgroup $C\subset\mbox{Aut}(\mathcal Q)$ that consists of
elements of $\mbox{Aut}(\mathcal Q)\subset \mbox{Aut}(\CP{n})$,
commuting with the element $\sigma\in\mbox{Aut}(\CP{n})$. The
subgroup $C$ is the intersection of the centralizer $Z(\sigma)$
(see \cite{vinberg}) of the element
$\sigma\in\mbox{Aut}\,(\CP{n})$ with the totally real subgroup
$\mbox{Aut}(\mathcal Q)\subset \mbox{Aut}(\CP{n})$.  Its tangent
algebra is $c=z(\sigma)\cap \mathfrak{hol}\,(\mathcal Q,p')$,
where $z(\sigma)$ is the tangent algebra to $Z(\sigma)$ (we also
call it the \it centralizer of $\sigma$). \rm The above arguments
imply the existence of an injective embedding of
$\mathfrak{hol}\,(M,p)$ into the algebra $c$.
\end{proof}

As an application we obtain
\begin{corol}
Let $M\subset\CC{2}$ be a smooth real-analytic hypersurface,
passing through the origin, and
$\mbox{dim}\,\mathfrak{hol}\,(M,0)\geq 5$. Then either (i) $M$ is
Levi-flat, or (ii) $(M,0)$ is spherical, or (iii) $M$ is
holomorphically equivalent to a hypersurface of class $\mathcal
P_0$ such that its monodromy operator $\sigma$ is the identity (in
other words, the associated mapping $\mathcal F$ is
single-valued).\end{corol}

\begin{proof} We consider several cases depending on the Levi form of $M$.

If $M$ is Levi-flat, then
$\mbox{dim}\,\mathfrak{hol}\,(M,0)=\infty$, see \cite{ber}.

If $M$ is Levi nondegenerate at $0$, then the classical results in
\cite{poincare}, \cite{chern} imply that
$\mbox{dim}\,\mathfrak{hol}\,(M,0)\leq 8.$ Further analysis in
\cite{belold} shows $\mbox{dim}\,\mathfrak{aut}\,(M,0)\leq 1$,
unless $(M,0)$ is spherical. Combining this with the
classification of E.~Cartan\,\cite{cartan} of homogeneous
hypersurfaces in $\CC{2}$, we obtain
$\mbox{dim}\,\mathfrak{hol}\,(M,0)\leq 3$, if $M$ is
Levi-nondegenerate and is not spherical at zero.

If $M$ is Levi-degenerate at $0$, but not Levi-flat, the
hypersurface $M$ can either be of finite type at $0$ (see
\cite{ber} for various definitions of type), which is equivalent
to its minimality, or $M$ can be of infinite type, which is
equivalent to its nonminimality. Some generalizations of
Poincar\'e-Chern-Moser arguments provide the estimate
$\mbox{dim}\,\mathfrak{hol}\,(M,0)\leq 4$ in the finite type case
(e.g., \cite{kolar}). Thus we may assume $M$ is nonminimal at $0$.
Let $\Sigma\subset M$ be the set of points where the Levi form is
degenerate. If $\Sigma \ne X$ near the origin, then, since $X$ is
the only complex hypersurface contained in $M$ in a sufficiently
small neighbourhood of the origin, there exist finite type Levi
degenerate points in $M$, arbitrarily close to $0$. Applying the
bounds from \cite{kolar}, we obtain again
$\mbox{dim}\,\mathfrak{hol}\,(M,0)\leq 4$. Thus, we may assume
that $M\setminus X$ is Levi-nondegenerate in a sufficiently small
neighbourhood of the origin. The inequality
$\mbox{dim}\,\mathfrak{hol}\,(M,0)\geq 5$ implies that for a
Levi-nondegenerate point $p\in M\setminus X$ its infinitesimal
automorphism algebra has dimension at least $5$. Applying again
\cite{belold} and \cite{cartan}, we conclude that $M\setminus X$
is spherical and therefore it is biholomorphically equivalent to
some $\tilde M\in\mathcal P_0$. Thus, it remains to consider only
the case when $M\in\mathcal P_0$. Theorem 3.7 gives
\begin{equation}\label{boundz}
\mbox{dim}\,\mathfrak{hol}\,(M,0)\leq
\mbox{dim}_{\CC{}}\,z(\sigma),
\end{equation}
where $\sigma$ is the monodromy operator for $M$ ($\sigma$ can be
interpreted as a $3\times 3$ matrix, defined up to scaling).
Centralizers of elements of $GL(3, \CC{})$ can be easily analyzed,
using the Jordan normal form, and it is not difficult to see that
for all nonscalar matrices the centralizer has dimension at most
$5$. Taking the scaling into account, we have
$\mbox{dim}_{\CC{}}\,z(\sigma)\leq 4$, unless $\sigma=\mbox{Id}$.
\end{proof}

The next result immediately follows from Corollary~8.1.

\begin{corol} Theorem 3.8 implies the Strong
Dimension Conjecture.\end{corol}\rm

The following proposition gives the answer in the case when
$\mathcal F$ is single-valued and extends to~$X$.

\begin{propos}
Let $M\subset\CC{2}$ be of class $\mathcal P_0$, and $U$ be the
associated neighbourhood. Assume, in addition, that the associated
mapping $\mathcal F$ extends to the complex locus $X$
holomorphically. Then $\mathfrak{hol}\,(M,0)$ can be injectively
embedded into the stability algebra $\mathfrak{aut}\,(S^3,o')$ for
some point $o'\in S^3$. In particular,
$\mbox{dim}\,\mathfrak{hol}\,(M,0)\leq 5$.
\end{propos}

\begin{proof}
First note that $\mathcal F(X)$ is a locally countable union of
locally complex analytic sets\,\cite{chirka}. On the other hand,
$\mathcal F(X)$ is connected and $\mathcal F(X)\subset S^3$, so
that we conclude $\mathcal F(X)=\{o'\}$ for some point $o'\in
S^3$.
Choose now a point $q\in M^+$ ($M^+,M^-$ are the sides of
$M\setminus X$) and a local flow $\psi_t$ of local automorphisms
of $M$ near the origin, $\psi_t(M^+)\subset M^+$. Shrinking $U$ if
necessary, we may suppose that $\psi_t$ is defined in $U$. Arguing
as in the proof of Theorem~3.7, we may consider the push-forward
$\tau_t:=\mathcal F\circ \psi_t \circ \mathcal F^{-1}$ defined in
a neighbourhood of the point $q'=\mathcal F(q)$ (we choose the
element of $\mathcal F^{-1}$ with $\mathcal F^{-1}(q')=q$). Since
$\psi_t(M)\subset M$, we have $\tau_t(S^3)\subset S^3$, so
$\tau_t$ extends to an element of $\mbox{Aut}\,(S^3)\subset
\mbox{Aut}\,(\CP{2})$ (see \cite{chern}). Then for points
$z\in\CC{2}$, close to $q$, we have $\mathcal
F\circ\psi_t(z)=\tau_t\circ \mathcal F(z)$. By uniqueness the
latter equality holds for all $z\in U$. Therefore, $\mathcal
F(\psi_t(0))=\tau_t(\mathcal F(0))$ and, since $0\in X$,
$\psi_t(X)\subset X$, $\mathcal F(X)=\{o'\},$ we conclude that
$\tau_t(o')=o'$, and so $\tau_t$ stabilize the point $o'$.
Applying this to a local flow $\psi_t$, generated by $\re L$ for
some $L\in\mathfrak{hol}\,(M,0)$, we conclude that the flow
$\tau_t:=\mathcal F\circ\psi_t\circ \mathcal F^{-1}$ extends to a
flow $\tau_t\in \mbox{Aut}\,(S^3)$ with $\tau_t(o')=o'$, and then
for the corresponding vector field $L'\in\mathfrak{hol}\,(S^3,q')$
we have $L'(o')=0$. As the correspondence $\psi_t\longrightarrow
\tau_t$ is injective w.r.t. a flow $\psi_t$, the proposition
follows.
\end{proof}

\begin{corol}
Theorem~3.8 holds true for any  hypersurface $M\in\mathcal P_0$,
except, possibly, the case of a hypersurface with a single-valued
associated mapping $\mathcal F$, which does not extend
holomorphically to the complex locus $X$. In particular, the
Strong Dimension Conjecture holds true for any 1-nonminimal at the
origin smooth real-analytic hypersurface
$M\subset\CC{2}$.\end{corol}

\section{Solution of the Dimension Conjecture}

In this section we complete the proof of the Dimension Conjecture.
In view of  Section~8, it remains to treat the case of an
$m$-nonminimal hypersurface $M\in\mathcal P_0$ with a
single-valued mapping $\mathcal F:\,U\setminus
X\longrightarrow\CP{2}$ associated with $M$, which does not extend
to $\{w=0\}$ holomorphically.

Consider the Lie algebra $\mathfrak g=\mathfrak{hol}(M,0)$ and its
complexification $\mathfrak h=\mathfrak g^{\CC{}}=\mathfrak
g\otimes\CC{}$.  Fix a Levi nondegenerate point $p\in M$, for
which all vector fields $L\in \mathfrak g$ are defined, and for a
vector field $L\in \mathfrak g$ consider, as in the proof of
Theorem 3.7, its push-forward $L^*\in \mathfrak{hol}(\CP{2})$.
Then we obtain a well-defined push-forward $(\mathfrak
g^*,\mathfrak h^*)$ for the pair $(\mathfrak g,\mathfrak h)$. Here
$\mathfrak g^*$ and $\mathfrak h^*$ are a real and a complex Lie
subalgebras of $\mathfrak{hol}(\CP{n})$ respectively, naturally
isomorphic to the algebras $\mathfrak g$ and $\mathfrak h$
respectively. It follows from our construction that the
pulled-back algebra $\mathcal F^{-1}\circ \mathfrak h^*$ coincides
with $\mathfrak h$, in particular, all vector fields from the
well-defined in $U\setminus X$ algebra $\mathcal F^{-1}\circ
\mathfrak h^*$ extend to $X$ holomorphically. We also note that a
projective change of coordinates in $\CP{2}$, given by
$\tau\in\mbox{PGL}(3,\CC{})$, replaces the mapping $\mathcal F$
with the mapping $\tau\circ\mathcal F$. At the same time, $\tau$
conjugates the Lie algebra
$\mathfrak{hol}(\CP{n})\simeq\mathfrak{sl}(3, \CC{})$, and
$\mathfrak h^*$ changes accordingly (see Section 2).

We now need the following statement.

\begin{propos}\label{p.9.1}
Fix an affine chart $V\subset\CP{2}$  with the affine coordinates
$(z^*,w^*)$. Then the algebra $\mathfrak h^*$ cannot contain the
2-dimensional subalgebra, given in $V$ by
\begin{equation}\label{2alg}
\mbox{span}_{\CC{}}\left\{\frac{\partial}{\partial
z^*},\frac{\partial}{\partial w^*}\right\}.
\end{equation}
\end{propos}

\begin{proof}
Assume on the contrary that
$\mbox{span}_{\CC{}}\left\{\frac{\partial}{\partial
z^*},\frac{\partial}{\partial w^*}\right\}\subset\mathfrak h^*$.
Take the regular set $U^0\subset U\setminus X$ (see Section 5) and
consider $\mathcal F$, restricted to $U^0$, as a mapping into $V$.
Consider first the case when $\alpha_0(w)\not\equiv 0$ in
\eqref{fralin}. We represent $\mathcal F$ as in \eqref{fgaalpha}
with single-valued $\alpha(w),\beta(w),a(w),b(w),\delta(w)$. Then,
applying \eqref{fgaalpha}, we have
\begin{gather}\label{eddz}
\mathcal F^{-1}\circ\frac{\partial}{\partial
z^*}=T_1(z,w)\frac{\partial}{\partial
z}+\frac{a}{\alpha'a-a'\alpha}(z+\delta)\frac{\partial}{\partial
w},\\
 \label{eddw}\mathcal F^{-1}\circ\frac{\partial}{\partial
w^*}=T_2(z,w)\frac{\partial}{\partial
z}-\frac{\alpha}{\alpha'a-a'\alpha}(z+\delta)\frac{\partial}{\partial
w}.
\end{gather}
Here $T_1(z,w),T_2(z,w)$ are some specific functions, but their
exact form is of no importance to us. Since the vector fields
in~\eqref{eddz} and~\eqref{eddw} extend holomorphically to $X$,
the functions $P(z,w)=\frac{a}{\alpha'a-a'\alpha}(z+\delta)$ and
$Q(z,w)=\frac{\alpha}{\alpha'a-a'\alpha}(z+\delta)$ are
holomorphic near the origin. From this it follows that
$\delta(w)\in\mathcal M(0)$. Further, letting
$a(w)=k(w)\alpha(w)$, we conclude that
$k(w)=\frac{P}{Q}\in\mathcal M(0)$. Since
$Q(z,w)=-\frac{1}{k'\alpha}(z+\delta)$, it follows that
$k'\alpha\in\mathcal M(0)$, so that $\alpha(w), a(w)\in\mathcal
M(0)$. Note that $k(w)$ is not a constant, as this would
contradict~\eqref{jacobian}. Thus, by Theorem~3.4, $\mathcal F$
extends to $X$ holomorphically, which is a contradiction.

Now consider the case when $\alpha_0(w)\equiv 0$ in
\eqref{fralin}. It follows that $\mathcal F=(f,g)$ satisfies
\begin{equation}
\label{lincase} f=\alpha z+\beta,\,g=az+b
\end{equation}
for some single-valued meromorphic in $\Delta^*_\epsilon$
functions $\alpha(w),\beta(w),a(w),b(w)$. Then either
$\alpha\not\equiv 0$ or $a\not\equiv 0$ (as $\mathcal F$ is
locally injective). Say, $\alpha\not\equiv 0$, so we set
$k(w):=\frac{a(w)}{\alpha(w)}$. Then the fact that $I_1(z,w)=0$ in
\eqref{e.4.7} (see Proposition 5.2) yields the special relation
$\alpha'a-a'\alpha=0$, which implies that $k$ is a constant. We
now apply \eqref{lincase} to conclude that the Jacobian of the
mapping $\mathcal F$ is equal to $\alpha(b'-k\beta')$, and that
\begin{eqnarray}\label{lindz} \mathcal
F^{-1}\circ\frac{\partial}{\partial
z^*}=\left(k\frac{\alpha'}{\alpha}\frac{1}{b'-k\beta'}z+\frac{b'}{\alpha}\frac{1}{b'-k\beta'}\right)\frac{\partial}{\partial
z}-k\frac{1}{b'-k\beta'}\frac{\partial}{\partial
w},\\
\label{lindw} \mathcal F^{-1}\circ\frac{\partial}{\partial
w^*}=-\left(\frac{\alpha'}{\alpha}\frac{1}{b'-k\beta'}z+\frac{\beta'}{\alpha}\frac{1}{b'-k\beta'}\right)\frac{\partial}{\partial
z}+\frac{1}{b'-k\beta'}\frac{\partial}{\partial w}.
\end{eqnarray}
As both \eqref{lindz},\eqref{lindw} extend to $X$ holomorphically,
we conclude first that $b'-k\beta'\in\mathcal M(0)$ and second,
considering the linear combination
$F^{-1}\circ\frac{\partial}{\partial z^*}+k\mathcal
F^{-1}\circ\frac{\partial}{\partial w^*}=\frac{1}{\alpha}\dz$,
that $\alpha\in\mathcal M(0)$. These two conclusions imply
$\beta',b'\in\mathcal M(0)$ and finally $\beta,b,a\in\mathcal
M(0)$. Then by Theorem 3.4 $\mathcal F$ extends to $X$
holomorphically, which is again a contradiction. This proves the
proposition.
\end{proof}

Our next goal is the classification of higher-dimensional Lie
subalgebras of $\mathfrak{sl}(3,\CC{})$. We could not find an
appropriate reference in the literature, so for the sake of
completeness we provide the proof that was suggested to us by
Andrey Minchenko. By a {\it matrix element} $e_{ij}$ we mean a
square matrix all of whose entries are zero, except the entry in
the $i$-th row and the $j$-th column which equals 1.

\begin{propos}\label{p.9.2}
Let $\mathfrak l\subset\mathfrak{sl}(3,\CC{})$ be a complex Lie
subalgebra, $\mbox{dim}\,\mathfrak l\geq 5$. Denote by $\mathfrak
b_\pm$ the subalgebras of upper-triangular and lower-triangular
elements of $\mathfrak{sl}(3,\CC{})$ respectively, and by
$\mathfrak r_{\pm}$ the subalgebras of zero last row and zero last
column elements of $\mathfrak{sl}(3,\CC{})$ respectively. Let
$\mathfrak p_{+}=\mathfrak b_{+}\oplus \CC{}e_{21}$, and
$\mathfrak p_{-}= \mathfrak b_{-}\oplus \CC{}e_{23}$. Then
$\mathfrak l$ is conjugated in $\mathfrak{sl}(3,\CC{})$ to one of
the subalgebras $\mathfrak b_{+}$, $\mathfrak r_{\pm}$, $\mathfrak
p_{\pm}$, or $\mathfrak{sl}(3,\CC{})$.\end{propos}

\begin{proof}
In what follows we refer to \cite{vinberg} for various facts from
the Lie theory. First, consider the case when $\mathfrak l$ is
solvable. Then, as $\mbox{dim}\,\mathfrak l\geq 5$, we conclude
that $\mathfrak l$ is the Borel subalgebra. As the Borel
subalgebra is unique, up to a conjugation, we conclude that
$\mathfrak l$ is conjugated to $\mathfrak b_+$. If, otherwise,
$\mathfrak l$ is not solvable, then its Levi-Malcev decomposition
contains a nontrivial semi-simple factor. From the structure
theory of semi-simple Lie algebras, any such factor contains a
subalgebra, isomorphic to $\mathfrak{sl}(2,\CC{})$. It is known
that there exist, up to a conjugation, exactly two subalgebras in
$\mathfrak{sl}(3,\CC{})$, isomorphic to $\mathfrak{sl}(2,\CC{})$:
the first one is
$\mathfrak{so}(3,\CC{})\subset\mathfrak{sl}(3,\CC{})$, and the
second one is $\mathfrak{sl}(2,\CC{})\subset
\mathfrak{sl}(3,\CC{})$, embedded as the left upper $2\times 2$
block, so that we may suppose that, after an appropriate
conjugation, $\mathfrak l$ contains one of the above subalgebras.
Consider first the case of $\mathfrak{so}(3,\CC{})\subset
\mathfrak l\subset\mathfrak{sl}(3,\CC{})$. Then the subalgebra
$\mathfrak{so}(3,\CC{})$ acts on $\mathfrak{sl}(3,\CC{})$ by the
adjoint representation of $\mathfrak{sl}(3,\CC{})$, restricted
onto $\mathfrak{so}(3,\CC{})$. Decomposing
$\mathfrak{sl}(3,\CC{})$ into a direct sum of irreducible
invariant subspaces for the above action, we get the decomposition
$$\mathfrak{sl}(3,\CC{})=\mathfrak{so}(3,\CC{})\oplus  V_,$$ where $V$ is the subspace of
all symmetric matrices from $\mathfrak{sl}(3,\CC{})$. The
subalgebra $\mathfrak l$ must be the sum of
$\mathfrak{so}(3,\CC{})$ and some of the invariant subspaces, so
$\mathfrak l=\mathfrak{sl}(3,\CC{})$ or $\mathfrak
l=\mathfrak{so}(3,\CC{})$. As $\mbox{dim}\,\mathfrak l\geq 5$, we
summarize the $\mathfrak{so}(3,\CC{})$-case with the conclusion
$\mathfrak l=\mathfrak{sl}(3,\CC{}).$

Consider now the case $\mathfrak{sl}(2,\CC{})\subset \mathfrak
l\subset\mathfrak{sl}(3,\CC{})$. Arguing as in the
$\mathfrak{so}(3,\CC{})$-case, we obtain the decomposition
$$\mathfrak{sl}(3,\CC{})=\mathfrak{sl}(2,\CC{})\oplus(\CC{}e_{13}\oplus\CC{}e_{23})\oplus(\CC{}e_{31}\oplus\CC{}e_{32})\oplus \CC{}h$$ of
$\mathfrak{sl}(3,\CC{})$ into the direct sum of irreducible
invariant subspaces of $\mathfrak{sl}(3,\CC{})$ under the adjoint
action of $\mathfrak{sl}(3,\CC{})$, restricted onto
$\mathfrak{sl}(2,\CC{})$. Here $h=\mbox{diag}\{1,1,-2\}$. The
algebra $\mathfrak l$ is the direct sum of
$\mathfrak{sl}(2,\CC{})$ and some of the invariant subspaces.
Then, in view of the assumption $\mbox{dim}\,\mathfrak l\geq 5$,
we obtain the following list of distinct decompositions of
$\mathfrak l$:
\begin{gather*}\mathfrak
l=\mathfrak{sl}(2,\CC{})\oplus(\CC{}e_{13}\oplus\CC{}e_{23});\,\mathfrak
l=\mathfrak{sl}(2,\CC{})\oplus(\CC{}e_{31}\oplus\CC{}e_{32});\,
\mathfrak
l=\mathfrak{sl}(2,\CC{})\oplus(\CC{}e_{13}\oplus\CC{}e_{23})\oplus\CC{}h;\\
\mathfrak
l=\mathfrak{sl}(2,\CC{})\oplus(\CC{}e_{31}\oplus\CC{}e_{32})\oplus\CC{}h;\,\mathfrak
l=\mathfrak{sl}(2,\CC{})\oplus(\CC{}e_{13}\oplus\CC{}e_{23})\oplus(\CC{}e_{31}\oplus\CC{}e_{32})\oplus
\CC{}h.\end{gather*} This implies the claim of the proposition.
\end{proof}

The classification implies

\begin{propos}\label{p.9.3}
Let $\mathfrak l$ be a subalgebra in $\mathfrak{hol}(\CP{2})$ with
$\mbox{dim}\, \mathfrak l\geq 5$. Then there exists an affine
chart $V$ with coordinates $(z^*,w^*)$ such that $\mathfrak l$
contains the 2-dimensional subalgebra $\mathfrak a$, given in $V$
by \eqref{2alg}.
\end{propos}

\begin{proof} Interpreting the commutative Lie algebra $\mathfrak a$ of holomorphic vector fields as a subalgebra in
$\mathfrak{sl}(3,\CC{})$, we obtain the representation of
$\mathfrak a$ as $\mbox{span}_{\CC{}}\{e_{13},e_{23}\}$ (we use
the notation of Proposition~\ref{p.9.2} in what follows). In order
to use the classification, given by Proposition~\ref{p.9.2}, we
assign to each conjugacy in the Lie algebra
$\mathfrak{sl}(3,\CC{})$ a projective coordinate change in
$\CP{2}$, and get the corresponding affine chart $V\subset\CP{2}$
with the coordinates $(z^*,w^*)$ (see Section 2.5). Note that the
subalgebras $\mathfrak b_+,\mathfrak r_+,\mathfrak p_+\subset
\mathfrak{sl}(3,\CC{})$ already contain $\mathfrak a$. Further, it
is straightforward to check that the matrix
$A=e_{31}+e_{12}+e_{23}\in\mbox{SL}(3,\CC{})$ conjugates the
matrices $e_{21},e_{31}\in\mathfrak r_- \cap \mathfrak p_-$ with
the matrices $e_{13},e_{23}$ respectively. The latter implies that
{\it any} subalgebra $\mathfrak l\subset \mathfrak{sl}(3,\CC{})$
with $\mbox{dim}\, \mathfrak l\geq 5$ contains, after an
appropriate conjugation, the algebra $\mathfrak a$. This proves
the proposition.
\end{proof}

Combined, Propositions~\ref{p.9.1}, Proposition 9.3 and Corollary
8.1 yield

\begin{corol}\label{c.9.4}
 Let $M\in\mathcal P_0$, and the associated mapping $\mathcal F$ does not extend to $X$ holomorphically. Then
$\mbox{dim}\,\mathfrak{hol}(M,0)\leq 4$.
\end{corol}

Corollary~\ref{c.9.4} immediately implies the proof of Theorem
3.8. Combining with Corollary 8.1, we obtain also Theorem 3.9.
Theorem 3.10 follows from a combination of Corollary 9.4, Theorem
3.9 and Corollary 8.1. Finally, Theorem 3.11 follows from the fact
the any Lie algebra of dimension $\leq 3$ is contained in
$\mathfrak{su}(2,1)$ (see, e.g., \cite{vinberg}), and in the case
$4\leq\mbox{dim}\,\mathfrak{hol}(M,0)<\infty$ $M$ needs to be
spherical at its generic point and the embedding into
$\mathfrak{hol}(S^3,o)$ is immediate. The bound
$\mbox{dim}\,\mathfrak{hol}(M,0)\leq 5$ in the nonspherical case
follows from Theorem 3.10.

\medskip

 All the results of the paper are completely proved now.


\end{document}